\documentclass[11pt, oneside]{article}   	% use "amsart" instead of "article" for AMSLaTeX format
\usepackage{geometry}                		% See geometry.pdf to learn the layout options. There are lots.
\usepackage{graphicx}				% Use pdf, png, jpg, or eps§ with pdflatex; use eps in DVI mode
\usepackage{amssymb,amsmath}
\usepackage{xcolor}
\usepackage{dsfont}

\usepackage{tikz-cd}

\usepackage{ wasysym }

\usepackage[linecolor=white,backgroundcolor=white,bordercolor=white]{todonotes}
\usepackage[T1]{fontenc}
\usepackage[applemac]{inputenc}	
\usepackage{amsthm}
\usepackage{mathrsfs} 
\usepackage{caption}
\usepackage{subcaption}
\usepackage[colorlinks=true,linkcolor=cyan,citecolor=magenta,bookmarks=true]{hyperref}
\usepackage{enumerate}
\definecolor{LGray}{rgb}{0.8,0.8,0.8}
\newtheorem{thm}{Theorem}[section]
\newtheorem{prop}[thm]{Proposition}
\newtheorem{coro}[thm]{Corollary}
\newtheorem{lem}[thm]{Lemma}
\newtheorem{conj}[thm]{Conjecture}

\newtheorem{ex}[thm]{Example}

       \theoremstyle{definition}
       \newtheorem{dfn}{Definition}[section]
        
       \theoremstyle{remark}
       \newtheorem*{rmk}{Remark}

\newcommand{\E}{\mathbb{E}}
\newcommand{\R}{\mathbb{R}}
\newcommand{\C}{\mathbb{C}}

\newcommand{\N}{\mathbb{N}}
\newcommand{\Z}{\mathbb{Z}}
\newcommand{\U}{\mathrm{U}}
\newcommand{\SU}{\mathrm{SU}}
\newcommand{\SO}{\mathrm{SO}}
\newcommand{\Sp}{\mathrm{Sp}}
\newcommand{\GL}{\mathrm{GL}}

\newcommand{\Tb}{\mathbb{T}}

\newcommand{\Sbb}{\mathbb{S}}

\newcommand{\Gbb}{\mathbb{G}}

\newcommand{\Fbb}{\mathbb{F}}

\newcommand{\Dbb}{\mathbb{D}}
\newcommand{\Tbb}{\mathbb{T}}

\newcommand{\Ld}{\mathrm{L}}

\newcommand{\YM}{\mathrm{YM}}
\newcommand{\Tr}{\mathrm{Tr}}
\newcommand{\tr}{\mathrm{tr}}

\newcommand{\build}[3]{\mathrel{\mathop{\kern 0pt#1}\limits_{#2}^{#3}}}

\def\br{\begin{color}{blue}}
\def\bb{\begin{color}{blue}}
\def\bg{\begin{color}{blue}}
\def\er{\end{color}}
\def\eg{\end{color}}
\def\eb{\end{color}}
\def\a{\alpha}

\def\b{\beta}

\def\g{\gamma}
\def\G{\Gamma}
\def\l{\lambda}

\def\P{\mathbb{P}}

\def\Id{\operatorname{Id}}

\def\mfk{\mathfrak}

\def\mfB{\mathfrak{B}}
\def\mfC{\mathfrak{C}}
\def\mfP{\mathfrak{P}}

\def\mfh{\mathfrak{h}}

\def\mfm{\mathfrak{m}}

\def\cA{\mathcal{A}}
\def\cB{\mathcal{B}}

\def\cH{\mathcal{H}}

\def\cC{\mathcal{C}}

\def\vol{\mathrm{vol}}

\def\scV{\mathscr{V}}

\def\<{\langle}
\def\>{\rangle}

\def\Hom{\mathrm{Hom}}

\def\Scr{\mathscr{S}}

\def\pl{\partial}
\def\lto{\longrightarrow}

\def\bth{\begin{thm}}
\def\eth{\end{thm}}

\def\ov{\overline}
\def\und{\underline}

\RequirePackage{ifthen}
\RequirePackage{xkeyval}
\RequirePackage{xcolor}
\RequirePackage{tikz}
\usetikzlibrary{positioning} 
\RequirePackage{calc}

\author{Antoine Dahlqvist\thanks{University of Sussex, School of Mathematical and Physical Sciences, Pevensey 3 Building, Brighton, UK} \and Thibaut Lemoine\thanks{Universit\'e de Lille, CNRS, UMR 9189 - CRIStAL, 59651 Villeneuve d'Ascq, France}}

\title{Large N limit of the Yang--Mills measure\\ on compact surfaces II: \\Makeenko--Migdal equations and planar master field}
\begin{document}
\bibliographystyle{plain}
\maketitle
\abstract{This paper  considers the large $N$ limit of  Wilson loops for the two-dimensional Euclidean Yang--Mills measure on all  orientable compact surfaces of genus larger or equal to $1$, with a   structure group given by a classical compact matrix Lie group. Our main theorem shows the convergence of  all Wilson loops in probability, given that it holds true on a restricted class of loops, obtained as a modification of geodesic paths. Combined with the result of \cite{DL}, a corollary is the convergence of all Wilson loops on the torus. Unlike the sphere case, we show that the limiting object is remarkably expressed thanks to the master field on the plane defined in \cite{AS,Lev3} and we conjecture that this phenomenon is also valid for all surfaces of higher genus. We prove that this conjecture holds true whenever it does for the restricted class of loops of the main theorem. Our result on the torus justifies the introduction of an interpolation between free and classical convolution of probability measures, defined with the free unitary Brownian motion but differing from $t$-freeness of \cite{BenaychLev} that was defined in terms the liberation process of Voiculescu \cite{Voi}. In contrast to \cite{DL}, our main tool is a fine use of Makeenko--Migdal equations, proving  uniqueness of their solution under suitable assumptions, generalising the arguments of \cite{DN,Hal2}.}

\section{Introduction}

The two-dimensional Yang--Mills measure  is a probability model originating from Euclidean quantum field theory in the setting of pure gauge theory. It describes a generalised random connection on a principle bundle over a two dimensional manifold, with a compact Lie group as structure group, making rigorous the path integral over connections for the so-called Yang--Mills action.  Different equivalent  mathematical definitions  have been given in two dimensions and are due to \cite{Gro,Dri,Sen0,GKS,AlbeverioI,AlbeverioII,Lev2,Chev}. The work of \cite{Wit} brought to light many special features of the Yang--Mills measure in two dimensions, including its partial integrability, used as a way to perform exact volume computations for the Atiyah--Bott--Goldman measure \cite{AB,Goldman} on the space of flat connections  \cite{Liu,BismutLab,Sen3}.

\vspace{0.5em}

When a compact Lie group $G$ and a surface $\Sigma$ are given, the Yang--Mills measure can be mathematically understood as a random matrix model   which assigns to any loop\footnote{with enough regularity.} of the surface a random matrix so that concatenation and reversion of loops are compatible with the group operations. In \cite{Lev3}, it is shown that it gives rise to a random homomorphism from the group of rectifiable reduced loops of the surface to the chosen group $G.$

\vspace{0.5em}

We consider here a compact, connected, orientable surface $\Sigma$ of genus $g\ge 1$ and a group $G$ belonging to a series of classical compact matrix groups. We are primarily interested in the traces of these matrices, called \emph{Wilson loops,}  when the rank of $G$ goes to infinity. We ask whether Wilson loops converge in probability under the Yang--Mills measure, towards a deterministic function.

\vspace{0.5em}

Let us try to give a brief historical account of this problem. In physics, a motivation for the focus on Wilson loops is due to K. Wilson's work \cite{Wil} related to quarks confinement. The idea of studying the large rank regime in gauge theories, known as large $N$ limit, was first initiated by t'Hooft \cite{Hoo} on QCD. This lead to many articles in theoretical physics in the 80's studying the question in two dimensions, a partial list being \cite{KazPlan,KK,MM,Mig,Wit2,GrossMat,GG,GrossTaylor}.  In mathematics, this problem was advertised by I. Singer in  \cite{Sin} where the candidate limit of Wilson loops was called \emph{master field}, following the physics literature. The case of the plane and the sphere have been respectively proved in \cite{Xu,AS,Lev} and\footnote{See also \cite{Hal2} where a conditional result was obtained implying the case of the sphere, given the convergence for simple loops.} \cite{DN}. The case of general compact surfaces has been first investigated by \cite{Hal2} where loops contained in topological disc can be considered whenever the convergence holds for simple loops. The study of similar questions in the plane for analogs of the Yang--Mills measure has been treated in \cite{CDG}. In higher dimension, an analog\footnote{though in this case, there is  at the time of writing, no construction of the continuous Yang--Mills measure in dimension 3 and higher is available.  } of this question for a lattice model has also been considered \cite{Chatt}. Very recently and independently from the current work, it was shown in  \cite{MageeII,MageeI} that under the Atiyah--Bott--Goldman measure, which can be understood as the weak limit of the Yang--Mills measure when the area of the surface vanishes, the expectation of Wilson loops converges and has a $\tfrac{1}{N}$ expansion when the group belongs to the series of special unitary 
matrices and the surface is closed, orientable and of genus $g\ge 2.$ For further details and  references on the motivations  of this problem, we refer to \cite[Sec. 1]{DL} and \cite[Sec. 2.5.]{Lev7}.

\vspace{0.5em}

In this article,  we give a complete answer in the case of the torus and a conjecture and a partial result for all surfaces with  genus $g\ge 2$. It is the sequel of \cite{DL} where we have shown the convergence for a large\footnote{informally described as all simple loops or iteration of simple loops, and all loops which do not visit one handle of the surface.} but incomplete class of loops. Let us recall that in the case of the plane, the master field can  be described thanks to free probability and more specifically in terms of free unitary Brownian motion \cite{AS,Lev3}. The case of the sphere involves a different non-commutative stochastic process called the free unitary Brownian bridge \cite{DN}. In contrast, for the torus, we show that after lifting loops to the universal cover, the master field is also described by the planar master field and we conjecture that the same holds true for any surface of higher genus. In the torus case, the master field provides an interpolation parametrised by the total area of the torus, between the free and the classical convolution of two Haar unitaries built with the free unitary Brownian motion, which differs from the $t$-freeness introduced by \cite{BenaychLev} using the liberation process of \cite{Voi}.

\vspace{0.5em}

The aim of the current paper is to investigate the stability of Wilson loops convergence under homotopy equivalences. To do so, we will use a set of recursive equations named after Makeenko and Migdal \cite{MM}. When a loop is deformed in a specific way -- that we  
call a Makeenko--Migdal deformation -- these equations relate the differential of the expected Wilson loops  with the expectation of a product of Wilson loops having a smaller number of intersection points. These equations can be understood as a remarkable analog of Schwinger--Dyson equations used in random matrix theory and were first inferred heuristically in \cite{MM} as an integration by part for the path integral over the space of connections. A first rigorous proof was given in the case of the plane in\footnote{See also \cite[sect. 7]{Dah2} for a variation of this proof and \cite[section 0]{Lev3} for the heuristics of the original proof of \cite{MM} based on an integration by part in infinite dimension. See also \cite{DriverMM} for a proof closer in spirits to the original argument of \cite{MM}.} \cite{Lev3} and was later tremendously simplified and generalised in \cite{DHK,DGHK} in a local way that applies to any surface. Makeenko--Migdal equations were crucial to \cite{DN,Hal2} leading to an induction argument  on the number of intersection points that reduced  the convergence of all Wilson loops on the sphere to the case of simple loops. In the case of other surfaces, the very same strategy fails a priori, as some loops cannot be deformed to simpler loops without raising the number of intersection points, while some homotopy classes do not contain any loop for which the convergence is known to hold. We show here that the first hurdle can be overcome, allowing to reduce the problem, completely in the torus case and partially when $g\ge 2$, to the class of loops considered in \cite{DL}. We leave the completion of this program for all compact surfaces to a future work.

\vspace{0.5em}

The paper is organised as follows. The first four following sections of the introduction give respectively an informal definition of the Yang--Mills measure and of the main results, a discussion on the relation with the Atiyah--Bott--Goldman measure and the work \cite{MageeI,MageeII}, 
a consequence of the result on the torus in non-commutative probability, and lastly, a sketch of the strategy of the main 
proofs. Section 2 recalls and adapts some combinatorial notions of discrete homotopy and homology of loops in embedded graph instrumental to the proof. Section 3 gives the definition of the Yang--Mills measure, a statement of the Makeenko--Migdal equations 
and states the main results of the article. Section 4 consists in the proof of our main technical result, which is Proposition \ref{-->Prop: MM  Gen Non Zero Homol}. Section 5 describes the behaviour of Wilson loops when one performs surgery on the underlying 
surface. Section 6 is finally discussing how the master field on the torus  yields an interpolation between classical and free convolution, different from Voiculescu's liberation process. In an appendix, for the sake of completeness, we recall and 
prove several results on Makeenko--Migdal equations, that are quite standard in the literature for unitary groups but not necessarily for all classical groups.

\tableofcontents

\subsection{Yang--Mills measure and master field, statement of results}
\label{--------sec: informal Statement}

We shall first give a heuristic definition of the Yang--Mills measure in its geometric setting and state informally the main results of the current 
article. Proper definitions and statements are respectively given in sections \ref{------sec:YM area continuity}  and \ref{-----sec:WL MF}.

\vspace{0.5em}

Let $\Sigma$ be either a compact, connected, closed orientable surface of  genus $g\ge 1$ endowed with a Riemannian metric -- we shall call it a \emph{compact surface of genus} $g$ in the sequel --, or the Euclidean plane $\R^2$ with its standard inner product. Let $G_N$ be a classical compact matrix Lie group of size $N$, i.e. viewed compact subgroup of $\GL_N(\C)$. We assume that the Lie algebra $\mathfrak{g}_N$ of $G_N$ is endowed with an $\mathrm{Ad}$-invariant inner product $\langle\cdot,\cdot\rangle$, as in section \ref{----sec:HK}. Given a $G_N$-principal bundle $(P,\pi,\Sigma)$, a connection is a 1-form $\omega$ on $M$ valued in adjoint fibre bundle  $\mathrm{ad}(P)$, its curvature  is the $\mathrm{ad}(P)$-valued 2-form $\Omega=d\omega+\frac12[\omega\wedge\omega]$. The \emph{Yang--Mills action} of a connection $\omega$ on a $G_N$-principal bundle $(P,\pi,\Sigma)$ is defined by
\begin{equation}\label{eq:YM_action}
S_{\YM}(\omega)=\frac12 \int_\Sigma \langle\Omega\wedge\star\Omega\rangle,
\end{equation}
where $\star$ denotes the Hodge operator. An important feature of dimension $2$ is that whenever $\Psi$ is a diffeomorphism of $\Sigma$ preserving its 
volume form,    
\begin{equation}
S_\YM({\Psi_*\omega})=S_{\YM}(\omega).\label{eq:InvArea}
\end{equation}
The \emph{Euclidean Yang--Mills measure} is the formal Gibbs measure 
\begin{equation}\label{eq:YM_informal}
d\mu_{\YM}(\omega) ``=" \frac{1}{Z}e^{-S_{\YM}(\omega)}\mathcal{D}\omega,
\end{equation}
where $\mathcal{D}\omega$  plays the role of a formal Lebesgue measure on the space of connections over an arbitrary principal bundle\footnote{There 
is here an apparent additional issue with this vague definition. A slightly less dubious state space could be obtained by fixing a representant of each principal bundle equivalence class over $\Sigma$ and by considering instead the set of pairs of a principal bundle belonging to this family together with a 
connection on it. When $\Sigma$ is a contractible space or if $G_N$ is simply connected, there is only one equivalence class of $G_N$-principal bundles over $\Sigma$ and this issue disappears. We shall not discuss further the question of the type of the principal bundle under the Yang--Mills in this text. For more details and rigorous results we refer to \cite{Lev8}.} and $Z$ is a normalisation constant supposed to ensure the total mass to be $1$. We choose here not to include a parameter in front of the action, as it can be included in the volume form of $\Sigma.$

\vspace{0.5em}

The space $\mathcal{A}(P)$ being infinite-dimensional, the latter equation has no mathematical meaning. Though at first stance, as the Yang--Mills 
action of $\omega$  can be seen as the $L^2$-norm of the curvature $\Omega$, an analogy with Gaussian measures can be hoped.  However, when $G_N$ is 
not abelian, $\Omega$ depends  non-linearly on 
$\omega$ which prevents any direct construction of $\mu_\YM$ using a Gaussian measure. In two dimensions, this non-linearity can be compensated by the 
so-called gauge symmetry of $S_{\YM}$ which allows to bypass this problem. This enabled the constructions of \cite{Gro,Dri,Sen0} based on stochastic 
calculus. See also \cite{Chev} for a recent approach defining further a random, distribution valued, connection on trivial bundles over the two-dimensional torus.  We  follow here instead the approach of \cite{Lev2} which focuses on the 
holonomy of a connection, whose law can be directly defined 
using the heat kernel on $G_N.$ The  definition we are using is recalled in section \ref{------sec:YM area continuity}, it agrees with  the 
construction of \cite{Gro,Dri,Sen0} thanks to the so-called Driver--Sengupta formula.  An important feature of this measure is suggested by  \eqref{eq:InvArea}. For any two-dimensional Riemannian manifold  $\Sigma'$ 
diffeomorphic to $\Sigma,$ and for any diffeomorphism $\Psi: \Sigma\to \Sigma'$, there is an induced measure $\Psi_*(\YM_{\Sigma})$ on connections of $(P,\Psi\circ\pi,\Sigma').$ If $\Psi$ preserves the area, then 
$$\Psi_*(\YM_{\Sigma})=\YM_{\Sigma'}.$$
We shall call this property the \emph{area-invariance} of the Yang--Mills measure. Moreover,  for any relatively compact, contractible, open subset $U$ of $\Sigma$, the restriction to $U$ induces a measure  
$\mathcal{R}^U_*(\YM_{\Sigma})$ 
on connections of $(\pi^{-1}(U),\pi,U)$. When $\Sigma $ is the Euclidean plane $\R^2$ or the Poincar\'e disc $\Dbb_\mfh$, with its usual (hyperbolic) metric, it 
satisfies\footnote{Compact surfaces do 
not have this 
property but there is still absolute continuity in place of continuity. This was instrumental in \cite{DL}. }  $\mathcal{R}_*^U(\YM_{\Sigma})=\YM_{U},$ where $U$ is endowed with the metric of $\Sigma$.

\vspace{0.5em}

Let $\omega$ be a connection on a $G_N$-principal bundle $(P,\pi,\Sigma)$, and $U$ be an open subset of $\Sigma$ where  $\pi:\pi^{-1}(U)\to U$ can be\footnote{the tubular neighbourhood of a smooth loop or  of an embedded graph could be such an open set.} trivialised. When such a trivialisation has been fixed, its \emph{holonomy} is a function $\gamma\mapsto\mathrm{hol}(\omega,\gamma)$ mapping paths\footnote{In this section the space of paths is not specified and could be taken as the space of piecewise smooth paths with constant speed and transverse intersections. A loop is a path with starting point equal to its endpoint.} $\gamma:[0,1]\to U$ to elements of the group $G_N$ such that
\[
\mathrm{hol}(\omega,\gamma_1\gamma_2)=\mathrm{hol}(\omega,\gamma_2)\mathrm{hol}(\omega,\gamma_1)
\]
for any paths $\gamma_1$ and $\gamma_2$ such that the endpoint of $\gamma_1$ coincides with the starting point of $\gamma_2$, while for any path 
$\gamma$,
\[
\mathrm{hol}(\omega,\gamma^{-1})=\mathrm{hol}(\omega,\gamma)^{-1},
\]
where $\g_1\g_2$ and $\g^{-1}$ denote the concatenation and reversion of the paths.

\vspace{0.5em}

When $G_N$ is a group of matrices of size $N$ and $\ell$ is a loop of $U$, the  \emph{Wilson loop}   associated to $\ell$ is the function
$$W_\ell(\omega)=\tr(\mathrm{hol}(\omega,\ell)),$$
where $\tr=\frac 1 N \Tr,$ with $\Tr$ the usual trace of matrices. This  function can be shown to be independent of the choice of local trivialisation of $(P,\pi,\Sigma)$
and is therefore only a function of $\omega$ and $\ell.$

Our  primary source of interest is the study of the random variables $W_{\ell}:=W_{\ell}(\omega)$, for  loops of $\Sigma,$ when $\omega$ is sampled 
according to $\YM_\Sigma.$ We are interested in the large $N$ limit of $W_\ell,$ when the scalar product $\<\cdot,\cdot\>$ is  chosen as in section \ref{----sec:HK} and  the volume form of the surface is fixed.  The paper \cite{Sin} seems to be the first mathematical article addressing this question, and it motivates the following conjecture, also suggested by \cite{Lev,DGHK,Hal2}. 

\begin{conj}\label{conj_Sing}Let $G_N$ be a classical compact matrix Lie group of size $N$, endowed with the metric of section \ref{----sec:HK} and denote by 
$\Sigma$ a compact surface of genus $g\geq 0$, the Euclidean plane $\R^2$ or the Poincar\'e disc $\mathbb{D}_\mathfrak{h}$. For any loop $\ell$ of $\Sigma$, there is a constant $\Phi_\Sigma(\ell)$ such that under $\mathrm{YM}_\Sigma$
\begin{equation}
W_\ell\to\Phi_\Sigma(\ell) \ \text{in probability as} \ N\to\infty.\label{eq:CV in P ConjSin}
\end{equation}
The functional $\Phi_\Sigma$ is called the \emph{master field on} $\Sigma$.
\end{conj}

The case of plane was first proved in \cite{Xu,AS} for $G_N=\U(N)$. In \cite{Lev}, the above statement was proved simultaneously to \cite{AS} 
for all groups mentioned, and for a large family of loops given by loops of finite length. Moreover, motivated by the physics articles \cite{MM,Mig,KK}, L\'evy proved in \cite{Lev} recursion relations giving a way to compute explicitly $\Phi_{\R^2}$ for all loops with finitely many intersections.

\vspace{0.5em}

By area invariance and restriction property, the result on the hyperbolic plane can be deduced directly from these latter works as follows. According to a theorem of Moser \cite{Moser}, any relatively compact open disc  $U$ of $\Dbb_\mfh$ with hyperbolic volume $t$ can be mapped to the open Euclidean disc  $D_t$ of $\R^2$ centered at $0$ and of area $t$, by a diffeomorphism $\Psi:U\to D_t$ sending the restriction of the hyperbolic volume form on $U$ to the restriction of Euclidean volume form on $D_t.$  By area-invariance, $\mathcal{R}^{U}_*(\YM_{\Dbb_\mfh})=\YM_U=\Psi^{-1}_*(\YM_{D_t})$, so that the conjecture holds true for $\Dbb_\mfh$ with 
$$\Phi_{\Dbb_\mfh}(\ell)=\Phi_{\R^2}(\Psi\circ\ell)$$ 
for any loop $\ell$ with range included in $U.$ 

\vspace{0.5em}

For  $\Sigma=\Sbb^2$, the conjecture was proved in  \cite{DN} for  all loops of finite length and $G_N=\U(N),$ while \cite{Hal2} 
gave a conditional result on $\Sbb^2$ based on an argument similar to \cite{DN}, as well as a 
conditional result for other surfaces for loops included in a topological disc, given  convergence of for simple loops. In \cite{DL} we gave an alternative  
argument proving  a generalisation of the results of \cite{Hal2} on compact surfaces without using the conditions \cite{Hal2}, see section \ref{-----sec:Strategy}. The current article was written with the aim to strengthen the  argument common to \cite{DN} and \cite{Hal2} in order to address the conjecture on all compact manifolds. This led to the following theorem and conjecture. 

\begin{thm}\label{-->THM: Torus Intro} When $\Tb_T$ is a torus of  volume $T>0$,   conjecture \ref{conj_Sing} is valid. Moreover, considering $\Tb_T$ as the quotient of  the Euclidean plane $\R^2$  by $\sqrt{T}. \Z^2,$ 
$$\Phi_{\Tb_T}(\ell)=\left\{\begin{array}{ll} \Phi_{\R^2}(\tilde \ell) & \text{if }\ell\text{ is contractible,} \\ &\\  0 &\text{ otherwise,} \end{array}\right. $$
where for any  continuous loop $\ell$ in $\Tb_T,$ $\tilde\ell$ is a lift of $\ell$ to $\R^2,$ that is a smooth loop of $\R^2$, whose projection on $\R^2/\sqrt{T}.\Z^2$ is $\ell.$ 
\end{thm}

We discuss an interpretation of this result in terms of non-commutative probability in section \ref{-----sec:Free P}. For compact surfaces of higher genus, a natural candidate is given as follows. Recall that for any compact surface $\Sigma$ of volume $T>0$ and genus $g\geq 2$, there is a covering map $p: \Dbb_\mfh\to \Sigma$  mapping the hyperbolic metric of $\Dbb_\mfh$ to the metric of $\Sigma$. 

\begin{conj}\label{conj_Lift}For any compact surface $\Sigma$ of genus $g\ge 2,$ with universal cover
$p: D_\mfh\to \Sigma,$ the conjecture \ref{conj_Sing} is valid with 
\begin{equation}
\Phi_{\Sigma}(\ell)=\left\{\begin{array}{ll} \Phi_{\Dbb_\mfh}(\tilde \ell) & \text{if }\ell\text{ is contractible,} \\ &\\  0 &\text{ otherwise.} \end{array}\right. \label{eq:Conj Lift}
\end{equation}
\end{conj}

In Lemma \ref{lem-----positivity Projected MF},  we check directly\footnote{without  using a matrix approximation such as \eqref{eq:CV in P ConjSin}} that the map considered in Theorem \ref{-->THM: Torus Intro} and  \eqref{eq:Conj Lift} is associated to a state.    The conjecture \ref{conj_Lift} is also justified by the main result of \cite{DL} which leads to the following. Recall that a simple loop $\ell$ of $\Sigma$ is separating, if the set $\Sigma\setminus \ell$, where $\ell$ also denotes the range of the loop, has two connected components $\Sigma_{1,\ell},\Sigma_{2,\ell}$.

\begin{coro} \label{coro---off handle} If  $\ell$ is a  separating loop of compact surface $\Sigma$ of genus $g\geq 1$ and $\Sigma_{2,\ell}$ is not a disc, then under $\YM_{\Sigma},$ the convergence \eqref{eq:CV in P ConjSin} holds true with the limit  
\eqref{eq:Conj Lift}, for all loops $\ell$  in $\Sigma_{1,g}.$
\end{coro}

We obtained here two conditional results proving stability of the claimed convergence.

\begin{prop}\label{--->Prop: Intro Non-Nul homology} For any compact surface of genus $g\ge 2,$ when $G_N$ is a classical compact matrix group of size $N$, assume that for any geodesic loop $\ell$ of $\Sigma$ with non-zero homology, under $\YM_{\Sigma},$
\begin{equation}
W_\ell\to 0\ \text{in probability as} \ N\to\infty.\label{eq: Vanishing Loops Non-zero Homology}
\end{equation}
Then \eqref{eq: Vanishing Loops Non-zero Homology} also holds true for all loops with non-zero homology.
\end{prop}

Assume  $g\ge 2$ and  $\Gamma_g$ is a discrete subgroup of isometry acting freely, properly on   $\Dbb_\mfh$ and that $\Dbb_\mfh/\Gamma_g$ is a 
compact 
surface of genus $g$ with finite total volume $T>0.$ There is a fundamental domain for this action given by a $4g$ hyperbolic polygon $D$ of 
volume $T,$ centred at $0.$

\begin{thm} \label{-->THM: Intro Delayed Geodesic} The conjecture \ref{conj_Lift} holds true if \eqref{eq: Vanishing Loops Non-zero Homology} is true for 
every   non contractible,  loop $\ell$ of $\Sigma$ such that its lift $\tilde\ell$ to $\Dbb_\mfh$ can be written   $\tilde \ell=\g_1\g_2,$ where $\g_2$ is a geodesic, and $\g_1$ is smooth, included in $\overline D$ and intersecting $\pl D$ at most once, transversally at its endpoint.
\end{thm}

A more precise statement is given in Theorem \ref{-->THM Wilson Loops Higher Genus}. Besides, the recent results of  \cite{MageeII} are furthermore coherent with the above statement as discussed in the next sub-section. 

\subsection{Atiyah--Bott--Goldman measure}

Another measure on connections is due to Atiyah, Bott and Goldman \cite{AB,Goldman} when $g\ge 2$. Recently, the limit of Wilson loops under this measure has been investigated by \cite{MageeI,MageeII}, we discuss the relation with our result.

\vspace{0.5em}

Let $G$ be a compact connected semisimple\footnote{Mind that this excludes $\U(N)$.} Lie group $G$, $\mathfrak{g}$ its Lie algebra, endowed with an invariant inner product, and $Z(G)$ its center. For any $g\geq 2$, let $K_g:G^{2g}\to G$ be the product of commutators:
\[
K_g(a_1,b_1,\ldots,a_g,b_g) = [a_1,b_1]\cdots[a_g,b_g].
\]
{The space
\[
\mathcal{M}_g=K_g^{-1}(e)/G
\]
is called the \emph{moduli space of flat} $G$-\emph{connections} over a compact surface of genus $g\geq 2$, where $G$ acts by diagonal conjugation, as 
$$h.(z_1,\ldots, z_{2g})=(hz_1h^{-1},\ldots , hz_{2g}h^{-1}),\ \forall z\in G^{2g},g\in G.$$}
For any $z\in G^{2g}$, its isotropy group is $Z_z=\{h\in G, h.z=z\}.$  The set $\mathcal{M}_g^0=\{z\in G^{2g}: Z_{z}=Z(G)\}$ can be shown to be a  
manifold \cite{Goldman,Sen6} of dimension $2g-2,$ endowed with a symplectic form $\omega_A$ with finite total volume.  Besides, using the holonomy map along a suitable $2g-$tuple $\ell_1,\ldots,\ell_{2g}$ of loops, $\mathcal{M}_g^0$ can be identified with a subset of smooth connections  $\omega$ on a $G$-principal bundle over $\Sigma$ such that $S_{\YM}(\omega)=0.$ This subset is  a manifold with a symplectic structure \cite{AB}, equal to the push-forward of $\omega_A$. The Atiyah--Bott--Goldman measure  is the volume form on $\mathcal{M}_g^0$ associated to $\omega_A$, given by
\begin{equation}
\vol_{g} = \frac{\omega_A^{\frac12\dim\mathcal{M}_g^0}}{(\frac 1 2\dim\mathcal{M}_g^0)!}.
\end{equation}
Let us denote by $\mu_{ABG,g}$ the probability measure on $\mathcal{M}^0_g$ obtained by normalising $\vol_g$. 
It appeared in \cite{Wit}, that integrating against the Yang--Mills measure on a compact surface of total area $T$ and letting $T$ tend to 0, allows to obtain formulas for  integrals against $\mu_{ABG,g}$. This convergence was proved rigorously by Sengupta in \cite{Sen6}. 
Using the holonomy mapping of the Yang--Mills measure, the convergence can be understood as follows. Consider a heat kernel $(p_t)_{t>0}$ on $G$, when its Lie algebra $\mathfrak{g}$ is endowed with its Killing form 
$\<\cdot,\cdot\>$. 

\begin{thm}[Symplectic limit of Yang--Mills measure]
Let $f:G^{2g}\to\C$ be a continuous $G$-invariant function, and $\tilde{f}:\mathcal{M}_g^0\to\C$ be the induced function on the moduli space. Then
\begin{equation}
\lim_{T\downarrow 0}\int_{G^{2g}} f(x) p_T(K_g(x))dx=\frac{\vol(G)^{2-2g}}{|Z|}\int_{\mathcal{M}_g^0}\tilde{f}d\vol_g.
\end{equation}
\end{thm}
For any  word $w$ in the variables $a_1,\ldots,b_g$ and their inverses, setting
$$W_{w}(z)=\frac{1}{N}\Tr(w(z_1,z_1^{-1},\ldots,z_{2g},z_{2g}^{-1} )),\ \forall z \in G^{2g}$$
defines also a function on $\mathcal{M}^0_g$. Denoting it also by $W_w$ and considering the loop $\ell_w$ obtained by the concatenation  $w(\ell_1,\ell_1^{-1},\ldots, \ell_{2g},\ell_{2g}^{-1})$, the last statement can be reformulated as 
$$\lim_{T\downarrow0} \E_{\YM_{\Sigma_T}}[W_{\ell_w}]= \int_{\mathcal{M}_g^0}W_{w}d\mu_{ABG,g},$$

Given the surface group
$$\Gamma_g=\<a_1,b_1,\ldots,a_g,b_g|[a_1,b_1]\ldots [a_g,b_g]\>,$$
consider the equivalence relation $\sim$ on the set of words with $2g$ letters and their inverses, such that $w\sim w'$ iff $w(a_1,\ldots,b_g)$ and 
$ w'(a_1,\ldots,b_g)$ are equal in $\G_g.$  Thanks to the defining relation of $\mathcal{M}_g,$ for any word $w$, the function $W_w$ depends only on 
the equivalence class of $w$.  When $\g\in \G_g$  is the evaluation of $w$ in $\G_g$, denote this function by $W_\g$.  In \cite{MageeII}, Magee obtained 
the 
following 
analog of asymptotic freeness of Haar  unitary random matrices.

\begin{thm}[\cite{MageeII} Cor. 1.2] Consider the group $G=\SU(N).$ For any $\gamma\in\Gamma_g$, 
$$\lim_{N\to \infty} \E_{\mu_{ABG,g}}[W_\g]= \left\{\begin{array}{ll} 1& \text{ if }\g=1,\\ &\\ 0 &\text{ otherwise.}\end{array}\right.$$
\end{thm}
Since for any word with evaluation $\g\in \G_g,$ it can be shown that $\g=1$ if and only if the loop $\ell_w$ is contractible, the above statement can be 
understood as  the $T=0$ case of the conjecture \ref{conj_Lift}, with a weaker 
convergence given in expectation instead of in probability. In \cite{MageeI}, it is also shown that 
$\E_{\mu_{ABG,g}}[W_\ell]$ admits an asymptotic expansion in powers of $\frac{1}{N}$.

\vspace{0.5em}

Let us discuss the main differences between the approach of \cite{MageeI,MageeII} and ours:

\begin{itemize}
\item Although both approaches use the convergence of the partition function of the model, we use in \cite{DL} the Markov property of the Yang--Mills holonomy field in order to prove the convergence for simple loops, then we use the Makeenko--Migdal equations to induce the convergence on a larger class of loops; the latter is actually not needed in the zero volume case.
\item We only consider the limit of Wilson loops, whereas \cite{MageeI}  proves a $\frac1N$ expansion.
\item We prove a convergence in probability whereas \cite{MageeII}  gets a convergence in expectation.
\item We also consider a larger family of matrix groups, whereas he only treats the unitary case.
\item In the case $g=1$, the Atiyah--Bott--Goldman measure is ill-defined, hence Magee's paper cannot handle it, but we still find a result when $T>0$, which gives a matrix approximation of an interpolation between classical and free convolution of Haar unitaries.
\end{itemize}

\subsection{Non-commutative distribution and master field on the torus: an interpolation between free and classical convolution}

\label{-----sec:Free P}

We discuss here the non-commutative distribution associated to the master field on the torus, leading to  the corollary \ref{-->CORO: Interpol Class <--> Free} below, obtained by specialising  Theorem \ref{-->THM: Torus Intro} 
to projection of loops  restrained to the lattice $\sqrt{T}.\Z^2.$

\subsubsection{Non-commutative probability and free independence}

Let us give an extremely brief account of these notions. We refer to \cite{VoiLect,MingoSpeicher}  for more details. A \emph{non-commutative probability space}\footnote{sometimes denoted NCPS} is the data of a tuple $(\mathcal{A},*,1,\tau)$ where $(\mathcal{A},*,1)$ is a unital $*$-algebra over $\C,$ and $\tau$ is a positive, tracial state, that is a linear map $\tau:\mathcal{A}\to\C$  with 
$$\tau(aa^*)\ge 0\text{ and }\tau(ab)=\tau(ba),\ \forall a,b\in \mathcal{A},$$
with furthermore $\tau(1)=1$ and $\tau(a^*)=\overline{\tau(a)},\ \forall a\in\cA.$ We shall often leave as implicit the choice of unit and $*$, and denote a non-commutative probability space simply as a pair $(\cA,\tau).$ 

\begin{ex} For $N\ge 1,$ the tuple $(M_N(\C),*, \mathrm{Id}_N,\tr)$, where $\tr=\frac 1 N\Tr,$ gives such a space. Consider  the group  $\U(N)$ of unitary 
complex matrices of size $N $ and a group $\G$ with unit element $1$. Let $(\C[\G],*)$ be the group algebra of $\G$ endowed with the skew-linear idempotent defined by $\g^*=\g^{-1},\ \forall \g\in \G.$   Then, whenever  $\rho:\G\to  \U_N(\C)$ is a 
unitary representation of $\G,$ setting $\tau_\rho=\tr\circ\rho$, the tuple 
$(\C[\G],*,1,\tau_\rho)$ is a non-commutative probability space.
\end{ex}

Let $(\mathcal{A}_1,\mathcal{A}_2)$ be unital sub-algebras of a non-commutative probability space $\mathcal{A}_1$.
\begin{itemize}
\item They are \emph{classically independent} if $\forall a_1,\ldots,a_n\in\mathcal{A}_1, b_1,\ldots,b_n\in\mathcal{A}_2,$
\[
\tau(a_1b_1a_2\ldots a_nb_n)=\tau(a_1\ldots a_n)\tau(b_1\ldots b_n).
\]
\item They are \emph{freely independent}    if for any $n\in\N$, for any $\{i_1,\ldots,i_n\}\in\{1,2\}^n$ such that $i_1\neq i_2,\ldots,i_{n-1}\neq i_n$ and for any $a_k\in\mathcal{A}_{i_k}$,
\[
\tau(a_k)=0,\ \forall 1\leq k\leq n \Longrightarrow \tau(a_1\cdots a_n)=0.
\]
\end{itemize}

These definitions can be generalised to any number of sub-algebras, and a family of elements $(a_i)_{i\in I}$ of a non-commutative probability space $(\mathcal{A},\tau)$ is said to be independent (resp. free) if the family $(\mathcal{A}_i)_{i\in I}$ is independent (resp. free), where for all $i\in I$, $\mathcal{A}_i$ is the subalgebra generated by $a_i$ and $a_i^*$. We shall then say that $(a_i)_{i\in I}$ are resp. independent and free under $\tau.$

\vspace{0.5em}

When $I$ is an arbitrary set, let us denote by $\C\<X_i,X_i^*,i\in I\>$ the unital $*$-algebra of non-commutative polynomials in the variables $X_i,X_i^*,\in I,$ with $*$ mapping $X_i$ to $X_i^*$ for all $i\in I.$  When 
$(\mathcal{A},*,1,\tau)$ is a non-commutative probability space and $\mathbf{a}=(a_i)_{i\in I}$ is a  family of elements of $\mathcal{A}$, its \emph{non-commutative distribution} is the positive, tracial, state on $\C\<X_i,X_i^*,i\in I\>$ given by
$$\tau_\mathbf{a}(P)=\tau(P(a_i,i\in I)),\ \forall P\in \C\<X_i,X_i^*,i\in I\>,$$
where $P(a_i,i\in I)\in\cA$ denotes the evaluation  of $P$ replacing $X_i$ and $X_i^*$ by $a_i$ and $a_i^*.$  Likewise, when $\cA$ and $\cB$ are sub-algebras of a same non-commutative probability space $(\cC,\tau)$, 
we call the state $\tau_{\<\cA,\cB\>}$ on $\C\<X_a,Y_b,a\in \cA,b\in \cB \>$ given by 
$$\tau_{\<\cA,\cB\>}(P(X_a,Y_b;a\in \cA,b\in \cB))=\tau(P(a,b;a\in \cA,b\in\cB)),$$
the joint distribution of $(\cA,\cB)$ in $(\cC,\tau).$

\vspace{0.5em}

When $a,b$ are two elements of non-commutative probability spaces with respective non-commutative distribution $\tau_a$ and $\tau_b$,  there are 
unique states $\tau_a\star\tau_b$ and $\tau_a*_c\tau_b$ on 
$\C\<X,Y,X^*,Y^*\>$ such that   $\tau_X=\tau_a$ and $\tau_Y=\tau_b$ both under and $\tau_a*_c\tau_b$ and $\tau_a\star\tau_b$, while the joint 
distribution $(X,Y)$ under   $\tau_a\star\tau_b$ and $\tau_a*_c\tau_b,$  are  
respectively freely and  classically independent. The states $\tau_a\star\tau_b$ and $\tau_a*_c\tau_b$ are resp. called the \emph{free} and the classical  
\emph{convolution} of $\tau_a$ and $\tau_b.$  We define 
likewise the free and classical convolution of two states on $\tau_\cA,\tau_\cB$ of NCPS $(\cA,\tau_\cA),(\cB,\tau_\cB)$ as states 
$\tau_\cA\star\tau_\cB$ and 
$\tau_\cA*_c\tau_\cB$  on $\C\<X_a,Y_b,a\in\cA,b\in\cB\>.$ 

\vspace{0.5em}

Let us recall the following result of  \emph{asymptotic freeness} due to Voiculescu \cite{VoiFreeRM}, and for the considered group series by \cite{CollinsSnia}, see also \cite[Sect. I-3]{Lev}.

\begin{thm}[\cite{VoiFreeRM,CollinsSnia,Lev}] Let $A$ and $B$ be two deterministic matrices of size $N$ with respective non-commutative distribution satisfying for all fixed $P\in \C\<X,X^*\>,$
$$\tau_A(P)\to \tau_a(P),\tau_B(P)\to \tau_b(P),  \ \text{ as} \ N\to\infty,$$
for some state $\tau_a,\tau_b$ on $\C\<X,X^*\>.$ 
Consider  $U$ and $V$ two independent  Haar unitary matrices on a group $G_N$ and $\rho_N: \C[\Fbb_2]\to G_N$ the associated unitary representation of the free group of rank $2.$ 

Then for any $\g\in \Fbb_2$ and $P\in \C\<X,Y,X^*,Y^*\>,$ the following limit holds in probability as $N\to \infty$,
\begin{equation}
\tau_{\rho_N}(\g)\to \left\{\begin{array}{ll} 1& \text{ if }\g=1,\\ &\\ 0 &\text{ if }\g\in \Fbb_2\setminus\{1\}\end{array}\right.\label{eq: Conv FreeGp}
\end{equation}
and 
\begin{equation}
\tau_{A,UBU^*}(P)\to \tau_a\star\tau_b(P).\label{eq: Conv Free Indep}
\end{equation}
\end{thm}
On the one hand, the first convergence \eqref{eq: Conv FreeGp} can be proved to be  a special case of \eqref{eq: Conv Free Indep} when $A$ and $B$ 
are themselves independent Haar unitary random variables. On the other hand, 
when $A$ and $B$ are unitary or  Hermitian with uniformly bounded spectrum,
\eqref{eq: Conv Free Indep} can be deduced from \eqref{eq: Conv FreeGp} by functional calculus.

One of the motivations of the current article was to understand an analog of \eqref{eq: Conv FreeGp}, when $(U,V)$ are sampled according to a different law with 
correlation, as
discussed in section \ref{-----sec:Matrix Approx Area Interpol }.

\subsubsection{Free Unitary Brownian motion and $t$-freeness} 
We refer here to \cite{Bia,Voi,BenaychLev} for more details. Consider a non-commutative probability space $(\mathcal{A},\tau,*,1)$. An element $u\in \cA$ is called unitary when $uu^*=u^*u=1.$ It is Haar unitary if for any integer $n>0,$ $\tau(u^n)=\tau((u^{*})^n)=0.$ The \emph{free unitary Brownian motion} on a $*$-probability space $(\mathcal{A},\tau,*,1)$ is a family $(u_t)_{t\geq 0}$ of unitary elements of $\mathcal{A}$ such that the increments $u_{t_1}u_0^*,\ldots,u_{t_n}u_{t_{n-1}}^*$ are 
free for all $0\leq t_1\leq\cdots\leq t_n$, and for any $k\in \Z^*$ and $0<s<t,$
\[\tau((u_tu_s^*)^k )=\tau(u_{t-s}^k )\]
while $\tau(u_{t}^k)=\nu_t(|k|)$ is $C^1$ with for all $m\ge 0,$
\begin{equation}
\frac{d}{dt}\nu_t(m)=-\frac{m}{2}\nu_t(m)-\frac{m}{2}\sum_{l=1}^{m}\nu_t(l)\nu_t(m-l),\ \forall t\ge 0, \ \nu_0(m)=1.\label{eq:ODE UBM}
\end{equation}
Let us set  $\nu_t=\tau_{u_t}$. It follows from the above expression that as $t$ tends respectively to $0$ and $+\infty,$ the distribution $\nu_t$ converges pointwise to the one of  respectively $1$ and a Haar unitary. In view of 
\eqref{eq: Conv Free Indep}, it  is also natural to introduce the following deformation of free convolution.

\begin{dfn}[\cite{Voi}]
Let $(\mathcal{A},\tau_\mathcal{A})$ and $(\mathcal{B},\tau_\mathcal{B})$  be two non-commutative probability spaces. Then there is a non-commutative probability space $(\mathcal{C}^{(t)},\tau_{\mathcal{C}^{(t)}})$ such that 

\begin{enumerate}
\item $\cA$ and $\cB$ can be identified with two \emph{independent} sub-algebras of  $(\mathcal{C}^{(t)},\tau_{\mathcal{C}^{(t)}})$ with  
$${\tau_{\mathcal{C}^{(t)}}}(a)=\tau_\cA(a) \text{ and }{\tau_{\mathcal{C}^{(t)}}}(b)=\tau_\cB(b),\ \forall (a,b)\in \cA\times\cB.$$   
\item  There is a unitary element $u_t\in \mathcal{C}^{(t)}$ free with the sub-algebra of $\mathcal{C}^{(t)}$ generated by $\cA$ and $\cB$, such that  
$u_t$ has distribution $\nu_t.$
\end{enumerate}
The $t$\emph{-free convolution product} of $\tau_\cA$ and $\tau_\cB$ is then the joint distribution $\tau_\cA\star_t\tau_\cB$ of  
$(\mathcal{A}, u_t\mathcal{B}u_t^*)$ 
in the non-commutative probability space $(\mathcal{C}^{(t)},\tau_{\mathcal{C}^{(t)}}).$ It does not depend on the choice of 
$(\mathcal{C}^{(t)},\tau_{\mathcal{C}^{(t)}})$ satisfying 1) and 2).  
\end{dfn}

The above construction was introduced more generally
\footnote{not necessarily with the assumption of classical independence for the initial state.} by Voiculescu \cite{Voi} in his study of free entropy and free Fisher information via the liberation process. For any $t>0,$  two sub-algebras $\cA$ and $\cB$ of a same non-commutative probability space $(\cC,\tau)$ with respective distribution $\tau_\cA$ and 
$\tau_\cB$ are said to be \emph{$t$-free}, if their joint distribution under $\tau$ is 
given by 
$\tau_\cA\star_t\tau_\cB.$   It can be shown (\cite{Voi,BenaychLev}) that the following limits hold pointwise,  
\[\lim_{t\downarrow0}\tau_\cA\star_t\tau_\cB =\tau_{\cA}*_c\tau_{\cB} \text{ and }\lim_{t\to +\infty}\tau_\cA\star_t\tau_\cB =\tau_{\cA}\star\tau_{\cB}.\]

\subsubsection{A matrix approximation for another interpolation from classical to free convolution}
\label{-----sec:Matrix Approx Area Interpol }

Let us present an application of Theorem \ref{-->THM: Torus Intro}. Consider a  heat kernel $(p_t)_{t>0}$ on a classical compact matrix Lie group  $G_N$ endowed with the metric  considered in section \ref{----sec:HK} and for any 
$T>0$, define a  probability 
measure setting
\begin{equation}
d\mu_{N,T}(A,B)=Z_T^{-1}p_T([A,B])dAdB\label{eq:Law GenLoops}
\end{equation}
on $G_N^2 $ where $dAdB$ denotes the Haar measure on $G_N^2$ and $Z_T=\int_{G_N^2} p_T([A,B])dAdB.$  As the limits  
$\lim_{T\downarrow 0}p_T(U)dU= \delta_{\Id_N}$ and $\lim_{T\to \infty}p_T(U)dU= dU$ hold weakly,  we can think about $\mu_{T}$ as a model of 
random 
matrices interpolating between  commuting and non-commuting settings. In \cite[Thm 2.15]{DL}, we have proved that though $A$ and $B$ are not Haar distributed for 
$N$ fixed, as $N\to \infty, $ they converge individually  
to Haar unitaries. Moreover,  we also saw that under $\mu_{N,T},$ $[A,B]$ converges in non-commutative distribution, 
with limit given 
by $\nu_T,$ a free unitary Brownian motion at time $T.$  In view of \eqref{eq: Conv FreeGp},  it is then natural to investigate the possible limit of the joint 
law, hoping for a non-trivial coupling of Haar unitaries.    Note that analog models with 
potentials\footnote{Though the class of potentials considered in \cite{ColGuionnetMS} do not 
cover the heat kernel. }  have been investigated in \cite{ColGuionnetMS}. A challenge appearing in the setting of \cite{ColGuionnetMS} is that these general results  are 
limited to  
weak 
coupling regimes.\footnote{meaning that the parameter of the potential responsible for the non-independence of $A$ and $B$ needs to be small enough.}  A consequence of
our work is that  $\mu_{N,T}$ has a non-commutative limit for all $T>0$, leading to an interpolation between independent and free Haar unitaries. Denote 
by $\tau_u$ the distribution of a Haar unitary.

\begin{coro}  \label{-->CORO: Interpol Class <--> Free}For any $T>0$, there is  a state $\Phi_T$ on $ \cA= \C\<X,X^*,Y,Y^*\>,$  such that for any $P\in \cA,$ under $\mu_{N,T}$,
$$\tr ( P(A,B)) \to \Phi_T(P) \text{ in probability as }N\to\infty$$
with 
$$\lim_{T\downarrow 0}\Phi_T(P)=\tau_u *_c \tau_u(P) \text{ and }\lim_{T\to +\infty}\Phi_T(P)=\tau_u \star \tau_u(P).$$
Besides, for all $T,t>0,$ 
\begin{equation}
\Phi_T\not=\tau_u\star_t\tau_u,\label{eq:Free}
\end{equation}
while $$\Phi_T((XYX^*Y^*)^n)=\nu_T(n)=\tau_u\star_{\frac T 4}\tau_u((XYX^*Y^*)^n),\ \forall n \in \Z^*.$$
\end{coro}

We prove in section \ref{-----sec:Liberation}  the above corollary together with a few other properties of $\Phi_T.$ Let us mention that the interpolation provided by Corollary \ref{-->CORO: Interpol Class <--> Free} is not the only possible interpolation, even if we exclude the $t$-free convolution; for instance another interpolation was proposed in  \cite{MerPot}  using rank one Harish--Chandra--Itzykson--Zuber integrals.  Let us also mention that there are variations of freeness for family of algebras which are partly commuting \cite{MR2044286,MR3552222}. 
\subsection{Strategy of  proof via Makeenko--Migdal deformations}

\label{-----sec:Strategy}

An important property formally inferred by integration by part from \eqref{eq:YM_informal} in \cite{MM} and rigorously proved in \cite{Lev} based on the 
Driver--Sengupta formula, are a family of equations almost characterising the function $\Phi_\Sigma$ when $\Sigma$ is the plane. Other proofs have been given in \cite{Dah2,DHK}. The proofs of \cite{DHK} were much shorter and local, and it was possible to adapt them to all compact surfaces \cite{DGHK}. See also \cite{DriverMM} for a different approach based on the construction of the Yang--Mills measure via white noise and \cite{PPSY} for a proof based on the representation of Wilson loop expectations as surface sums.

\vspace{0.5em}

These equations can be described informally as follows. Consider a smooth loop $\ell$ with a transverse intersection at a point $v.$  Assume that $(\ell_\varepsilon)_{\varepsilon }$ is a deformation of $\ell$ in a neighborhood of $v$ such that the areas of the four corners adjacent to $v$ are modified as in Figure \ref{Fig----MMDef}.
\begin{figure}[!h]
\centering
\includegraphics[scale=0.6]{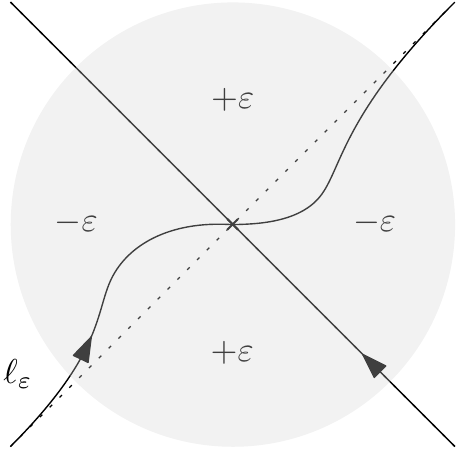}
\caption{\small Makeenko--Migdal deformation near an intersection point.}\label{Fig----MMDef}
\end{figure}
Then the Makeenko--Migdal equation at $v$ for a master field $\Phi_\Sigma$ is given by
\begin{equation}
\left.\frac{d}{d\varepsilon}\right|_{\varepsilon=0}\Phi_\Sigma( {\ell_\varepsilon})=\Phi_\Sigma(\ell_1) \Phi_\Sigma(\ell_2)\label{eq:MM Inform}
\end{equation}
where $\ell_1,\ell_2$ are two loops obtained by de-singularising $\ell$ at $v$ as on figure \ref{Fig----MMDeSing}.

\begin{figure}[!h]
\centering
\includegraphics[scale=0.4]{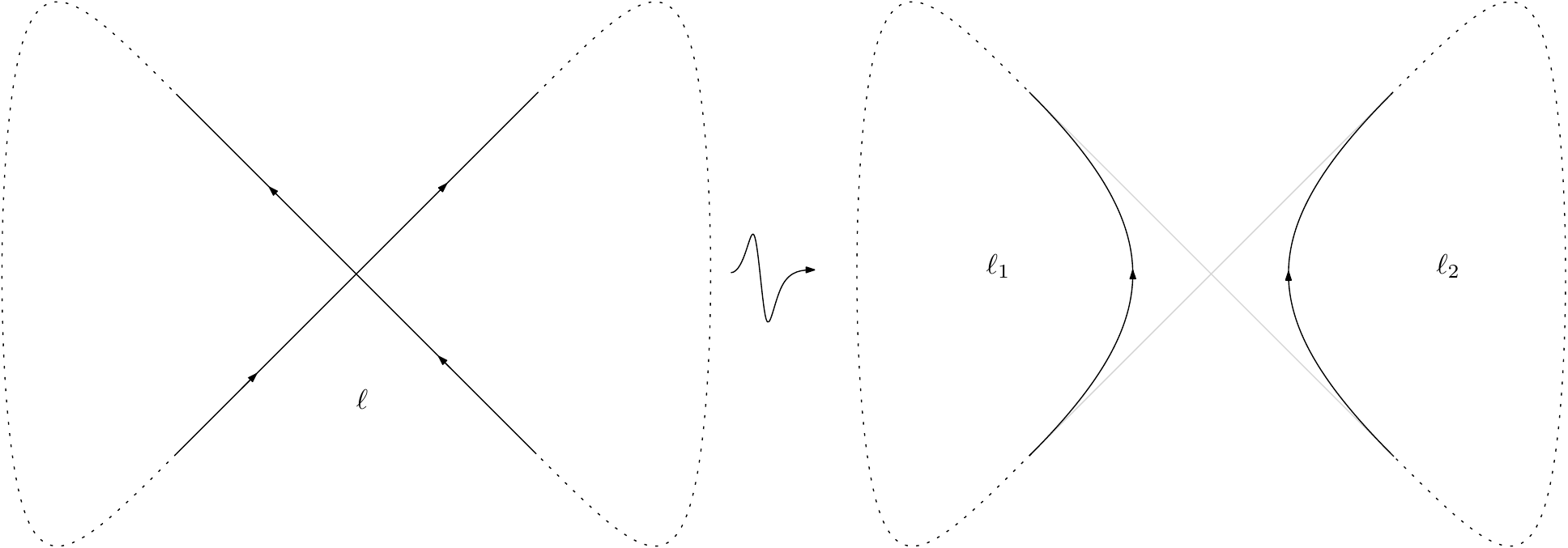}
\caption{\small De-singularisation at a simple intersection point.}\label{Fig----MMDeSing}
\end{figure}

The works \cite{Lev,DN,Hal2} can be understood as a study of existence and uniqueness of variants of the equation \eqref{eq:MM Inform}.  Our strategy 
here is to extend these results to all compact surfaces of genus $g\ge 1$.

\vspace{0.5em}

A motivation  of \cite{Lev} for proving these relations was to compute explicitly the planar master field by induction on the number of intersections and to characterise it through differential equations. It was realised there that for the plane there is no uniqueness for the Makeenko--Migdal equations alone, but there is if they are completed by an additional family of equations\footnote{associated to each face adjacent to an infinite face.}. In \cite{DN,Hal2}, the authors are interested in a perturbation of \eqref{eq:MM Inform} arising from finite $N$ analogs of \eqref{eq:MM Inform} in view of proving the convergence of Wilson loops. The same lack of uniqueness occurs but is dealt with differently, adding in some sense boundary conditions, specifying the value of the master field\footnote{or the convergence of Wilson loops} for simple loops. With this boundary condition, both \cite{DN,Hal2} are able to deduce the convergence of Wilson loops\footnote{This argument is valid for loops with finitely many transverse intersections. An additional step which is not considered in \cite{Hal2} is to extend it to loops with finite length.} on the sphere by induction on the number of intersection points. To complete the proof of Wilson loops convergence, it is then necessary to prove the convergence for boundary conditions via other means: this was done in \cite{DN} using a representation through a discrete\footnote{as suggested in \cite{Hal2}, another route here could be to relate Wilson loops for simple loops 
on the sphere to the Dyson Brownian bridge on the unit circle, which has been studied recently at another scale in \cite{LiechtyWang}. } $\beta$-ensemble. 

\vspace{0.5em}

In \cite{Hal2}, the author applied the same argument on all compact surfaces with a boundary condition given by simple loops within a disc and a 
uniquess or convergence result for loops within a disc. See the introduction of \cite{DL} for a more detailed discussion. In \cite{DL} we were 
able, using an independent argument, to prove the same result but without any boundary condition and making a relation with the planar master field.

\begin{thm}\label{--->THM: Simple loops Intro}Let $\ell$ be a loop in a compact connected orientable Riemann surface $\Sigma$ of genus $g\geq 1$ with area measure $\vol$.
\begin{enumerate}
\item If $\ell$ is topologically trivial and included in a disc $U$ such that $\vol(U)<\vol(\Sigma)$, then as $N$ tends to infinity, under $\YM_\Sigma$,
\[
W_\ell\to\tilde\Phi(\psi\circ\ell) \ \text{in probability}	,
\]
where $\tilde\Phi$ denotes the master field in the planar disc $\psi(U)$ where $\psi:U\to \psi(U)\subset\R^2$ is an area-preserving diffeomorphism.
\item If $\ell$ is simple and non-contractible, then for any $n\in\Z^*$, as $N$ tends to infinity,
\[
W_{\ell^n}\to 0 \ \text{in probability}.
\]
\end{enumerate}
\end{thm}

A first remark is that evaluating the planar master field at lift of contractible loops to   the universal cover of $\Sigma$, as in the conjecture \ref{conj_Lift}, gives  a solution to Makeenko--Migdal equations. Our main focus will therefore be to study uniqueness of the Makeenko--Migdal equations or its deformation arising for finite $N$.

\vspace{0.5em}

The general strategy of this article is to use  
Theorem \ref{--->THM: Simple loops Intro} as boundary condition to prove  Proposition \ref{--->Prop: Intro Non-Nul homology} and  
Theorem \ref{-->THM: Intro Delayed Geodesic}. 
For the torus, any non-trivial closed geodesic is whether simple or the iteration of a simple closed loop, Proposition  \ref{--->Prop: Intro Non-Nul homology} together with Theorem
\ref{-->THM: Intro Delayed Geodesic} yield Theorem \ref{-->THM: Torus Intro}. For surfaces of genus  $g\ge2$, the result of \cite{DL} do not  the loops in the assumption of Theorem \ref{-->THM: Intro Delayed Geodesic}
and there are then loops whose 
homotopy class does not include any simple loop, or any loop obtained by iterating a simple loop \cite{BiS} (moreover most geodesics have intersection points).

\vspace{0.5em}

Let us now discuss how this strategy is implemented here. When applying the  argument of \cite{DN} or \cite{Hal2}, it is difficult to prove a result better than Theorem \ref{--->THM: Simple loops Intro}, 
which, given point 1. of Theorem \ref{--->THM: Simple loops Intro}, makes the use of Makeenko--Migdal equations pointless. A first obstacle being for instance a loop like in figure \ref{----Fig:RaqMax}, where it does not seem 
possible to apply Makeenko--Migdal equations at any vertex to deform the loop into a simpler loop. 

\begin{figure}[!h]
\centering
\includegraphics[scale=0.6]{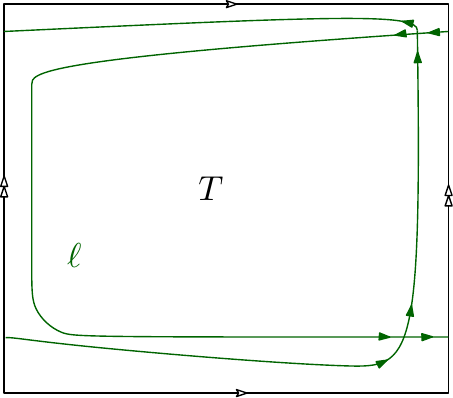}
\caption{\small In this example, it is impossible to change the area  around any intersection point, respecting the constraint of  Makeenko--Migdal given in figure \ref{Fig----MMDef}, without raising the number of intersection points. }\label{----Fig:RaqMax}
\end{figure}

To improve on \cite{Hal2}, a first step is to  characterise for surfaces of genus $g\ge 1,$ the allowed deformations  in Makeenko--Migdal equations. Viewing the evaluation at a regular loop of the master field as a function of faces area, we wonder along which deformation of loops, the derivative of the master field is a linear combination of area derivatives such as the one  involved in the left-hand-side of \eqref{eq:MM Inform}. This was understood first in the plane by \cite{Lev}. This is achieved here for surfaces in section \ref{----sec: Makeenko--Migdal Vectors} with the following conclusions:
\begin{itemize}
\item When a loop has non-zero homology, then any reasonable deformation is allowed;
\item When a loop  has zero homology, then it is possible to define the winding number and algebraic area of the 
loop and a deformation is allowed if and only if it preserves the algebraic area.
\end{itemize}

This observation allows to consider the simpler case of loops with non-zero homology separately.  In this case, it is possible to argue as follows by induction, showing at each step that the derivative along a suitable deformation is bounded by induction assumption. First, considering  the lift of a loop with non-zero homology to the universal cover, by induction on the number of intersections, it is  
possible to reduce the problem to loops with non-zero homology such that each strand of the lift  going through a fundamental region has\footnote{We shall call below these loops proper loops} no intersection point. Then Proposition \ref{--->Prop: Intro Non-Nul homology} can be proved by induction on the number of 
fundamental domains visited. A key  remark in this case is that at each intersection point, the two loops obtained by de-singularisation have both non-zero homology and visit strictly less fundamental domains.  This programme is 
carried out in section \ref{-----sec:Section Non-null Homology Proofs}. 

\vspace{0.5em}

A second step is to overcome the difficulty met in Figure \ref{----Fig:RaqMax}. This loop has vanishing homology.  The cause of the obstruction becomes clearer thanks to the first step:  it is not possible to decrease the area of the central face as it is a strict maximum of the winding number function. An idea is  to deform the loop in a face that we want to ``inflate'' so that the algebraic area remains preserved, as suggested on the following figure.   

\begin{figure}[!h]
\centering
\includegraphics[scale=0.6]{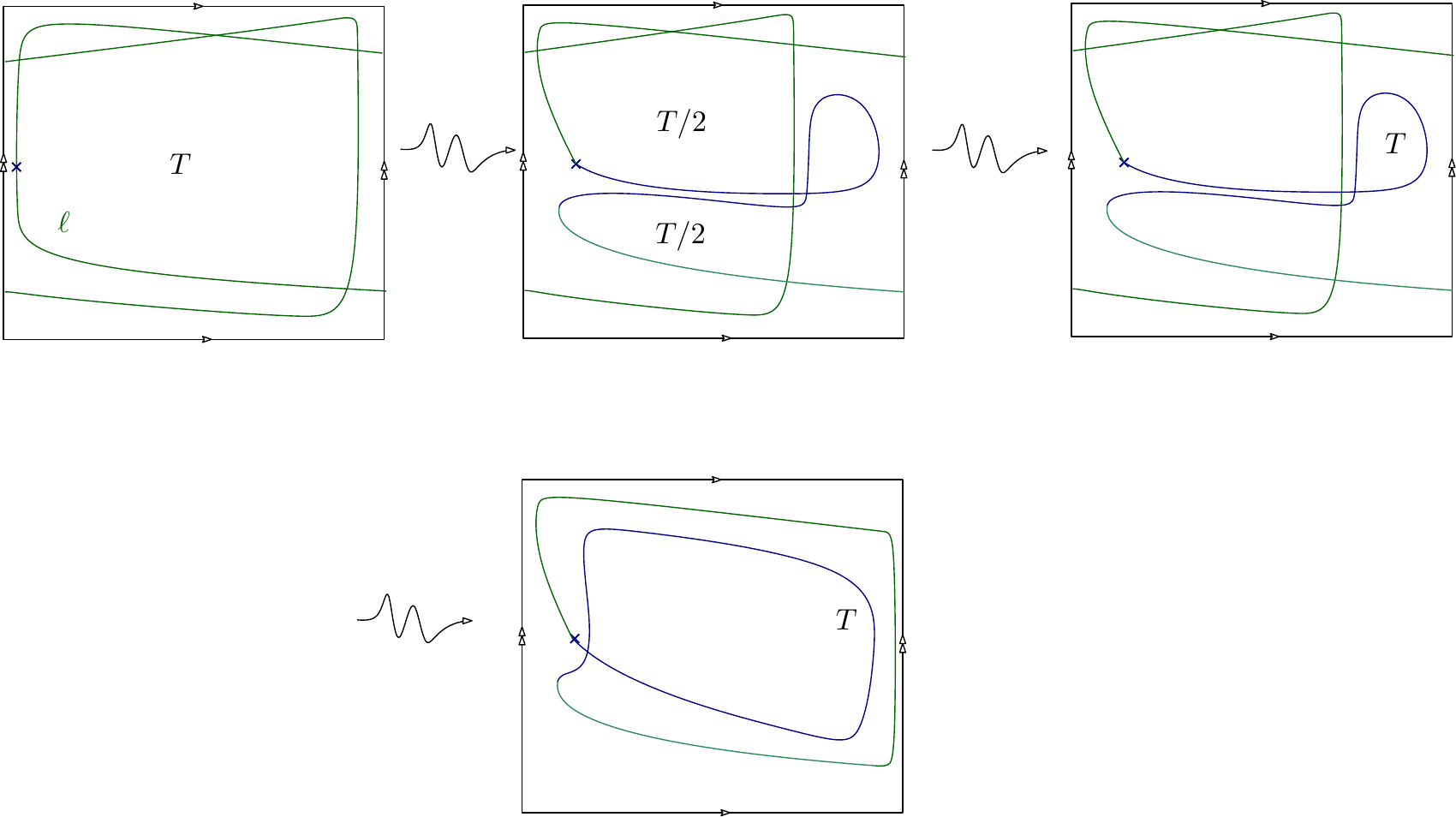}
\caption{\small Discrete homotopy towards a loop included in a disc preserving the algebraic area. Faces are labeled by their area. Faces without label have area $0$.}\label{Fig----Homotopie Raquette 2}
\end{figure}

An apparent issue with this argument is that the number of intersections of the loops involved in the different steps may  increase, preventing a direct    induction on the number of intersections  as in \cite{Lev,DN,Hal2}. 
To solve this issue we consider a family of ``marked'' loops, consisting of two paths whose concatenation  $\ell$ is a loop, where the second path is generic,  while the first one has a specific 
form\footnote{we call this part \emph{nested},  see section \ref{-----sec:Nested Marked Loops}.}. In particular we require that a loop obtained by  de-singularisation at an intersection point of the first part  is whether in a 
fundamental domain,  or  is the contraction of the  initial loop $\ell$
along  some  faces bounded by the perturbed part. Because of the nested structure of the perturbation part, it becomes  possible\footnote{see p. \pageref{example------Raq} where this idea is illustrated to prove the uniqueness of Makeenko--Migdal equations for the example of Figure \ref{Fig----Homotopie Raquette 2}.} to argue by induction  determining a complexity function on marked loops adapted\footnote{We believe there is a lot of flexibility here in the argument. We choose here a combinatorial approach, but it would be 
interesting to use instead a continuous functional on loops.} to the boundary conditions considered. The choice of complexity is done in  section \ref{====Sec: Shortening Homotopy sequence}, the induction is then 
performed in  Theorem \ref{-->TH:  Homotopy for null homology}.

\vspace{0.5em}

Lastly it remains to extend our convergence result to a wider family of loops.  This is first done using the property of  continuity and compatibility on closed simplices of areas  for loops with finitely many transverse intersections\footnote{see Lemma \ref{__Lem:YM Faces simplex}.}. Then a more general argument introduced in \cite{CDG,DN}, building on the construction of \cite{Lev2},  allows to consider all loops with finite length.\footnote{This second step is  not needed to consider projection of loops on a lattice.} 

\section{Homology and homotopy on embedded graphs}

\subsection{Four equivalence relations on paths and loops on maps}
\label{----sec:Maps}

We recall briefly here standard notions and define some notations of topological discretisation of a surface.  

\vspace{0.5em}

\begin{dfn}
A \emph{graph}\index{Graph} $\Gbb=(V,E,I)$ is a triple consisting of two sets $V$ and $E$ and an incidence relation $I\subset V\times E$ such that for any $e\in E$, the cardinal of $\{v\in V: (v,e)\in I\}$ is 1 or 2. The elements of $V$ (resp. $E$) are called \emph{vertices} (resp. \emph{edges}).
\end{dfn}

This definition might seem very abstract at first sight, but it is actually simple: it merely says that a graph is made of edges and vertices, and that each edge is incident to either 1 vertex (the edge is then called a loop) or 2 vertices. Let $\Gbb=(V,E,I)$ be a graph, and $e_1,e_2\in E$ be two distinct edges.
\begin{enumerate}
\item If there is $v\in V$ such that $(v,e_1)\in I$ and $(v,e_2)\in I$, then $e_1$ and $e_2$ are called \emph{adjacent}.
\item If there are $v_1,v_2\in V$ such that $(v_i,e_j)\in I$ for all $1\leq i,j\leq 2$, then $e_1$ and $e_2$ form a \emph{double edge}.
\end{enumerate}
More generally, if $n$ edges share the same incidence vertices, they form a multiple edge, and $\Gbb$ is called a \emph{multigraph}. A \emph{topological map} $M$ on a surface $\Sigma$ 
is a multigraph $\Gbb=(V,E,I)$ together with an embedding $\theta:\Gbb\to\Sigma$ such that:
\begin{itemize}
\item The images of two distinct vertices $v_1,v_2\in V$ by $\theta$ are distinct points of $\Sigma$,
\item The images of edges $e\in E$ are continuous curves $\theta_e:[0,1]\to\Sigma$ with endpoints $\und e = \theta_e(0)$ and $\ov e = \theta_e(1)$ such that $\theta_e$ and $\theta_{e'}$ can only intersect at their endpoints,
\item For any edge $e\in E$, there is an edge $e^{-1}\in E$ such that $\theta_{e^{-1}}=\theta_e^{-1}$,
\item The complement $F$ of the skeleton $\mathrm{Sk}(\Gbb)=\bigcup_{e\in E}\theta_e$ of $\Gbb$ in $\Sigma$ is split in one or several connected components that are all homeomorphic to discs, and represent the faces of the map.
\end{itemize}
An \emph{orientation} of the map is the choice of a subset $E_+\subset E$ such that for any $e\in E$, $\vert \{e,e^{-1}\}\cap E_+\vert =1$. The orientation of $\Sigma$ also induces an orientation of the faces as follows: a face $f$ is \emph{positively oriented} if its boundary $\partial f$ is represented by $e_1\cdots e_n$, where $e_1,\ldots,e_n$ are the edges constituting $\partial f$ in positive order. It is \emph{negatively oriented} if its boundary is represented by $e_n^{-1}\cdots e_1^{-1}$. We denote by $F$ (resp. $F_+$) the set of all faces with both orientations (resp. the positively oriented faces), and for any $f\in F_+$ we denote by $f^{-1}\in F$ the same face with reverse orientation.

\begin{rmk}
With our conventions, each unoriented edge and unoriented face is counted twice, therefore Euler's formula reads
\[
\vert V\vert -\tfrac12\vert E\vert + \tfrac12\vert F\vert = \vert V\vert - \vert E_+\vert + \vert F_+\vert = 2-2g,
\]
if $\Gbb$ is embedded in a surface of genus $g$. See Fig. \ref{Fig----Oriented faces} for an illustration.
\end{rmk}

\begin{figure}[!h]
\centering
\includegraphics[scale=0.8]{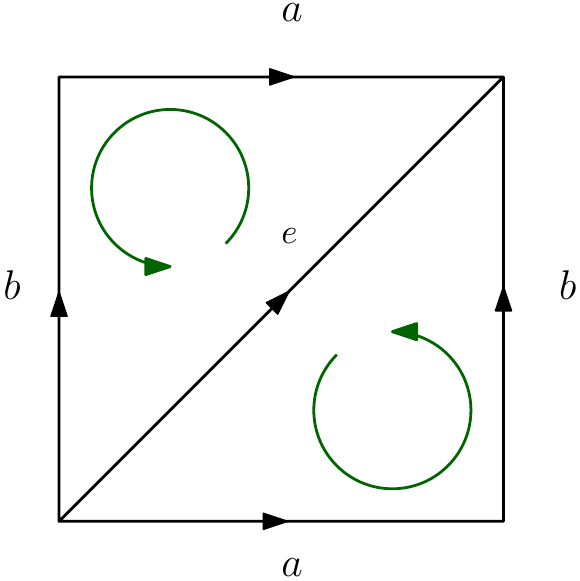}
\caption{\small A map embedded in the torus, with $\vert V\vert = 1$, $\vert E_+\vert = 3$ and $\vert F_+\vert = 2$. The edges named $a$ are glued together, and same for the edges named $b$. There are two positively-oriented faces, with respective boundaries $ea^{-1}b^{-1}$ and $abe^{-1}$; the orientations are represented by the counterclockwise green arrows. }\label{Fig----Oriented faces}
\end{figure}

From now on we will denote by $\Gbb=(V,E,F)$ a topological map, and $\theta$ and $\Sigma$ will be implied. Given $\Gbb=(V,E,F)$, one can describe a CW-complex corresponding to the map: its vertices are 0-cells, its edges 1-cells and its faces 2-cells. Besides, the skeleton of the map is exactly the skeleton of the complex. We will describe the corresponding chain and cochain complexes in the next section, as well as their (co)homology. In this section, we rather focus on topological features of maps, and algebraic properties of loops in a map.

\vspace{0.5em}

A map with boundary is a map $(V,E,F)$ together with a proper subset $B$ of $F$ such that the closures of 2-cells associated to any pair of distinct elements of $B$ do not intersect. A \emph{path} in $\Gbb$ is whether a single vertex or a finite string of edges $e_1\ldots e_n$ with $n\ge 1$ such that for all $k\in\{1,\ldots,n-2\},$ $\und{e}_{k+1}=\ov{e}_k.$ We say it is constant in the first case and set $|\g|=0$, while in the second, we denote by $\ov \g= \ov e_n$ and $\und \g= \und e_1$ its endpoint and starting point, and by $|\g|=n$ its \emph{length}. A loop of $\Gbb$ is a path $\g$ with $\und \g=\ov \g.$  A loop $\ell$ is \emph{based} at a vertex $v$ when $\und{\ell}=v.$ We say it is \emph{simple} when all vertices of $\ell$ occur only once $\ell$ but $\und{\ell}$ which occurs exactly twice. We write respectively $\mathrm{P}(\Gbb)$ and $\mathrm{L}(\Gbb)$ for the set of paths  and loops of $\Gbb.$ The respective sets of paths starting from a vertex $v\in V$ are denoted by $\mathrm{P}_v(\Gbb)$ and $\mathrm{L}_v(\Gbb).$  Whenever $\a$ and $\b$ are two paths with $\ov \a=\und \beta,$ $\a\b$ denotes their \emph{concatenation}, while $\a^{-1}$ is the path run in reverse direction, with the convention that $\g_1\a\g_2=\a$ when $\g_1$ and $\g_2$ are constant paths at $\und \a$ and $\ov \a$. We say that $\b$ is a \emph{subpath} of $\delta\in \mathrm{P}(\Gbb)$ and write $\b\prec \delta,$ if there are paths $\a$ and $\g$ with $\delta=\a\b \g.$

\vspace{0.5em}

\emph{Homeomorphic loops:} When two maps $\Gbb,\Gbb'$ yields  homeomorphic CW complexes, it induces a bijection between cells of same dimension. Denote by 
$\Phi: E\to E'$ the associated bijection between edges of $\Gbb$ and $\Gbb'$ and the associated bijection between $\mathrm{P}(\Gbb)$ and $\mathrm{P}(\Gbb').$ Consider two paths $\a$ and $\b$ within maps $\Gbb_\a$ and $\Gbb_\b$. We say that $\a$ and $\b$ are \emph{homeomorphic} and write $$\a\sim_\Sigma\b$$ if there are maps $\Gbb$ and $\Gbb'$ finer than respectively $\Gbb_\a$ and $\Gbb_\b$ such that  $\Gbb$ and $\Gbb'$ are homeomorphic, with induced bijection $\Phi:\mathrm{P}(\Gbb)\to \mathrm{P}(\Gbb')$ such that $$\Phi(\a)=\b.$$

\vspace{0.5em}

\emph{Cyclically equivalent loops:} We say that two loops  are \emph{cyclically equivalent} when one can be obtained from the other by cyclically permuting its 
edges. By convention, two constant loops are cyclically equivalent if they have equal base-point. This defines an equivalence relation $\sim_c$ on $\mathrm{L}(\Gbb)$. An element of the quotient $\mathrm{L}_c(\Gbb)=\mathrm{L}(\Gbb)/\sim_c$ is called an \emph{unrooted loop}.

\vspace{0.5em}

\emph{Reduced loops:}  A path $\g'$ is obtained by insertion of an edge in a path $\g$, if  $\g=\g_1\g_2$ and $\g'=\g_1ee^{-1}\g_2$ with  $\g_1,\g_2$ two subpaths of $\g$ and $e$ an edge such that $\ov {\g_1}=\und e=\und{\g_2}.$ Vice-versa, we say in this situation that $\g$ is obtained by erasing of an edge of $\g'$. Two paths are said to have the same \emph{reduction}  if  a finite sequence of erasures and insertions of edges transforms one into the other. This defines an equivalence relation $\sim_r$ on $\mathrm{P}(\Gbb)$ and we write $\mathrm{RP}(\Gbb)=\mathrm{P}(\Gbb)/\sim_r,$  $\mathrm{RP}_v(\Gbb)=\mathrm{P}_v(\Gbb)/\sim_r$ and $\mathrm{RL}_v(\Gbb)=\mathrm{L}_v(\Gbb)/\sim_r$ for any $v\in V$. The reduction of a path $\g\in\mathrm{P}(\Gbb)$ is the unique path of minimal length in its $\sim_r$-equivalence class. We say that two loops are $\sim_{r,c}$-equivalent if one can be obtained from the other by iterated cyclic permutations, insertions and erasures of edges.

\vspace{0.5em}

\emph{Lassos and discrete homotopy:}  For any face $f\in F$, its boundary can be identified with an unrooted loop  $\pl f.$ When $r\in V$ is a vertex of 
$\pl f,$ we write $\pl_r f$ for the loop in the $\sim_c$-class of $\pl f$ with $\und{\pl_r f}= r.$  When $F_*$ is a subset of $F,$ a $F_*$-\emph{lasso} is a loop 
of the form $\a \pl_r f \a^{-1},$ where $f$ is an oriented face belonging up to orientation to $F_*$ and  $\a\in \mathrm{P}(\Gbb)$ is a path such that $r=\ov \a$ is a vertex of $\pl f.$   When 
$\g\in \mathrm{P}(\Gbb)$, $\g'$ is obtained by \emph{lasso insertion}  from $\g$ if $\g=\g_1\g_2$ for some paths $\g_1,\g_2\in \mathrm{P}(\Gbb)$ and $\g'= \g_1 \ell \g_2,$ where $\ell$ is a lasso with 
$\ov{\g}_1=\und{\ell}=\und{\g}_2.$ 
Conversely, $\g'$ is said to be obtained from $\g$ by \emph{lasso erasure}. We say that two paths are \emph{discrete homotopic} if there is a finite sequence of lassos or edge erasures and insertions transforming one into the other. This defines an equivalence relation $\sim_h$ on $\mathrm{P}(\Gbb)$ which is also well defined on $\mathrm{RP}(\Gbb).$ Moreover, two paths of $\Gbb$ are discrete homotopic if and only if their image in $\Sigma_\Gbb$ are homotopic with fixed endpoints.  For any $v\in V,$ we denote  the quotient  $\mathrm{P}_v(\Gbb)/\sim_h$  and $\mathrm{L}_v(\Gbb)/\sim_h$ by $\tilde V_v$ and $\pi_{1,v}(\Gbb).$ When $F_*\subset F,$ we say that two paths of $\Gbb$ are \emph{$F_*$-homotopic} if there is a finite sequence of $F_*$-lassos or edge erasures and insertions transforming one into the other.  This defines an equivalence relation on $\mathrm{P}(\Gbb)$ denoted by $\sim_{F_*}.$ When $K$ is a closed, compact, contractible subset of $\Sigma_{\Gbb}$ given by the closure of the union of images of $F_*$, for any pair of paths $\g_1,\g_2\in\mathrm{P}(\Gbb)$ whose image in $\Sigma_{\Gbb}$ is included in $K$ and with  same endpoints, $\g_1\sim_{F_*}\g_2.$ 

\vspace{0.5em}

\emph{The group of reduced loops and the fundamental group:}  For any vertex $v\in V$, we define a group by endowing $\mathrm{RL}_v(\Gbb)$ with the multiplication given by concatenation and the inverse map given by  reversing the orientation of loops. The group $\pi_{1,v}(\Gbb)$ is the quotient of $\mathrm{RL}_v(\Gbb)$ by the normal subgroup generated by lassos based at $v.$ Since two loops of $\Gbb$ are discrete homotopic if and only if their image in $\Sigma_\Gbb$ are homotopic, the group $\pi_{1,v}(\Gbb)$ is isomorphic to the fundamental group of the surface $\Sigma_\Gbb$. For any group $G$, let us write $[a,b]=aba^{-1}b^{-1},\ \forall a,b\in G.$  Then  
$\pi_{1,v}(\Gbb)$ is isomorphic to the surface group
$$\Gamma_g=\< x_1,y_1,\ldots, x_g,y_g| [x_1,y_1]\ldots [x_g,y_g]\>.$$
We will also consider, for $r\geq 1$, the group
\[
\Gamma_{r,g}=\<z_1,\ldots,z_{r}, x_1,y_1,\ldots, x_g,y_g\ |\ z_1\ldots z_{r} = [x_1,y_1]\ldots [x_g,y_g]\>.
\]

\begin{lem}[\cite{Lev2}]\label{__Lem:Basis Reduced Loops}For any map $\Gbb$, the following assertions hold:
\begin{enumerate}
\item The group $\mathrm{RL}_v(\Gbb)$ is free of rank $\vert E_+\vert-\vert V\vert+1=\vert F_+\vert +2g-1.$ 
\item Assume that $g\ge 0$ and $\vert F_+\vert=r$. For any $v\in V,$ there are   lassos $(\ell_i,1\le i\le r)$ based at $v$, with faces in bijection with $F$, and  loops  $a_1,b_1,\ldots, a_g,b_g\in \mathrm{L}_v(\Gbb)$ such that the application
$$\Theta:\Gamma_{r,g}\to\mathrm{RL}_v(\Gbb)$$
that maps $z_i$ to $\ell_i$ for all $1\le i\le r$, $x_m$ (resp. $y_m$) to $a_m$ (resp. $b_m$) for all $1\le m\le g$ is an isomorphism\footnote{denoting here abusively the $\sim_r$ class of a loop by the same symbol as the loop.}. The diagram 
$$\begin{array}{cccc}1&\to \Gamma_{r,g}&\to \mathrm{RL}_v(\Gbb)&\to 1\\
& \downarrow&\downarrow&\\
1 &\hspace{-0.2 cm}\to \Gamma_{g} &\to \pi_{1,v}(\Gbb)&\to 1
\end{array}$$
is then commutative, where the left downwards morphism is the group morphism mapping $z$ to $1\in\Gamma_{r,g}$ for any $z\in \{z_1,\ldots,z_r\}$, and $t\in \Gamma_{r,g}$ to $t\in\Gamma_g$ for any $t\in\{x_1,y_1,\ldots, x_g,y_{g}\}.$
\end{enumerate}
\end{lem}

\vspace{0.5em}

\emph{Refining maps:} When $\Gbb'=(V',E',F')$ and $\Gbb=(V,E,F)$ are two maps, $\Gbb'$ is \emph{finer} than $\Gbb$ if  $(V,E)$ is a subgraph of $(V',E')$ and $\Sigma_{\Gbb'}=\Sigma_{\Gbb}$, so that we can identify $V$ and $E$ with subsets of respectively $V'$ and $\mathrm{P}(\Gbb),$ while  any face of $\Gbb$ is the union of faces of $\Gbb'.$  Conversely, we say that $\Gbb$ is \emph{coarser} then $\Gbb'$.

\vspace{0.5em}

\emph{Dual map:} When $\Gbb=(V,E,F)$ is a map of genus $g$ with surface $\Sigma_\Gbb,$ we define its \emph{dual map} as follows: we put a vertex $f^*$ inside each face $f\in F$, and for each edge $e\in E$ that separates two faces $f_1$ and $f_2$ we draw a new edge $e^*$ that intersects it in its midpoint and connects the vertices $f_1^*$ and $f_2^*$. There is a bijection $V^*\simeq F,$ $E^*\simeq E$ and $F^*\simeq V$ and a dual edge inherits the orientation from the edges it crosses as follows: if $e^*$ crosses $e\in E_+$ from the right\footnote{Formally, it means that the dual edge $(f_1,f_2)$ with $f_1,f_2\in F_+$ is in $E_+^*$ if the edge $e\in E_+$ it crosses satisfies $e\in\partial f_2$ and $e^{-1}\in\partial f_1$.}, then $e^*\in E_+^*$. See Fig. \ref{Fig----Dualmap} for an illustration. In particular, we see that if $e=(\und e,\ov e)$ is an edge and $e^*=(\und{e^*},\ov{e^*})$ is the dual edge, then we have the following facts:
\begin{equation}\label{eq:dual_edges}
e^*\in\partial \und{e},\ (e_*^{-1})\in\partial\ov{e},\ e\in\partial\ov{e^*},\ e^{-1}\in\partial\und{e^*}.
\end{equation}

\begin{figure}[!h]
\centering
\includegraphics[scale=0.8]{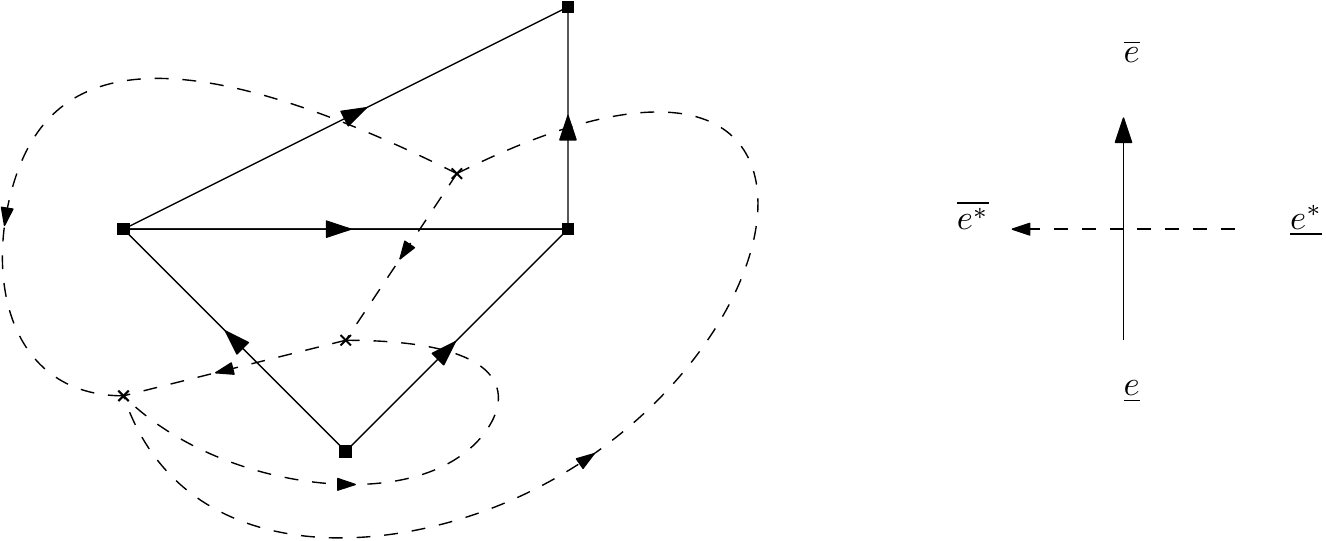}
\caption{\small On the left: a map embedded in the sphere (in plain lines), and its dual map (in dashed lines). Each dual edges are oriented such that it crosses the corresponding edge ``from the right''. On the right: the orientation convention of an edge and its dual. We have $\und e,\ov e\in V=F^*$ and $\und{e^*},\ov{e^*}\in F=V^*$.}\label{Fig----Dualmap}
\end{figure}

\vspace{0.5em}

{\emph{Cut of a map:} When $\Gbb=(V,E,F)$ is a map and $\ell$ is a simple loop of $\Gbb,$ with dual edges $E_\ell^*$, we say that $\ell$ is separating if the graph $(F,E^*\setminus E_\ell^*)$ has exactly two connected components $(F_1,E_1^*)$ and $(F_2,E_2^*)$. Consider $i\in\{1,2\}.$ Denote by  $E_i$ the union of $E_\ell$ with the set of edges of $\Gbb$ dual to  $E_i^*$, and by  $V_i$ the vertices of $\Gbb$ endpoints of edges in $E_i$.  We then define a map with one boundary component by setting $\Gbb_i=(V_i,E_i,F_i\sqcup\{f_{i,\infty}\})$ where $\{f_i,\infty\}$ is the label of a boundary face with boundary $\ell.$   We say that the pair of maps with boundary $(\Gbb_1,\{f_{1,\infty}\}),(\Gbb_2,\{f_{2,\infty}\})$ is the \emph{cut} of $\Gbb$ along $\ell.$ We say that the cut is \emph{essential} if $\ell$ is not contractible. A cut is essential if and only if  the maps $\Gbb_1$ 
and $\Gbb_2$ have genus larger or equal  to $1.$  When a map is cut, the lemma \ref{__Lem:Basis Reduced Loops} can  be specified as follows. 

\begin{lem}\label{__Lem:Basis Reduced Loops Cut} Assume that $(\Gbb_1,\{f_{1,\infty}\}),(\Gbb_2,\{f_{2,\infty}\})$ is the cut of a map $\Gbb=(V,E,F)$ of genus $g\ge 0,$ along a simple loop $\ell\in\mathrm{L}_v(\Gbb).$ 
Denote by $g_1$ and $g_2$ the genus of $\Gbb_1$ and $\Gbb_2$ and by  $r_1$ and $r_2$ their number of non-boundary faces, so that $\Gbb$ has genus $g=g_1+g_2$ and $r=r_1+r_2$ faces.  Then the following holds true.
\begin{enumerate}
\item The group $\mathrm{RL}_v(\Gbb)$ is isomorphic to the free product $\mathrm{RL}_v(\Gbb_1)*\mathrm{RL}_v(\Gbb_2).$
\item  There are    lassos $(\ell_i,1\le i\le r)$ based at $v$, and  loops  $a_1,b_1,\ldots, a_g,b_g\in \mathrm{L}_v(\Gbb)$ as in Lemma \ref{__Lem:Basis Reduced Loops Cut} with the additional property that 
$\ell_1,\ldots,\ell_{r_1},a_1,\ldots,b_{g_1}\in \mathrm{L}_v(\Gbb_1)$  and  $\ell_{r_1+1},\ldots,\ell_{r},a_{g_1+1},\ldots,b_{g}\in \mathrm{L}_v(\Gbb_2).$  The group $\mathrm{RL}_v(\Gbb_1)$ is then free over the basis $\ell_1,\ldots,\ell_{r_1},a_1,\ldots,b_{g_1}.$
\end{enumerate}

\end{lem}

}

\subsection{Discrete homology, winding function and Makeenko--Migdal vectors}

\label{----sec: Makeenko--Migdal Vectors}

We recall here some elementary results about the homology of topological maps and discuss their relation to Makeenko--Migdal vectors introduced in \cite{Lev,DN,Hal2}. It will lead us to a construction of the winding function, as well as a characterisation of the Makeenko--Migdal vectors, which, as we recall, encode the deformations that are allowed by Makeenko--Migdal equations. In the sequel, $R$ will denote a ring that is either $\R$ or $\Z$, unless specified otherwise. We shall start with a general property of finitely-generated free modules.

\begin{prop}\label{__prop:dual_mod}
Let $A$ be a finitely-generated free $R$-module, and $(e_1,\ldots,e_n)$ be a free basis of $A$.
\begin{enumerate}
\item There exists a nondegenerate bilinear form $\langle \cdot,\cdot\rangle$ on $A$ such that $(e_1,\ldots,e_n)$ is an orthonormal basis.
\item There is a canonical isomorphism $A\cong\Hom(A,R)$ expressed through the bilinear form $\langle\cdot,\cdot\rangle$.
\end{enumerate}
\end{prop}

\begin{proof}
The first point is obvious: we set $\langle e_i,e_j\rangle=\delta_{ij}$ and we extend the form by bilinearity. For the second point, set
\[
\Phi:\left\lbrace\begin{array}{ccc}
A & \longrightarrow & \Hom(A,R)\\
x & \longmapsto & (y\mapsto \langle x,y\rangle),
\end{array}\right.
\]
and notice that it is indeed an isomorphism.
\end{proof}

We can present the homology of a topological map from several equivalent ways, and we will present three of them.

\begin{dfn}\label{def:chain_cpx}
Let $\Gbb=(V,E,F)$ be a topological map. We define its associated \emph{chain complex} by the sequence
\[
0\longrightarrow C_2(\Gbb;R)\build{\longrightarrow}{}{\partial} C_1(\Gbb;R)\build{\longrightarrow}{}{\partial} C_0(\Gbb;R)\longrightarrow 0,
\]
where $C_0(\Gbb;R)$ (resp. $C_1(\Gbb;R)$, $C_2(\Gbb;R)$) is the free $R$-module generated\footnote{Remark that $E$ and $F$ define generating families but not free families. A free basis of $C_1(\Gbb,R)$ (resp. $C_2(\Gbb;R)$) is given by $E_+$ (resp. $F_+$).} by $V$ (resp. $E$, $F$). The boundary operator is defined by linear extension of the boundary operator in the underlying surface, defined by
\begin{align*}
\partial e = & \ov e - \und e,\ \forall e\in E,\\
\partial f = & \sum_{e\in\partial f} e,\ \forall f\in F.
\end{align*}
\end{dfn}

Let $\Gbb=(V,E,F)$ be a topological map, and let $\Gbb^*=(V^*,E^*,F^*)$ be its dual map. For any $v\in V=F^*$ we define its \emph{boundary} $\partial v$ as the cycle $e_1^*\cdots e_n^*$ of dual edges constituting the positively-oriented boundary of the face $v$.

\begin{dfn}
Let $\Gbb=(V,E,F)$ be a topological map. Its associated \emph{cochain complex} is defined by the sequence
\[
0\longleftarrow \Omega^2(\Gbb,R)\build{\longleftarrow}{}{d} \Omega^1(\Gbb,R)\build{\longleftarrow}{}{d} \Omega^0(\Gbb,R)\longleftarrow 0,
\]
where $\Omega^k(\Gbb,R)=\Hom(C_k(\Gbb;R),R)$ for any $0\leq k\leq 2$, and $d$ is the dual of the boundary operator:
\begin{align*}
df(e) = & f(\partial e) = f(\ov e) - f(\und e), \ \forall e\in E, \ \forall f\in\Omega^0(\Gbb,R),\\
d\omega(f) = & \omega(\partial f) = \sum_{e\in\partial f} \omega(e),\ \forall f\in F,\ \forall \omega\in\Omega^1(\Gbb,R).
\end{align*}
The elements of $\Omega^k(\Gbb,R)$ are called $R$-valued $k$-\emph{forms} on $\Gbb$.
\end{dfn}

\begin{prop}\label{prop:dual_complex}
For any topological map $\Gbb$, its cochain complex is isomorphic to the chain complex of the dual map $\Gbb^*$.
\end{prop}

\begin{proof}
Let us first note that, as free modules, we have indeed canonical isomorphisms $\Phi:C_k(\Gbb^*;R)\cong\Omega^{2-k}(\Gbb,R)$ for $0\leq k\leq 2$. They are explicitely given as follows:
\[
f_v:v'\in V\mapsto \delta_{vv'},\ \forall v\in F_+^*\simeq V,
\]
\[
\omega_e:e'\in E_+\mapsto \delta_{ee'},\ \forall e^*\in E_+^*\simeq E_+,
\]
\[
\mu_f:f'\in F_+\mapsto \delta_{ff'},\ \forall f\in V^*\simeq F_+.
\]
Using \eqref{eq:dual_edges} in conjunction with the definitions of $\partial$ and $d$, one can easily find that the diagram
\[
\begin{tikzcd}
0 \arrow[r] & C_2(\Gbb^*;R) \arrow[r, "\partial"] \arrow[d, equal] & C_1(\Gbb^*;R) \arrow[r, "\partial"] \arrow[d, equal] & C_0(\Gbb^*;R) \arrow[r] \arrow[d, equal] & 0\\
0 \arrow[r] & \Omega^0(\Gbb,R) \arrow[r, "d"] & \Omega^1(\Gbb,R) \arrow[r, "d"] & \Omega^2(\Gbb,R) \arrow[r] & 0
\end{tikzcd}
\]
commutes, which proves the isomorphism.
\end{proof}

We shall denote by $\mu_*$ the constant 2-form defined by $\mu_*(f)=1$ for all $f\in F_+$. Thanks to Proposition \ref{__prop:dual_mod}, there is for any $0\leq k\leq 2$ a canonical isomorphism $\Phi:C_k(\Gbb;R)\build{\longrightarrow}{}{\cong}\Omega^k(\Gbb;R)$, represented by the applications $v\mapsto f_v$, $e\mapsto\omega_e$ and $f\mapsto\mu_f$ used in the proof of Proposition \ref{prop:dual_complex}.

\vspace{0.5em}

We define $d^*$ as the adjoint of $d$ on the cochain complex of $\Gbb$, meaning that
\begin{align*}
d^*\omega = & \sum_{e\in E_+}\omega(e)f_{\partial e}=\sum_{e\in E_+}\sum_{v\in \partial e} \omega(e)f_v,\ \forall \omega\in\Omega^1(\Gbb,R),\\
d^*\mu = & \sum_{f\in F_+} \mu(f)\omega_{\partial f} = \sum_{f\in F_+}\sum_{e\in\partial f}\mu(f)\omega_e,\ \forall \mu\in\Omega^2(\Gbb,R)\cong\Omega.
\end{align*}
In particular, $d^*:\Omega^1(\Gbb,R)\to\Omega^0(\Gbb,R)$ is the \emph{divergence operator} in R. Kenyon's terminology \cite{Ken}. Let $e\in E_+$ be an oriented edge, and $e^*\in E_+^*$ be the dual edge, \emph{i.e.} the faces $f,f'\in F_+$ such that $e\in\partial f$ and $e^{-1}\in\partial f$ satisfy $f'=\und{e^*}$ and $f=\ov{e^*}$. Then we have, for any $\mu\in\Omega^2(\Gbb,R)\cong C_0(\Gbb^*;R)$:
\[
d^*\mu(e) = \mu(f)-\mu(f') = \langle \mu,\partial e^*\rangle.
\]

We obtain an isomorphism of chain complexes given by the following commutative diagram:

\[
\begin{tikzcd}
0 \arrow[r] & C_2(\Gbb;R) \arrow[r, "\partial"] \arrow[d, equal] & C_1(\Gbb;R) \arrow[r, "\partial"] \arrow[d, equal] & C_0(\Gbb;R) \arrow[r] \arrow[d, equal] & 0\\
0 \arrow[r] & \Omega^2(\Gbb,R) \arrow[r, "d^*"] & \Omega^1(\Gbb,R) \arrow[r, "d^*"] & \Omega^0(\Gbb,R) \arrow[r] & 0
\end{tikzcd}
\]
Equipped with these chain complexes, we can then do some (basic) homology. Let us introduce a few notations: we set:
\begin{itemize}
\item $\diamondsuit_1=\ker(d^*:\Omega^1(\Gbb,R))\simeq\ker(\partial:C_1(\Gbb;R)\to C_2(\Gbb;R)$ the module of cycles,
\item $\bigstar_1^*=d^*(\Omega^2(\Gbb,R))\simeq\partial(C_2(\Gbb;R))$ the module of boundaries,
\item $\bigstar_1=d(\Omega^0(\Gbb,R))\simeq\partial(C_2(\Gbb^*;R))$ the module of coboundaries.
\end{itemize}

\begin{dfn}
Let $\Gbb$ be a topological map. Its \emph{first homology module} is defined as the $R$-module
\begin{align*}
H_1(\Gbb;R) = \diamondsuit_1/\bigstar_1^*.
\end{align*}
When $\ell$ is a loop of $\Gbb$,  its $R$-homology $[\ell]_R$  is the  image of the element $\omega_\ell$ in $H_1(\Gbb;R).$ For any $n\ge 2, $  its $\Z_n$-homology $[\ell]_{\Z_n}$ is the  element $1\otimes [\ell]_\Z\in H_1(\Gbb;\Z_n)=\Z_n\otimes_\Z H_1(\Gbb;\Z).$
\end{dfn}

Note that by the universal coefficient theorem for homology, the change of ring commutes with the homology: 
\[
H_1(\Gbb;R)=R\otimes_\Z H_1(\Gbb;\Z),
\]
even if we take $R=\Z_n$.

\begin{prop}\label{prop:decomp_forms}
Denote by $\diamondsuit_0$ the $R$-module spanned by $\omega_\ell$ for all loops $\ell$ in $\Gbb$.
\begin{enumerate}
\item We have the following equality of $R$-modules:
\[
\diamondsuit_0 =\diamondsuit_1=\bigstar_1^\perp.
\]
\item We have the following direct sum decomposition into orthogonal subspaces:
\begin{equation}
\Omega^1(\Gbb,R) = \bigstar_1\oplus\diamondsuit_1.
\end{equation}
\end{enumerate}
\end{prop}

\begin{proof}
Let us start by showing that $\diamondsuit_1=\bigstar_1^\perp$. If $\omega\in\diamondsuit_1$, then for any $f\in\Omega^0(\Gbb,R)$ we have
\[
\langle \omega, df\rangle = \langle d^*\omega, f\rangle = 0
\]
and $\omega\in\bigstar_1^\perp$. Conversely, remark that a free basis of $\bigstar_1$ is given by $(df_v,v\in V)$, so that for any $\omega\in\bigstar_1^\perp$ we have
\[
\langle d^*\omega, f_v\rangle = \langle \omega,df_v\rangle = 0,
\]
and $\omega\in\diamondsuit_1$.

Now let us prove that $\diamondsuit_0=\diamondsuit_1$. If $\ell=e_1\cdots e_n$ is a loop in $\Gbb$, then
\[
d^*\omega_\ell = \sum_{i=1}^n\sum_{e\in E_+}\omega_{e_i}(e)f_{\partial e} =\sum_{i=1}^nf_{\partial e_i} = 0,
\]
so that $\diamondsuit_0\subset\diamondsuit_1$. Let $\omega\in\diamondsuit_0^\perp$ be a 1-form, we define a 0-form $f_\omega\in\Omega^0(\Gbb,R)$ by setting
\[
f_\omega(v) = \sum_{i=1}^n \omega(e_i),
\]
where $e_1\cdots e_n$ is a path in $\Gbb$ starting from a given reference vertex $v_0$ and ending at $v$. The fact that it does not depend on the path follows from the fact that $\omega\perp\omega_\ell$ for any loop $\ell$. We see that for any $e\in E$, $df_\omega(e)=\omega(e)$, therefore we have the inclusion $\diamondsuit_0^\perp\subset \bigstar_1=\diamondsuit_1^\perp$. We finally get that $\diamondsuit_0=\diamondsuit_1$. The direct sum decomposition follows from the standard decomposition of a module into a submodule and its orthogonal, provided that the bilinear form is not degenerate on this submodule, which is trivially the case here.
\end{proof}

\begin{prop}\label{__prop:isom_homology}
The $R$-module
\[
\mathcal{H}_1=(\bigstar_1^*)^\perp\cap\diamondsuit_1
\]
is isomorphic to $H_1(\Gbb;R)$.
\end{prop}

\begin{proof}
Recall that $\bigstar_1^*\subset\diamondsuit_1$ thanks to the property of the chain complex. It follows from the direct sum decomposition
\[
\diamondsuit_1 = \bigstar_1^*\oplus\mathcal{H}_1
\]
that for any $\omega\in\diamondsuit_1$ there is a unique couple $(\omega_0,\mu)\in\mathcal{H}_1\times \Omega^2(\Gbb,R)$ such that $\omega=\omega_1+d^*\mu$. It is then straightforward to check that the map $[\omega]\mapsto \omega_0$ is the isomorphism that we were looking for.
\end{proof}

The winding number of a loop $\ell=e_1\cdots e_n$ around a point is an integer that counts how many times the loop cycles around the point; in particular we see that in the case of a topological map it defines a function $n_\ell\in\Omega^2(\Gbb,\Z)$ that counts how many times the loop cycles around each face. One can see that it is equivalent to require that $d^*n_\ell (e_i) = 1$ for any $i$ such that $e_i\in E_+$, and $-1$ for any $i$ such that $e_i^{-1}\in E_+$. It sums up as
\[
d^*n_\ell = \omega_\ell.
\]
Is it possible to get such a construction for compact orientable surfaces? The general answer is \emph{not exactly}, because ``bad'' things can happen when $\ell$ has a nontrivial homology, but it is still possible when $[\ell]=0$, as stated by the following lemma.

\begin{lem}\label{__Lem: Homology basis choice} Assume that $\Gbb$ is  embedded in an orientable surface of genus $g.$
\begin{enumerate}
\item  $H_1(\Gbb;R)$ is free of rank $2g$ and there are   $2g$ simple loops $a_1,b_1,\ldots, a_g,
b_g$ of $\Gbb$ such that $[a_1]_R,[b_1]_R,\ldots,[a_g]_R,[b_g]_R$ is a free basis of $H_1(\Gbb;R)$. Equivalently, $\omega_{a_1},\omega_{b_1},\ldots,\omega_{a_g},\omega_{b_g}$ is a free basis of $\mathcal{H}_1$.
\item When  $g\ge 1$  and   $v\in V,$ $(\ell_i,1\le i\le r)$ and  $a_1,b_1,\ldots, a_g,b_g\in \mathrm{L}_v(\Gbb)$ are as in Lemma \ref{__Lem:Basis Reduced Loops}, the map 
\[
\Gamma_{g}=\<x_1,y_1,\ldots, x_g,y_g\ | \ [x_1,y_1]\ldots [x_g,y_g]\> \to H_1(\Gbb;\Z)
\]
that maps $x_m$ to $[a_m]_\Z$ and $y_m$ to $[b_m]_\Z$ is a well defined, onto morphism, with kernel given by the commutator group $[\Gamma_g,\Gamma_g]$. 
\item For any loop $\ell$ of $\Gbb$ such that $[\ell]_R=0$, there is a unique $n_\ell\in \Omega^2(\Gbb,R)$ such that
\[
\omega_\ell= d^*n_\ell.
\]
We call the 2-form $n_\ell$ the winding function of $\ell$, and we shall identify it to an element of  $\{\mu_*\}^\bot.$  
\end{enumerate}
\end{lem}
\begin{proof}
The points $1.$ and $2.$ are standard results, and their proof can be found in Chapter 2 of \cite{Sti}. We shall prove the last point. Recall that
\[
\diamondsuit_1 = \bigstar_1^*\oplus\mathcal{H}_1,
\]
and that for any loop $\ell$ the 1-form $\omega_\ell$ is in $\diamondsuit_1$. Hence, there is a unique pair $(h_\ell,n_\ell)$ with $h_\ell\in\mathcal{H}_1$ and $n_\ell\in\Omega^2(\Gbb,R)$ such that $\omega_\ell=h_\ell+d^*n_\ell$. If $[\ell]_R=0$, it means that $h_\ell=0$ and $\omega_\ell=d^*n_\ell$ as expected.
\end{proof}

\begin{figure}[!h]
\centering
\includegraphics[scale=0.8]{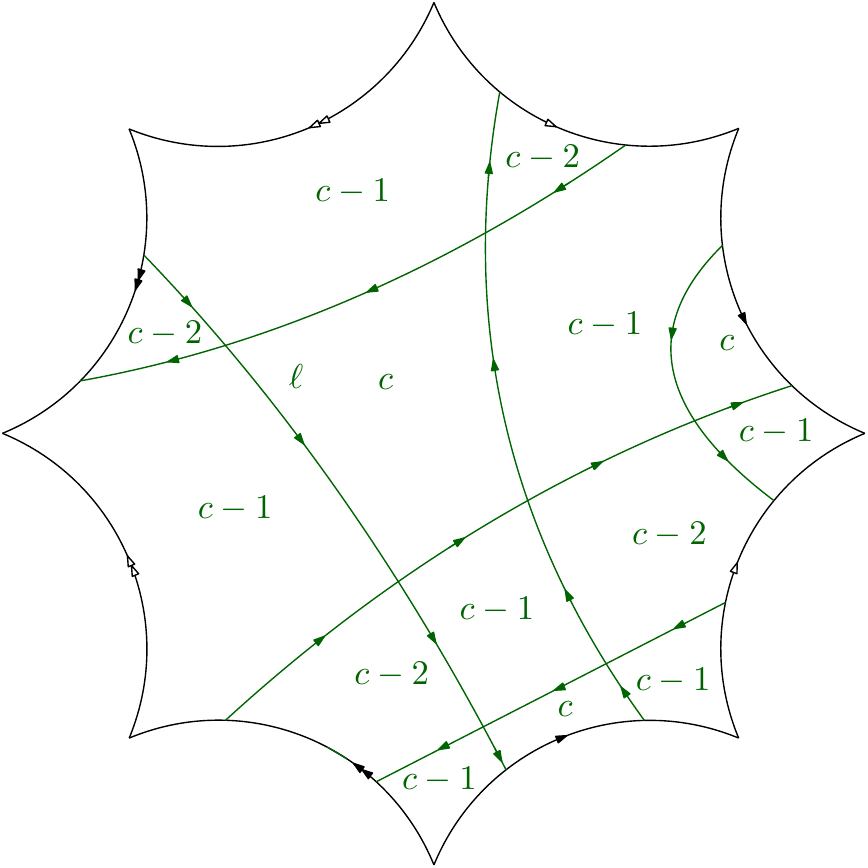}
\caption{\small A representant of the winding number function with $c\in R$, for a loop $\ell$ of null homology, on a map of genus $2$. The loop is drawn 
in green and the value on each 
positively oriented face is displayed on each $2$-cell.}\label{Fig----Winding Number Def}
\end{figure}

Let $\ell$ be a loop of a topological map $\Gbb=(V,E,F)$ which uses each non-oriented edge at most once and each vertex at most twice. We denote by $E_\ell$ the subset of edges $e\in E$ such that $\ell$ runs through $e$ or $e^{-1}$.

\vspace{0.5em}

Whenever a vertex $v$ is visited twice, the four outgoing edges at $v$ visited by $\ell$ can be ordered $e_1,e_2,e_3,e_4$ respecting the counterclockwise, cyclic ordering of the orientation of the map, so that $\ell$ is cyclically equivalent to a loop of the form $\a e_1^{-1}e_3\b e_2^{-1}e_4 \g$, $\a e_1^{-1}e_4\b e_3^{-1}e_2 \g$, $\a e_1^{-1}e_4^{-1}\b e_2^{-1}e_3 \g$ or $\a e_1^{-1}e_4\b e_2^{-1}e_3 \g$, these four cases being exclusive. See Figure \ref{Fig----Intersection Types}. We say that $\ell$ is a \emph{tame loop} if only the first case occurs. 
The set $V_\ell$  of  vertices  visited twice by $\ell$ are then called the \emph{(transverse) intersection points} of $\ell.$

\begin{figure}[!h]
\centering 
\includegraphics[scale=0.6]{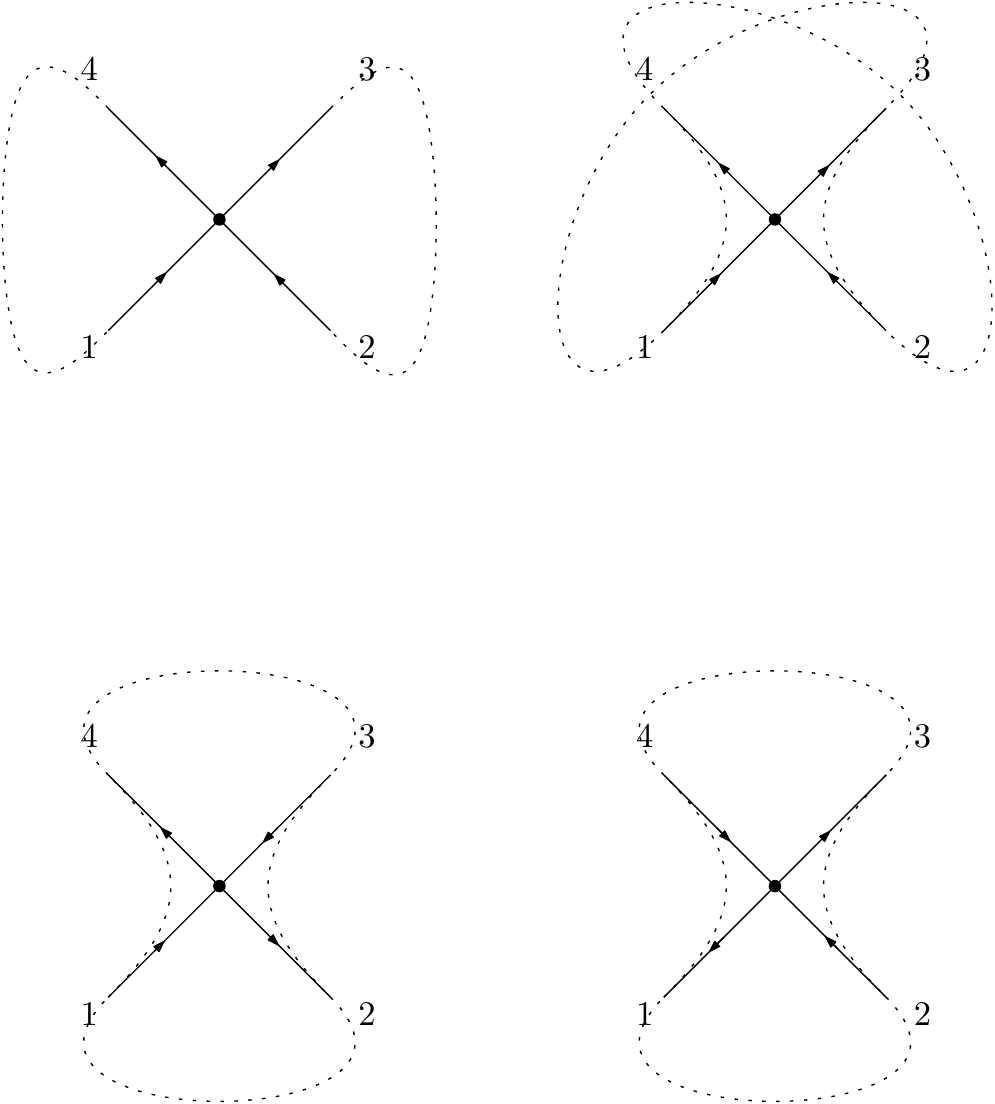}
\caption{\small The four types of transverse simple intersections. }\label{Fig----Intersection Types}
\end{figure}

\begin{dfn}
Let $\ell$ be a tame loop in a map $\Gbb$. The Makeenko--Migdal vector at $v\in V_\ell$ is the 2-form
\begin{equation}
\mu_v=d(\omega_{e_1} )+ d(\omega_{e_3})=-d(\omega_{e_2})-d(\omega_{e_4}). \label{eq:MM vec}
\end{equation}
We denote by $\mfm_\ell$ the $\R$-vector space generated by $\{\mu_v, v\in V_\ell\}$ and $\{d\omega_e, e\notin E_\ell\}$.
\end{dfn}
The Makeenko--Migdal vectors are an algebraic representation of the Makeenko--Migdal deformations described in Section \ref{-----sec:Strategy}, see in particular Fig. \ref{Fig----MMDef}.

\begin{figure}[!h]
\centering 
\includegraphics[scale=0.6]{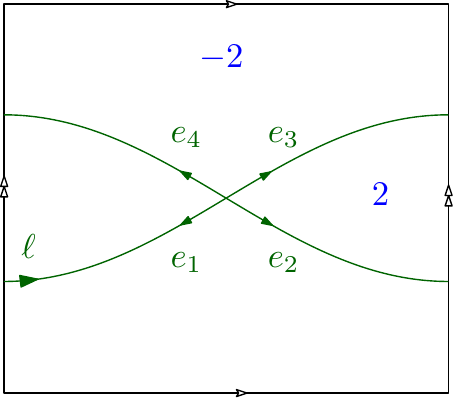}
\caption{\small  A  tame loop in a graph with one vertex and $2$ faces. The value of $\mu_v$ is displayed on each face in blue. }\label{Fig----MM Two 
Faces}
\end{figure}

\label{--------sec: Regular Graph}

\begin{lem} \label{__Lem: Charac MM}Let $\ell$  be a tame loop of a  map $\Gbb$. Then
$$\mfm_\ell=\left\{\begin{array}{ll}  \{\a\in \Omega^2(\Gbb,R): \<\a,\mu_*\> =0 \}& \text{ if }[\ell]_R\not=0, \\&\\\{\a\in \Omega^2(\Gbb,R): \<\a,\mu_*\> = \<\a, n_\ell 
\>=0 \}& \text{ if }[\ell]_R=0.\end{array}\right.$$
\end{lem}

\begin{rmk}
The signification of the conditions given in the characterisation of $\mfm_\ell$ is the following:
$\langle\a,\mu_*\rangle=0$ means that the deformation corresponding to $\a$ preserves the total area of the graph, whereas $\langle\a,n_\ell\rangle=0$ means that the deformation preserves the algebraic area of the graph, which corresponds to multiplying the area of each face by its winding number (with respect to the loop $\ell$).
\end{rmk}

\begin{proof}[Proof of Lemma \ref{__Lem: Charac MM}] Let us first remark that the above construction is invariant by the following appropriate subdivisions. Let us call subdivision of an oriented face $f_*$, the operation of adding two new vertices on its boundary and adding an edge $e$ connecting them; the new map $\mathbb{G}'$ has $2$ new vertices, $1$ more edge and $1$ more face, with in place of $f_*$, two faces $f_1$ and $f_2$ with the same orientation induced from $f_*$, while any other face is identified with a face of $\Gbb$. The map $\Gbb'$ being finer than $\Gbb$, $\ell$ can be identified with a tame loop of $\Gbb'$ that we denote by the same letter. Consider the map $P:\Omega^2(\Gbb',R)\to\Omega^2(\Gbb,R)$ with $P(\varphi)(f)=\varphi(f')$ whenever a face $f$ of $\Gbb$ is identified with a face of $\Gbb'$ and $\varphi(f_1)+\varphi(f_2)$ when $f=f_*.$ On one hand, $P(d\omega_e)=0$ and $P$ maps all other vectors of the defining generating family of $\mfm_\ell'$ to the generating family of $\mfm_\ell.$ Therefore,  $P(\mfm_\ell')=\mfm_\ell$. As $P: \{d\omega_e\}^\bot\to \Omega^2(\Gbb,R)$ is an isometry, while $d\omega_e\in \mfm_\ell'\cap \ker(P),$ $P({\mfm_\ell'}^\bot)=\mfm_\ell^\bot.$ On the other hand, $P(\mu_*')=\mu_*$ and when $[\ell]_R=0,$ $P(n_{\ell}')=n_\ell.$ We conclude that it is enough to prove the claim for any subdivision of $\Gbb.$

\vspace{0.5em}

We can then w.l.o.g. assume that $\ell$ and the paths $a_1,b_1,\ldots, a_g,b_g$ of Lemma \ref{__Lem: Homology basis choice} do not share any edge in common. Under this assumption, let us set $\Scr=\{\ell,a_1,b_1,\ldots, a_g,b_g\},$ denote by $\mathrm{T}(\Scr)$ the set of oriented edges $e$ such that an element of $\Scr$ runs through $e$ or $e^{-1}.$   Let $\eta$ be the permutation of the edges $E$ such that $
\eta(e^{-1})=\eta(e)^{-1}$ for any edge $e\in E$, with $2+4g$ non-trivial cycles  associated to elements of $\mathscr{S}$ forgetting the base point.  More precisely, for each $\g\in \{\ell,a_1,b_1,\ldots, a_g,b_g\}$ with $\g=e_1\ldots e_n$,  $(e_1,\ldots,e_n)$ and $(e_1^{-1},
\ldots, e_n^{-1})$ are cycles of $\eta,$ whereas $\eta(e)=e$ for any $e\not\in \mathrm{T}(\Scr).$ For any $\omega\in \Omega^1(\Gbb,R),$ setting
$$\eta.\omega=\omega \circ\eta^{-1}$$
defines a 1-form. We claim that for any oriented edge $e\in \mathrm{T}(\Scr)$, 
$$\a_e=d\omega_e- d (\eta.\omega_e) \in \mfm_\ell. $$
Indeed, it is non-zero only when $\g\in \Scr$ runs through $e$ or $e^{-1}$, in which case, it follows from \eqref{eq:MM vec} that $\a_e$ is a Makeenko--Migdal vector at respectively $\ov e$  or $
\und e$.

\vspace{0.5em}

Let us now consider  $\beta\in \mfm_\ell^\bot\cap\{\mu_*\}^\bot.$ Then 
\[\<\beta,\a_e\>= \<\b,(d-d\circ\eta)\omega_e\>=0,\ \forall e\in \mathrm{T}(\Scr),
\]
whereas $\<\beta, d\omega_e\>=0,\ \forall e\not\in \mathrm{T}(\Scr)$, so that
\[
d^*\b= (d\circ\eta)^*(\beta)=\eta^{-1}\circ d^*\beta \text{ and }\<d^*\beta,\omega_e\>=0,\ \forall e\not\in \mathrm{T}(\Scr).
\]
It follows that 
\[
d^*\beta= c\omega_\ell+\sum_{i=1}^g( a_i\omega_{a_i}+b_i\omega_{b_i}), \text{ for some }c,a_1,b_1,\ldots,a_g,b_g\in \R.
\]
Using the decomposition $\diamondsuit_1= \bigstar_1^*\oplus \cH_1,$ we find
$$d^*\beta= cd^* n_\ell  \text{ and }  c h_\ell+\sum_{i=1}^g( a_i\omega_{a_i}+b_i\omega_{b_i})=0.$$

Since $\beta \in \hat{\mfm}_\ell^\bot,$  for any edge  $e$ such that $e,e^{-1}$ do not belong to $\ell,$ $\<d^*\beta,\omega_e \>=0.$ In 
particular, $a_i=b_i=0$ for all $i$ and $d^*\beta = c\omega_\ell.$ Since $\beta\in \mu_*^\bot,$ it follows that whether $[\ell]_R=0$ and $\beta=cn_\ell $ or 
$c=0$ and $\beta =0.$ 
We conclude  that whether $[\ell]_R=0$ and $\mfm_\ell^\bot\cap \{\mu_*\}^\bot =\R. n_\ell$, or $[\ell]_R\not=0$ and $\mfm_\ell^\bot\cap \{\mu_*\}^\bot= \{0\}.$
\end{proof}

\subsection{Regular polygon tilings of the universal cover, tiling-length of a tame loop and geodesic loops}

\label{----sec:DiscreteFundamentalCover}

To simplify the presentation, we shall work only with surfaces of genus $g$ obtained by a standard quotient of $4g$ polygons.  We fix here notations and definitions relative to the universal cover of such maps. We refer to \cite{BiS} for more details. 

\vspace{0.5em}

\emph{Regular maps and regular loops:} A $2g$-\emph{bouquet map} is a map $(V,E,F)$ with $1$ vertex $v$, $1$ face and $2g$ edges, so that for $f\in F$, there are $2g$ oriented edges $a_1,b_1,\ldots, a_g,b_g\in E$ 
corresponding to distinct edges, with $\pl_v f=[a_1,b_1]\ldots [a_g,b_g].$ A $2g$-bouquet map can be obtained by labelling the edges of a $4g$-polygon counterclockwise $e_1,\ldots,e_{4g}$ and gluing  $e_{i+4k}$ to $e_{i+4k+1}$ 
for all $0\le k\le g-1$ $i\in\{1,2\}.$ A \emph{regular map} is a pair given by a map $\Gbb=(V,E,F)$ and  a $2g$-bouquet map $\Gbb_g$, such that $\Gbb$ is finer than $\Gbb_g$. Each edge of $\Gbb_g$ is uniquely decomposed as a 
concatenation of edges of $\Gbb.$ Let $\partial E\subset E $ be the set of edges appearing in these concatenations. We then denote by $\pl V$ the set of endpoints of edges of $\pl E$ and $\mathring{V}=V\setminus \pl V.$ 
When $(\Gbb,\Gbb_g)$ is a regular map, we refine the notion of tame loops defined in the previous section as follows. A \emph{loop} $\ell\in \mathrm{L}(\Gbb)$ is \emph{regular} whenever it  is tame,  none of its edges belong to 
$\pl E$ and $\und{\ell}\in \mathring{V}$. In particular its intersection points satisfy $V_\ell\subset \mathring{V}.$

\vspace{0.5em}

\emph{Universal cover of a regular map:} Let $(\Gbb,\Gbb_g)$  be a regular map with $\Gbb=(V,E,F).$  When $g=1,$ consider the closed 
square $P_1$ with vertices coordinates in $\{-\frac{1}{2},\frac 1 2\}$ and the tiling of $\R^2$ by translation of $P_1$ by $\Z^2.$ When $g\ge 2,$ consider a tiling of the Poincar\'e hyperbolic disc $\mathbb{H}$ by a family of closed regular $4g$-polygons  of $\mathbb{H}$ whose sides do not intersect $0$ and denote by $P_1$ the  polygon among them 
enclosing $0.$  The group $\Gamma_g$ can be identified with $\Z^2$ when $g=1$ and with a subgroup of M\"obius transformations that acts properly by isometry on $\mathbb{H}$ when $g\geq 2$. The group $\Gamma_g$ acts freely on the set of tiles and for each $h\in \Gamma_g,$ there is a unique tile $P_h$ with $h\cdot 0$ belonging to the interior of $P_h.$  Let us define  $\tilde \Sigma_{\Gbb}$  as  $\R^2$  when $g=1$ and $\mathbb{H}$ when $g\ge 2.$  The quotient of $\tilde \Sigma_\Gbb$ by $\Gamma_g$ is homeomorphic to $
\Sigma_\Gbb$ and we denote by $p:\tilde \Sigma_\Gbb\to \Sigma_\Gbb$ the quotient mapping. There is a unique CW decomposition of $\tilde \Sigma_{\Gbb}$  such that the restriction of $p$ to the interior of
each cell of $\tilde \Sigma_\Gbb$ is an homeomorphism onto the interior of a cell of $\Sigma_\Gbb$ labeled by an element of  $V,E$ or $F$.  We denote a labelling of the cells of this CW complex by $\tilde\Gbb=(\tilde V,\tilde E,\tilde F)$ and call $\tilde 
\Gbb$ a \emph{universal cover} of $\Gbb.$ There is a natural map from $\tilde V,\tilde E,\tilde F$ to respectively $V,E$ and $F$ that we also denote by $p.$ The map $\tilde \Gbb$ is finer than the universal cover $\tilde\Gbb_g=(\tilde V_g,\tilde E_g,\tilde F_g)$ of $\Gbb_g,$ where faces  $\tilde F_g$ can be identified with polygons $(P_g)_{g\in\Gamma_g},$ and $\Gamma_g$ acts free transitively  on $\tilde V_g.$
As for maps, the pair $(\tilde V, \tilde E)$ can be identified with a graph, and we denote by $\mathrm{P}
(\tilde \Gbb)$ its set of paths.  For each path $\g=e_1\ldots e_n\in \mathrm{P}(\Gbb)$ and $\tilde v\in p^{-1}(v_0),$ the lift of $\g$ from $\tilde v$ is the unique path $\tilde \g=\tilde e_1\ldots, \tilde e_n\in \mathrm{P}(\tilde\Gbb)$ with 
$\underline{\tilde \g}=\tilde v$ and $p(\tilde e_k)=e_k$ for 
all $1\le k\le n.$ Vice-versa, when $\tilde\gamma=(\tilde e_1,\ldots,\tilde e_n)\in \mathrm{P}(\tilde\Gbb),$ its projection is the path $p(\gamma)=(p(\tilde e_1),\ldots,p(\tilde e_n))\in\mathrm{P}(\Gbb).$  Its image in 
$\mathrm{RP}(\Gbb)$ does not depend on the $\sim_r$ equivalence class $[\gamma]$ of $\gamma;$ we denote it by  $p([\gamma])\in \mathrm{RP}(\Gbb).$  When $\tilde v\in \tilde V$ and $v=p(\tilde v),$ the group $\mathrm{RL}_{\tilde v}(\tilde \Gbb) $ of reduced loop of $(\tilde V,\tilde E)$ based at $\tilde v$ allows to complete the diagram of Lemma \ref{__Lem:Basis Reduced Loops}  in the following way. The proof is standard and left to the reader. 

\begin{lem}\label{__Lem:Basis Reduced Loops lift} Let   $(\Gbb,\Gbb_g)$ be a regular map, the following assertions hold:
\begin{enumerate}
\item The sequence 
$$\begin{array}{cccc}1&\to \mathrm{RL}_{\tilde v}(\tilde \Gbb)&\overset{p}{\to} \mathrm{RL}_v(\Gbb)&\to \pi_{1,v}(\Gbb)\to1
\end{array}$$
is a  short exact sequence.
\item  Denote by $\Gamma_c$ the kernel of the morphism $\Gamma_{r,g}\to\Gamma_g$ considered in Lemma \ref{__Lem:Basis Reduced Loops} and let  $s:\Gamma_g\to\Gamma_{r,g}$ be an injective right-inverse map with 
$s(\Gamma_g)=\Gamma_{top}$, where $\Gamma_{top}$ the sub-group of $\Gamma_{r,g}$ generated  $S_{top}=\{x_1,y_1,\ldots,x_g,y_g\}$, built as follows. Consider a spanning tree $\mathcal{T}$ of the Cayley graph of 
$\Gamma_g$ generated by $x_1,\ldots,y_g.$ Identifying  $\Gamma_{top}$ with paths of the Cayley graph of $\Gamma_g$ starting from $1,$ set for any $\gamma\in \Gamma_g,$ $s(\g)\in\Gamma_{top}$ to be  the unique path of $\mathcal{T}$ from 
$1$ to $\g.$
Then $\Gamma_c$ is free of infinite countable rank with free basis $\{s(\g) z_i s(\g)^{-1}, \g\in \Gamma_{g}\}.$

\item Assume $\tilde v\in \tilde V_g, $ and that  $(\ell_i,1\le i\le r),$  $a_1,b_1,\ldots, a_g,b_g\in \mathrm{L}_v(\Gbb)$ and $\Theta: \Gamma_{r,g}\to\mathrm{RL}_v(\Gbb)$ are as in Lemma \ref{__Lem:Basis Reduced Loops}. Denote by $\mathrm{RL}_{top}(\Gbb)$ the sub-group of $\mathrm{RL}_v(\Gbb)$ generated by $a_1,\ldots,b_g.$  Then the restrictions of  $\Theta$ to $\Gamma_{top}$ and $\Gamma_c$ yield  isomorphisms $\Theta: \Gamma_{top}\to \mathrm{RL}_{top}(\Gbb)$ and $\Theta:\Gamma_c\to p(\mathrm{RL}_{\tilde v} (\tilde\Gbb))$.  Denoting by $\tilde\Theta: \Gamma_c\to \mathrm{RL}_{\tilde v} (\tilde\Gbb) $ the morphism with $p\circ\tilde\Theta=\Theta, $  the  diagram 
$$\begin{array}{cccccccc}
1 &\to&\Gamma_c&\to&  \Gamma_{r,g} & \to& \Gamma_g &\to 1\\
&&\hspace{0.2 cm}\downarrow \tilde\Theta&& \hspace{0.2 cm}\downarrow \Theta&&\downarrow&\\
1&\to &\mathrm{RL}_{\tilde v}(\tilde \Gbb)&\overset{p}{\to}& \mathrm{RL}_v(\Gbb)&\to&\pi_{1,v}(\Gbb)&\to 1
\end{array}$$
is commutative and exact. Consider  a spanning tree  $\mathcal{T}$ of $\tilde V_g$ and for any $x\in \tilde V_g$ denote by $\g_x$  the unique path of $\mathcal{T}$ from $\tilde v$ to $x.$ Then $\mathrm{RL}_{\tilde v}(\tilde \Gbb)$ is free of infinite rank, with free basis  $\left\{ \widetilde{ \g_x \ell_i\g_x^{-1}},x\in \tilde V_g,1\le i\le r\right\}.$
\end{enumerate}
\end{lem}

\emph{Tile decomposition:} For all $h\in \Gamma_g,$ we denote by $D_h\subset \tilde V ,  D^*_h\subset \tilde F $ and $ \mathring{D_h}\subset \tilde V$  the subsets of vertices and faces of $\Gbb$, whose image in $\tilde \Sigma_{\Gbb}$ is 
included respectively in $P_h$ and its interior $\mathring{P_h}$. The projection $\mathring{D}$ of $\mathring{D}_h$ does not depend on $h\in \G.$  When $U\subset \tilde F$ and $E_c\subset \tilde E,$ we denote by 
$U\setminus E_c$ the subgraph of  the graph of $\tilde\Gbb^*$ where all faces from $\tilde F\setminus U$ and  all edges dual 
to $E_c$ are removed. Let us consider the oriented graph with vertices $\Gamma_g$ such that there is an edge between $a$ and $b$ if and only if $P_{a}$ and $P_b$ share a side. The action of $\G_g$ on $\mathbb{H}$ induces a 
free, transitive, isometric 
action on this graph and we denote by $|h|_{\Gamma_g}$ the \emph{distance} between any $h\in \G_g$ and $1$. For any non constant regular loop $\ell=(e_1,\ldots,e_n)\in \mathrm{L}(\Gbb),$ we call 
$$|\ell|_D=n-1- \#\left\{1\le i\le n-1:\exists h\in \Gamma_g \text{ with } \{\underline{e_i},\overline{e_i},\overline{e_{i+1}}\} \subset  D_h \right\}$$
the \emph{tiling length} of\footnote{since the loop is regular, it is also understood as the number of pair consecutive  edges  of $\ell$ crossing the boundary of a polygon.} $\ell.$  There is then a unique tuple $\g_0,\ldots, \g_{|\ell|_D}$ of paths of $\Gbb,$ such that for any lift $\tilde \ell$ of $\ell$,  there are lifts $\tilde\g_0,\ldots , \tilde\g_{|\gamma|_D}$ of $\g_0,\ldots, \g_{|\ell|_D}$  such that
\begin{equation}
\tilde\ell=\tilde\g_0\ldots \tilde\g_{|\gamma|_D}\label{eq: Polygon traces}
\end{equation}
and for all $0\le k\le |\ell|_D$, there are 
$h_0,h_1,\ldots, h_{|\ell|_{|D|}}\in \Gamma_g$ such 
that all\footnote{This latter claim does not hold if $\ell$ is not regular.} vertices of $\tilde\g_k$ belong to $D_{h_k},$ while $\ell_\Gamma=(h_0,\ldots, h_{|\ell|_D})$ is a path in $\G_g$. We call $\ell_D=\g_{|\ell|_D}\g_0$ the \emph{initial strand} of $\ell.$ We call $\ell_\Gamma$  the \emph{tiling path} of $\ell$ and set
$$|\ell|_{\G}= |h_{|\ell|_D}|_{\G_g}.$$  
A loop $\ell_1$ of $(\Gbb,\Gbb_g)$ is called an \emph{inner loop} of $\ell$ if $\ell_1$ is regular, included in $\mathring{D}$ and $\ell_1\prec \ell.$ We then say that $\und{\ell_1}$ is a \emph{contractible intersection point} of $\ell$ and 
denote by $V_{c,\ell}$ the set of such points.  A \emph{proper loop} is a regular loop $\ell$ with $\#V_{c,\ell}=0.$

\vspace{0.5em}

A path $\g\in \mathrm{P}(\Gbb)$ is said to be \emph{geodesic} when its embedding in the surface is the restriction of a geodesic of the surface\footnote{Mind that we also consider the power of a geodesic to be a geodesic.}. A \emph{path} in $\G_g$ is \emph{geodesic} if it is the tiling path of a geodesic path of a regular map.

\subsection{Shortening homotopy sequence}
\label{====Sec: Shortening Homotopy sequence}

We  define here operations on regular loops allowing to decrease their tiling length. We say that a sequence $\ell_1,\ldots,\ell_n$ is a \emph{shortening homotopy sequence} from $\ell_1$ to $\ell_n$ if $\ell_1,\ldots,\ell_n$ are regular loops such that $|\ell_1|_D\ge \ldots\ge  |\ell_n|_D$  and for all $1\le l<n,$
$$\#V_{c,l}=\#V_{c,l+1}=0 \text{ or } \#V_{c,l}>\#V_{c,l+1},$$
while there is a regular map $(V,E,F)$ with $\ell_{l},\ell_{l+1}\in \mathrm{P}(\Gbb)$ and a subset of faces  $K_l\varsubsetneq F$, with $$\ell_l\sim_{K_l} \ell_{l+1}.$$ The aim of this section is to prove the following. 

\begin{prop}\label{__Prop: Shortening Homotopies} For any proper loop $\ell,$ there is a shortening homotopy sequence $\ell_1,\ldots,\ell_m$, a geodesic loop $\ell'$ and a path $\eta$ within the same map $\Gbb=(V,E,F)$ as $\ell_m,$ such that  $\ell_m\sim_K\eta \ell' \eta^{-1}$ for some $K\subset F$ with $K\not=F.$  The path $\eta$ can be chosen simple, within a fundamental domain and crossing $\ell_m$ and $\ell'$ only at their endpoints. \end{prop}

We need two additional notions for this proof. 

\vspace{0.5em}

\emph{Bulk of a loop:} Consider  a regular map $(\Gbb,\Gbb_g)$ with $\Gbb=(V,E,F)$, and  a contractible loop $\ell$ of $\Gbb$ whose lift is a  loop $\tilde\ell$ of 
$\tilde \Gbb$. Let $E_c$ be the set of edges used by $\tilde\ell$ and let $O_\ell$ be the unbounded component of 
$\tilde\Gbb^* \setminus E_c.$ The \emph{bulk} of $\ell$ is then $K_{\ell}=p(\tilde F\setminus O_{\ell}).$  Since $E_\ell$ is connected, the image of 
$O_\ell$ in $\tilde\Sigma_{\Gbb}$ is a surface with one boundary and  the image  $\tilde X_\ell$ of $\tilde F\setminus O_\ell$ in $\tilde \Sigma_{\Gbb}$ 
is a contractible set. The image of $\ell$ is then contractible within  $X_\ell=p(\tilde X_\ell)$ and  $$\ell\sim_{K_{\ell}} \ell_*$$
where $\ell_*$ is the constant loop at $\und{\ell}.$   

\vspace{0.5em}

\emph{Adding a rim to a regular map:} When $(\Gbb,\Gbb_g)$ is a regular map, let us define a map $\Gbb_r$ finer than $\Gbb$ in the following way. 
First add exactly one vertex to  each edge\footnote{Recall the definition of $\partial E$ and $\partial V$  for a regular map  in section  \ref{----sec:DiscreteFundamentalCover}.} of $E\setminus \pl E$ with one endpoint in $\pl V$ and exactly two when both endpoints belong to $\pl V$. Each new vertex is paired uniquely with a vertex of $\pl V$ and their set inherit the cyclic order of vertices of $\pl V.$  Second add an edge for each consecutive new vertices. We denote by $\Gbb_r$ 
the new map defined thereby  and call the set  $\pl_rE$ of edges added in the second step the \emph{rim} of $\Gbb$. Each face of the new map, whose boundary has an edge in $\pl E$ has exactly four adjacent edges with exactly one in $\pl E_r$.  We denote this set of faces by $F_{r}$. We denote all other faces of $\Gbb_r$ by $F_i$. For any $f\in F,$ either its boundary has no edge in $\pl E$ and it is 
identified to a face of $F_i$, or it is the union of faces of $\Gbb_r$ with exactly one in $F_i$, that we abusively also denote by $f$. For any oriented edge  
$e$ of $\Gbb_r$ belonging to $\pl E,$  its right retract  is the oriented edge of  $\pl E_r$ belonging to the face of $F_{r}$ on the right of $e.$ When $\g$ is a path with edges in $\pl E$, its \emph{right retraction} is the concatenation of the right retraction of its edges. The left retraction is defined likewise. 

\vspace{0.5em}
We can now prove the existence of shortening homotopy sequence starting from any regular loop, using a $5$ type of 
operations.

\vspace{0.5em}

\emph{\textbf{Step 1--}Deleting contraction points:} Consider a regular loop $\ell$ with $\#V_{c,\ell}>0$ of a regular map  with faces set $F$.  Any lift 
$\tilde \a$ of an inner loop $\a\prec \ell$ is a loop and we can consider its bulk. Denote by $K$ the union of  bulks for all inner loops. Any face bordering 
$\pl E$ does not belong to $K$ so that  $K \varsubsetneq F$ while $\ell$ is $\sim_{K}$-equivalent to the regular loop $\ell'$ with all inner loops 
erased.

\vspace{0.5em}

\emph{\textbf{Step 2--}Backtrack erasure:} Assume that $\ell$ is a regular loop of a regular map $(\Gbb,\Gbb_g)$  such that there is $1< i<|\ell|_D$ with $h_{i-1}=h_{i+1}$, where $(h_1,\ldots, h_{|\ell|_D})$ is the tiling path of 
$\ell.$ Consider the decomposition of $\tilde\ell$ as in \eqref{eq: Polygon traces}. Let $\Gbb'$  be the map $(\Gbb,\Gbb_g)$  with a rim added. Denote by $e_i$ and $e_o$ 
the last and first edge of $\g_{i-1}$ and $\g_{i+1}$. Then $\ov e_i$ and $\und e_o$ belong the same edge $e$ of $\Gbb_g.$  Let $\b'\in \P(\Gbb')$ be the reduced   path using only edges of the rim with 
$\und\b'=\ov e_i$ and $\ov\b'=\und{e_o}$. Denote by $\g_{i-1}'$ and $\g_{i+1}'$ the reduction of $\g_{i-1}e_i^{-1}$ and $e_o^{-1}\g_{i+1}.$ The \emph{backtrack erasure} for the backtracking 
$(h_{i-1},h_i,h_{i+1})$  of $\ell$  is the regular loop
$$\ell'= \g_1\ldots \g_{i-2}\g_{i-1}'\b'\g_{i+1}'\g_{i+2}\ldots \g_{|\ell|_D}. $$
It can be obtained from $\ell$ by the following discrete homotopy. Since a lift of the paths $\b'$ and $e_i\g_ie_o$ starting in $D_{h_{i-1}}$ both ends in 
$D_{h_i},$ the loop of $e_i\g_ie_o{\b'}^{-1} $ is contractible. Denote by $K_{bt}$ its bulk.  Then 
$$\ell\sim_{F_{bt}} \ell'.$$
Since $\g_i$ only intersect the rim of $\Gbb'$ through the edge 
$e$,  any face belonging to the rim whose boundary intersects two different edges of $\Gbb_b$ is not in  $K_{bt}.$ It follows that $K_{bt}\not=F'.$

\vspace{0.5em}

\emph{\textbf{Step 3-}Vertex switch:} Let $\ell$ be a regular loop of a regular map $(\Gbb,\Gbb_g)$ and consider its decomposition as in \eqref{eq: Polygon traces}.  A \emph{half  turn} of $\ell$  is a sequence 
$\g_{l},\ldots,\g_{l+k}$ such that
$2g\le k\le 4g-1,$ and  $D_{h_l},D_{h_{l+1}},\ldots, D_{h_{l+k}}$ runs around a common vertex $v\in \Gbb_g$. Consider such a long turn and let $\Gbb'=(V',E',F')$ be the map obtained from $\Gbb$ by adding twice a rim as described in the last paragraph. See Figure \ref{Fig----TS} for an example. Let $e_i$ and $e_o$ be respectively the last and the first edge of $\g_{l}$ and $\g_{l+k}$ in $\Gbb'.$ Besides, let $\b_p\in \mathrm{P}({\Gbb'}^*)$ be the shortest reduced 
path from a face adjacent of $e_i$ to a face adjacent of $e_o$ that crosses first $e_i$ and uses only faces of $F_{r}$ so that its lift starting from 
$D_{h_{l}}^*$ goes through $D_{h_{l}}^*\cup D_{h_{l+k}}^*$ and ends in $D_{h_{l+k}}^*$. 
Let $\b'\in \mathrm{P}(\Gbb')$ be the reduced
path  from $\und{e_i}$ to 
$\und{e_o}$, such that each edge of $\b'$ 
is bordering a face of $\b_p$.   Denote by $\g_{l}'$ and $\g_{k+l}'$ the reduction of $\g_{l}e_i^{-1}$ and $e_o^{-1}\g_{k+l}$. The  \emph{vertex switch} of $\ell$ for the considered half  turn is  the regular loop 
$$\ell'= \g_0\g_1\ldots \g_{l-1} \g_{l}'\beta' \g_{k+l}'\g_{k+l+1}\ldots \g_{|\ell|_D}.$$
It can be obtained from $\ell$ by the following discrete homotopy.  Consider the loop $e_i\g_{l+1}\ldots \g_{l+k-1}e_o{\b'}^{-1}.$ Since a lift of $\b'$ starting in $D_{h_l}$ ends in $D_{h_{l+k}}$ it follows that $e_i\g_{l+1}\ldots \g_{l+k-1}e_o{\b'}^{-1}$ is contractible. Denote by $K_{sw}$ its bulk.  

Then, 
$$ e_i\g_{l+1}\ldots \g_{l+k-1}e_o\sim_{K_{sw}} \b'$$
and
$$\ell\sim_{K_{sw}} \g_0\ldots\g_{l-1}\g_l' e_i \g_{l+1}\ldots \g_{l+k-1}e_o \g_{k+l+1}\ldots \g_{|\ell|_D} \sim_{K_{sw}} \ell'.$$ 

Besides, $F_{sw}\not= F'.$  Indeed, consider the map $\Gbb_1$ obtained by adding a single rim to $\Gbb, $ so that $\Gbb'$ is finer than $\Gbb_1.$ Let $F_{cr},\tilde F_{cr}$ be the set of of faces of $\Gbb_1$ 
neighbouring respectively $p(v)$ and $v$. The restriction of $p$ to $\tilde F_{cr}$ is a homeomorphism onto $F_{cr}$. Since $k<4g$, there is at least one face $\tilde f_{cr}$ of $\tilde F_{cr}$ that does not belong to $p^{-1}(F_{sw}).$ Since $\b$ uses only faces of $F_{r}'$, any face of $F'\setminus F_{r}$  included in  $f_{cr}=p(\tilde f_{cr})$ does not belong to $K_{sw}.$

\begin{figure}[!h]
\centering
\includegraphics[scale=0.6]{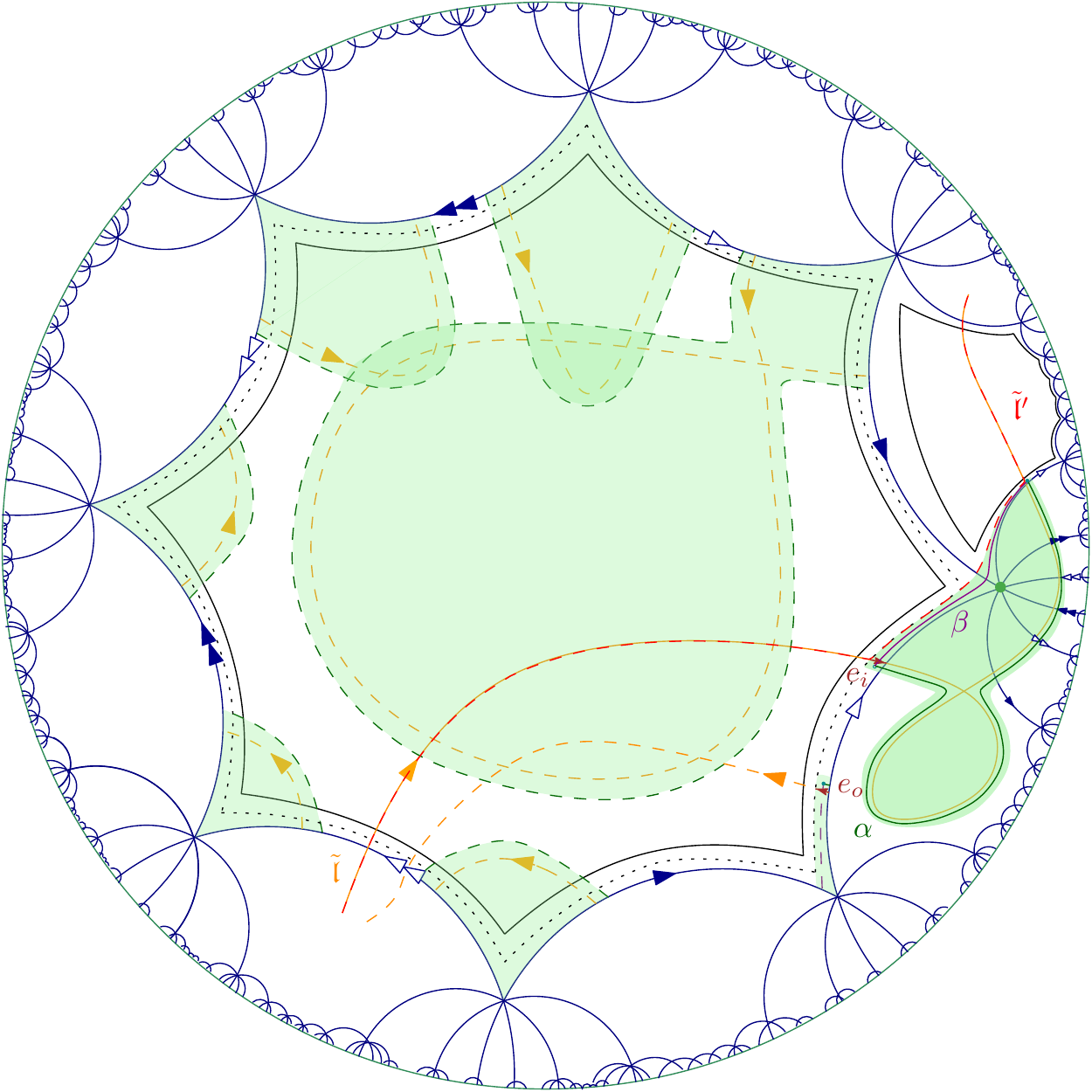}
\caption{\small Discrete homotopy at a left turn of $\ell$ when  $g=2$ and $k=7.$  The latter vertex is shown as a green dot, contracted faces are shown in green. The second rim is displayed with dotted lines. A lift $\tilde\ell$ of the initial loop in displayed in plain orange line, while a lift $\tilde\ell'$ of the terminal  loop is displayed in dashed red line.\label{Fig----TS}}
\end{figure}

\vspace{0.5em}

The following lemma reformulates a result due to \cite{BiS} relating $|\ell|_D$ to long turns of $\ell$ when $g\ge 2.$  

\begin{lem}\label{__Lem: Towards Geodesic Tiling Path} Let  $\ell$ be a  regular loop of a regular map $(\Gbb,\Gbb_b)$. There is a finite sequence $\ell_1,\ldots,\ell_n$  or regular loops obtained by vertex switches or backtracking erasures such that 
$\ell_1=\ell,$ $|\ell_1|_D\ge |\ell_2|_D\ldots \ge  |\ell_n|_D$  and 
$$|\ell_n|_D= |\ell|_{\Gamma_g}.$$
\end{lem}
\begin{proof} The case $g=1$ is elementary. An argument goes as follows.  The path in $\G_1=\Z^2$ associated to $\ell$ can be assumed up to axial symmetries that $h_{\ell}$ has non-negative coordinates. A backtracking of $\ell_\G$ can be erased by a backtracking erasure of $\ell.$  A path is geodesic if and only if all of its increments coordinates are non-negative. There are two consecutive increments with a negative followed by a positive sign.  This pair corresponds to a backtrack or a half turn of $\ell$ if one or two coordinates change.  Applying to a backtrack erasure or a switch at the half turn, the new loop has one less pair of increments with coordinates changing sign.

\vspace{0.5em}

When $g\ge 2,$ the result follows from  \cite[Lemma 2.5]{BiS}. In the setting of \cite{BiS}, a half 
turn of $\ell$ is  a half cycle of the path in $\G_g$ associated to $\ell.$ A switch at a half turn corresponds to a replacement of a half cycle with its complementary. Moreover in the setting of \cite{BiS}, replacing a long chain by its complementary chain can be obtained by successively replacing a long cycle by its complementary cycle.
\end{proof}

\vspace{0.5em}

\emph{\textbf{Step 4--}From minimal tiling length to geodesic tiling paths:} We say that a regular path $\g$ of a regular map  has \emph{minimal tiling length} when $|\g|_D= |\g|_\G.$ When $g\ge 2$, the following is a consequence of \cite[Thm 2.8]{BiS}.

\begin{lem} \label{__Lem: Towards Geodesics} If $\ell$ is a regular loop of a regular map, there is a sequence of regular loops $\ell_1,\ldots,\ell_n$ with minimal tiling length equal to $|\ell_1|_\G$ obtained by switches and backtrack erasure, such that $\ell_1=\ell$ while 
the tiling path of $\ell_n$ is geodesic. 
\end{lem}
\begin{proof} When $g\ge 2,$ in the setting of \cite{BiS},  our condition for a tiling path to be geodesic is equivalent for it to be a shortest path. Since switches at half turns imply switches for half cycles of the tiling path 
in the setting of \cite{BiS}, the result follows from point (c) of \cite[Thm 2.8]{BiS}. 

When $g\ge 1,$ for any regular loop with minimal tiling length, we can assume w.l.o.g. that both coordinates of the endpoint $(a,b) $ of $\ell_\G$  are non-negative. When $\g$ is a path of $\Z^2$ with only positive coordinates, a corner swap of $\g$ is the path obtained by replacing a sequence of the form  $(x,y),(x+1,y),(x+1,y+1)$ with  $(x,y),(x,y+1),(x+1,y+1)$  or vice-versa.  Any other path of $\Z^2$ with same endpoints can be obtained by corner swaps and backtrack erasure. Since a switch at a half turn of $\ell$ implies a corner swap of its tiling path and that tiling paths with positive coordinates have minimal length in $\Z^2$,  the claim follows. 
\end{proof}

\emph{\textbf{Step 5--}From geodesic tiling paths to geodesic paths:}  Assume that $\ell$ is a regular loop  such that $\ell_\G$ is geodesic and set $n=|\ell|_D=|\ell|_\G$. Let $\ell^{(*)}$  be a  geodesic loop   with $\ell^{(*)}_\G=\ell_\G.$ Up to translation of the geodesic associated to $\ell^{(*)}$, we can assume that $\ell^{(0)}$ and $\ell^{(*)}$ are regular paths of a same regular map $(\Gbb^{(0)},\Gbb_g^{(0)})$. Let $\eta\in P(\Gbb^{(0)})$ that does not cross the boundary of the polygon,  while $\und{\eta} = \und{\ell}$ and $\ov{\eta}= \und{\ell}^{(*)},$ without using any edge of $\ell^{(*)}.$  Denote by $(\Gbb,\Gbb_g)$ the regular map obtained by adding a rim to $(\Gbb^{(0)},\Gbb_g^{(0)}).$ Using the same notation as in \eqref{eq: Polygon traces}, consider the tile paths decompositions of $\ell$ and $\ell^{(*)}$ adding an upper-script $(*)$ for the second decomposition. For any $0\le k\le n-1,$  let  $e_k$ and $e_k^{(*)}$ be the last edges of respectively $\g_{k}$ and $\g_{k}^{(*)}$, denote by $\b_{k}$  the reduced path with edges in $\pl E_r $ from  $\und{e_k^{(*)}}$ to $\und {e_k}$ and define $\ell^{(k)}$ as the reduction of
$$\eta\g^{(*)}_0\ldots \g^{(*)}_{k} {e_k^{(*)}}^{-1}\b_k e_k\g_{k+1}\ldots \g_n.$$  
Let us set $\ell^{(n)}= \eta \ell^{(*)}\eta^{-1} $ and   $\ell^{(-1)}=\ell.$ Let $\a_k$ be  the reduction of the loop  $\eta^{-1}\g_0e_0^{-1}\b_0^{-1}e_0^{(*)}{\g_0^{(*)}}^{-1}$ when $k=0,$ ${e_{k-1}^{(*)}}^{-1}\beta_{k-1} e_{k-1} \g_k {e_k}^{-1} \b_k^{-1}{\g_k^{(*)}}^{-1}$ when $0< k<n$ and ${e_{n-1}^{(*)}}^{-1}\beta_{n-1} e_{k-1} \g_n \eta{\g_n^{(*)}}^{-1} $  when $k=n.$  With this notation
$$\ell\sim_r\eta\a_0 \g_0^{*} \a_1\g_1^{*} \ldots \a_k  \g_{k}^{(*)}{e_k^{(*)}}^{-1} \beta_k e_k \g_{k+1} \ldots \g_n \text{ for }0\le k<n $$
and $$\ell\sim_r \eta\a_0 \g_0^{*} \a_1\g_1^{*} \ldots \a_n \g_{n}^{(*)} \eta^{-1}. $$
Therefore, for all $0\le k\le n$
\begin{equation}
\ell^{(k-1)}=\a\b\text{ and }\ell^{(k)}= \a\a_k\b\label{eq:Induction Inner contraction to geodesic}
\end{equation}
for some paths $\a,\b\in \mathrm{P}(\Gbb).$ For all $0\le k\le n,$ $\a_k$ is contractible. Denoting by $K_k$ its associated bulk,  \eqref{eq:Induction Inner contraction to geodesic} yields
$$\ell^{(k)}\sim_{K_k} \ell^{(k-1)} \text{ for all }0\le k\le n.$$
Besides, since $\a_k$ intersects at most two edges of $\Gbb_g$, any face within the rim $f\in F_{r}$, which borders a different edge of $\Gbb_g$, does not belong to $K_k.$ Therefore  $K_k\not=F.$ 

\begin{proof}[Proof of Proposition \ref{__Prop: Shortening Homotopies}]  For any regular loop $\ell$, the claimed shortening homotopy sequence can be obtained by applying first the deletion of contraction points, followed by Lemma \ref{__Lem: Towards Geodesic Tiling Path}, \ref{__Lem: Towards Geodesics} and lastly a shortening homotopy sequence from a loop with  geodesic  tiling path to a loop conjugated to a geodesic loop.
\end{proof}

The following lemma is not necessary for our main argument and can be skipped at first reading. Let us note that it is also possible to do the vertex switch operation (step 3) before deleting contraction points  (Step 1) thanks to the following.

\begin{lem} \label{__Lem:Complement Inner} Consider $\ell$ is a regular loop within a regular map $(\Gbb,\Gbb_g)$ with faces set $F$.  Denote respectively by  $K$ and  $E_{in}$  the union of bulks and the set of edges of its initial strand $\ell_D.$ Then $F\setminus K$ is connected in $\Gbb^*\setminus (\pl E\cup E_{in})$.
\end{lem} 

\begin{proof} Since $\ell$ is regular, any edge crossing $\pl E$ does not belong to $E_{in}$  and faces adjacent to $\pl E$ belong to the same connected component $X$ of $F\setminus K$ in $\Gbb^*\setminus (\pl E\cup E_{in})$. Denote by $\tilde X$ the lift of $X$ in $D^*_1.$ Assume that $F\setminus K$ is not connected in $\Gbb^*\setminus (\pl E\cup E_{in})$ and consider a connected component $K'$ different from $X$.  Then all edges of $\pl K'$ belong to $E_{in}$.  Since the infinite connected component of  $\tilde\Gbb^*\setminus E_{in}$ is  given by $\tilde F\setminus D^*\cup X,$ the lift of  $K'$ in $D^*_1$ is included in the bounded connected component of $\tilde\Gbb^*\setminus E_{in},$  where we identified $E_{in} $ with the set of edges of the lift of $\ell_D$ starting from $D_1.$   It follows that $K'$ is included in $K$, which is a contradiction. \end{proof}

\subsection{Nested and marked loops}

\label{-----sec:Nested Marked Loops}

\emph{Nested loop:}  We say that a  loop $\ell$ of a regular map with $n$ transverse intersection points is \emph{nested} if it is regular and if there  are sub-loops $\ell_1\prec \ell_2\prec \ldots \prec\ell_{n}$ with a strictly increasing 
number of intersection points.
By convention, a constant loop is a nested loop. A regular loop is nested if and only if its transverse intersection points can  be labeled  $v_1,v_2,\ldots, v_n$ so that it visits them in the order $(v_1v_2\ldots v_{n-1}v_n v_nv_{n-1} \ldots v_2 v_1).$ See figure \ref{Fig----Nested loops}.

\begin{figure}[!h]
\centering
\includegraphics[scale=0.4]{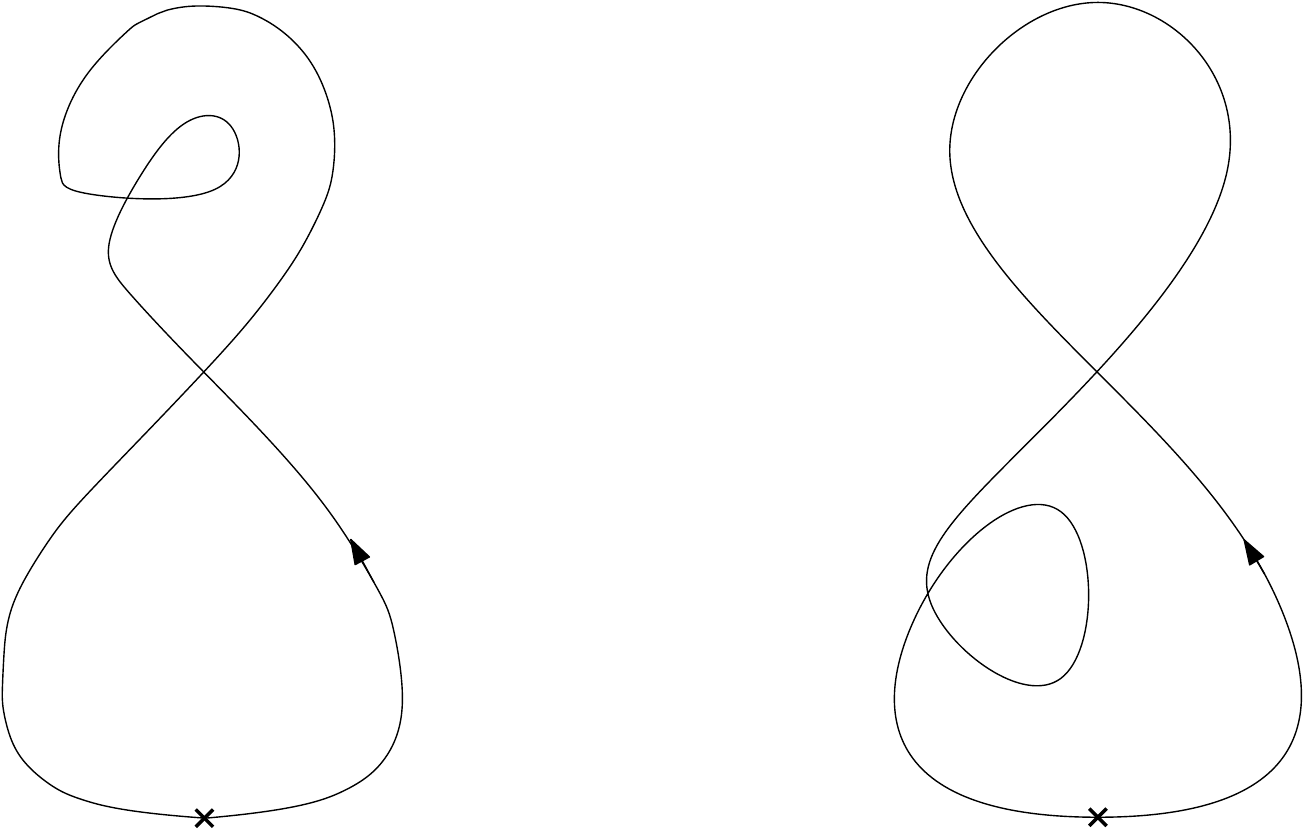}
\caption{Left, a nested loop. Right, this is not a nested loop. \label{Fig----Nested loops}}
\end{figure}

\begin{rmk}  A nested loop is an example of a splittable loop as defined in \cite[Section 6.5]{DN}, originally introduced in \cite{KazPlan} and called therein planar loops. Note that the right example in figure \ref{Fig----Nested loops} is splittable but not nested.

\end{rmk}

\emph{Marked loops:}  A \emph{marked loop} is a couple $(\ell,\g_{nest})$ of a regular loop and a regular path within a regular map $\Gbb$ such that 
\begin{enumerate}
\item When $(\g_0,\ldots,\g_{|\ell|_D})$ denotes  the tiling decomposition of $\ell,$  $\g_0=\g_{nest}\g'$, for some path $\g'.$
\item  The path $\g_{nest}$ is non-constant and of the form $\a \ell_{nest}\b$ where $\ell_{nest}$ is a nested loop and $\a,\b$ are  simple paths,  such that the only intersections between  $\a,\b$  and $\ell_{nest}$ are at $\ov\a$ and $\und \b.$
\item The path $\g_{nest}$ does not intersect transversally the two components of  the initial strand $\ell_D.$
\item The path $\g_{nest}$ does not intersect any inner loop of $\a\b\g'\g_1\ldots\g_{|\ell|_D}.$ 
\item   The bulk  $F_{nest}$ of the contractible loop $\ell_{nest}$ has exactly $\#V_{\ell}$ faces of $\Gbb$ and there is exactly one face $f_o$ of $\Gbb$ adjacent to $F_{nest}$ in $\Gbb^*.$
\end{enumerate}

We call the loop  and the path defined by $(\ell,\g_{nest})^{\wedge}=\a\b\g'\g_1\ldots\g_{|\ell|_D}$ and $(\ell,\g_{nest})^{\wedge_*}=\g'\g_1\ldots\g_{|\ell|_D}$ the \emph{pruning}  and the \emph{cut} of $(\ell,\g_{nest}).$  We shall often denote them abusively simply by $\ell^\wedge$ and $\ell^{\wedge_*}.$ We call $f_o $ the \emph{outer face} of $(\ell,\g_{nest})$ and the simple sub-loop of $\ell$ with length $1$ the \emph{central} loop of $(\ell,\g_{nest}).$ Being a sub-loop of $\ell_{nest}$,  it is contractible, faces belonging to its bulk are called \emph{central}. A \emph{moving edge} is  an edge  $e$ of $\ell$ with the following property: 

\begin{itemize}
\item When $\ell_{nest}$ is constant,  $e$ is any edge of $\g_{nest}.$
\item Otherwise, $e$ bounds a central face of $\ell_{nest}$.
\end{itemize}

\begin{figure}[!h]
\centering
\includegraphics[scale=0.7]{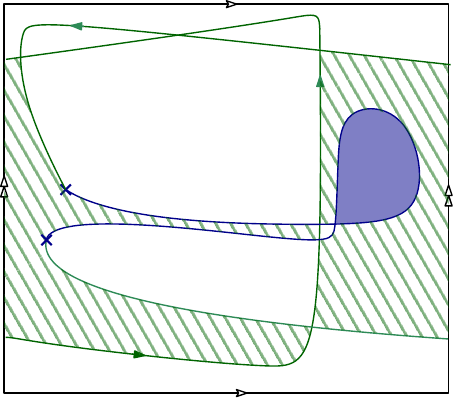}
\caption{A marked loop. Its nested part is drawn in blue. There are exactly one central face coloured in blue and one outer face filled with dashed green lines.   \label{Fig----Marked loop}}
\end{figure}

\begin{rmk}  For any nested loop $\ell$  included  in a fundamental domain, it easily shown by induction on $n=\# V_{\ell},$ that  the dual  graph $\Gbb^* $ with the edges of $\ell$ removed has exactly $n+1$ connected components. The fifth condition above can be 
removed  considering regular maps finer than $\Gbb.$ 
\end{rmk}

The following is then a simple variation of  Proposition \ref{__Prop: Shortening Homotopies}.

\begin{lem}  \label{__Lem: Shortening Homotopies ML}
For any marked  loop $(\ell,\g_{nest})$ with $\ell^\wedge$ proper,  there is a shortening homotopy sequence $\ell_1,\ldots,\ell_{m}$ such that \begin{enumerate}
\item  $\ell_1\sim_c\ell,$
\item  There is a  common nested sub-path $\g_{nest}$ of $\ell_1,\ldots,\ell_n$, such that $(\ell_k,\g_{nest})$  is a marked loop for all $k\ge 1$ and $\ell_k^\wedge$ is proper for $k\ge 2$. 
\item There are proper   subsets $K_1,\ldots, K_m$ of faces,  such that $\ell_k^{\wedge_*} \sim_{K_k}\ell_{k+1}^{\wedge_*}$ for all $1\le k< m.$
\item There is a marked loop $(\ell',\g_{nest}')$ such that  $ \ell_{m}\sim_\Sigma\ell' $ and ${\ell'}^{\wedge}$ is geodesic. 
\end{enumerate}
\end{lem}

\subsection{Pull and twist moves}

We introduce here two operations on loops in order to later modify shortening homotopy sequences  to satisfy the constraint imposed by Makeenko--Migdal equations, namely to keep constant the algebraic area of loops introduced in section  \ref{----sec: Makeenko--Migdal Vectors}. This type of operation shall be required only when considering loops with vanishing homology.

\vspace{0.5em}

\emph{Pull move:} Consider a non-constant marked loop $(\ell,\g_{nest})$ in a regular map $(\Gbb,\Gbb_g)$ with $\ell^\wedge$ has no inner loops that are sub-paths of $\ell_D$.   Then the graph obtained from the dual graph $\Gbb^* $  by removing all  edges crossing $\pl E$ or $\ell_D$  but  edges of $\g_{nest}$ is connected. For any face $f$ of $\Gbb$ and any moving edge $e$ that does not bound $f,$  there is therefore  a  simple path $\g^*=a^*_1\ldots a^*_m$ in the dual graph $\Gbb^* $  with endpoint $f$ and  first edge $a_1^*$ dual to  $e$, that  crosses neither $\pl E$ nor $\ell_D$ but possibly at $\g_{nest}$.  Let us define inductively a new  map  $\Gbb'$ finer than $\Gbb$, a new marked loop $(\ell',\g_{nest}')$, as well as a subset $F_{stem}$ of faces of $\Gbb'$.   An example of the result in displayed in Figure \ref{Fig----Pull}.  Let us first set $F_{stem}=\emptyset.$ Denote by $a_1,\ldots,a_m$ the edges dual to $a_1^*,\ldots,a_m^*.$ Let $k\ge 1$ be the largest $k$ such that $a^*_k$ is dual to an edge of $\g_{nest}$. 
\begin{enumerate}
\item  Add  two new vertices to all edges dual to  $a_{k}^*,\ldots,a^*_m$. For all $l\ge k,$ when $a_l=a_{l,0}a_{l,1}a_{l,2}$ is  the edge decomposition of  $a_l$ in the new map, replace $a_l$ by $a_{l,1}.$
\item  Cut  all faces visited by  $a_{k}^*\ldots a^*_m$ but $f$ into three faces adding two non-crossing 
edges such  that endpoints of a new edge do not belong to the same initial edge. Add to $F_{stem}$ all new faces bounded by $2$ new edges.
\item Cut the face $f$ into two faces, adding an edge connecting the two new vertices on the edge dual to $e^*_m$ introduced in step 2. Add to $F_{stem}$  the new face included in $f$ whose boundary has only two edges. 
\item Denote by $\eta$ the simple path using only edges added in step 2 and 3 such that $\und{\eta}=\und{a_k}$ and $\ov{\eta}=\ov{a_k}.$ Transform $\ell$ and $\g_{nest}$ replacing the occurrence of the edge $a_k$ and $a_k^{-1}$ by respectively $\eta$ and $\eta^{-1}.$ 
\item  When $k=1$ stop the procedure. Otherwise, repeat this operation for the nested loop obtained in step 4 and the path $ a^*_{1}\ldots a^*_m$.
\end{enumerate}
The last marked loop produced is called the \emph{pull of  $(\ell,\g_{nest})$ along}  $\g^*.$

\begin{figure}[!h] 
\centering
\includegraphics[scale=0.9]{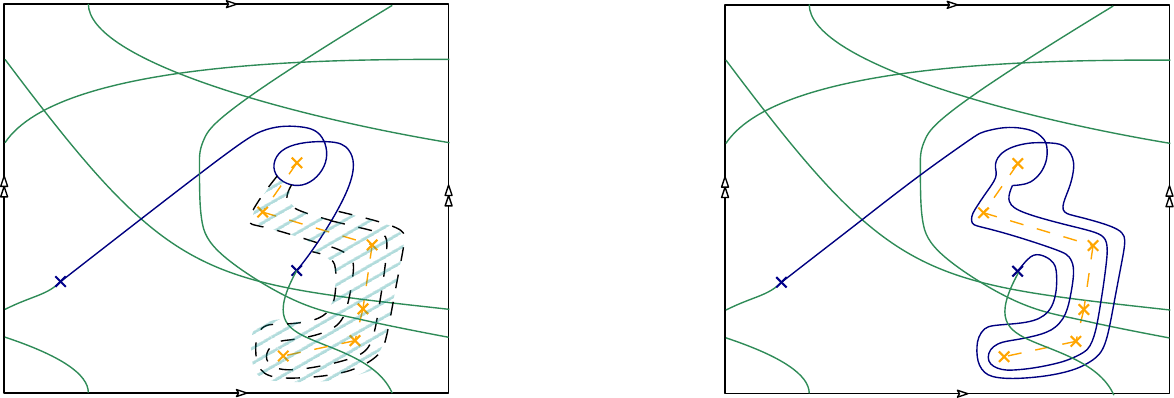}
\caption{ \label{Fig----Pull} Left: A marked loop with the nested part drawn in blue. New edges of the modified regular map are drawn with dashed lines. The union of faces of $F_{stem}$ is stroke with dashed lines. Right: Pull of the left marked loop along  the path of the dual drawn in orange.}
\end{figure}

\begin{figure}[!h] 
\centering
\includegraphics[scale=0.9]{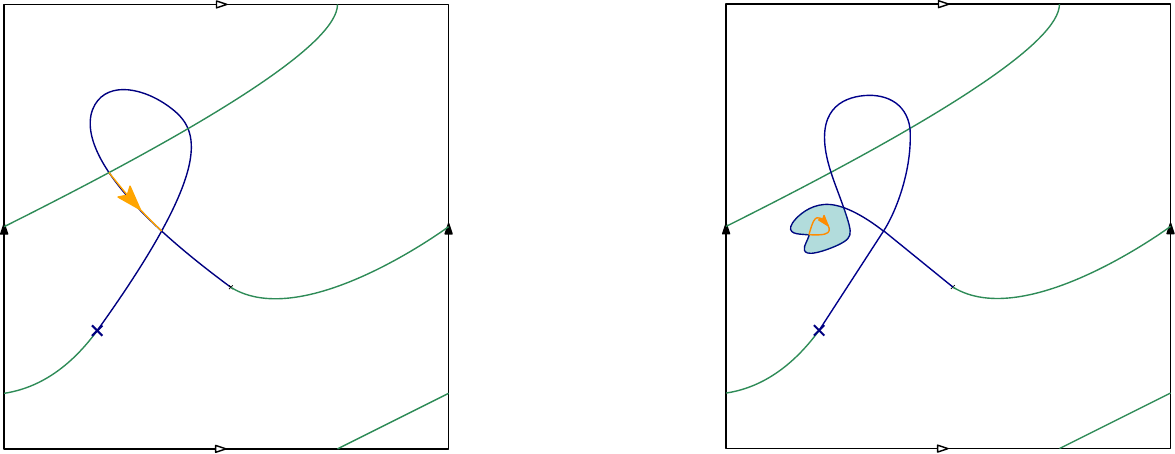}
\caption{\label{Fig----Twist}Left: A marked loop with the nested part drawn in blue. The chosen moving edge is drawn in orange. Right: $n$-twist of the left marked loop, with $n=-2$ and the chosen moving edge.  The new moving edge is displayed in orange.}
\end{figure}

\emph{Twist move:} Consider a marked loop $(\ell,\g_{nest})$  with a moving edge $e$.  

Let us refine a regular and marked loop as follows. Add a vertex to $e$ and cut the face left of $e$ into two faces, adding an oriented edge $e'$ with both endpoints equal to the new vertex, such that $e'$ is the boundary a positively oriented face.  The initial 
moving edge reads $e=e_1e_2$ in the new map. The \emph{left twist} of $(\ell,\g_{nest})$ is the marked loop obtained by replacing the occurrence of $e$ by $e_1e'e_2$ in both $\ell$ and $\g_{nest}.$  The new marked loop has then $e'$ as unique moving 
edge. We denote by $F_{tw}$ the face bounded by $e'.$ The right twist of $(\ell,\g_{nest})$ is defined similarly considering the right face and a negative orientation. When $n$ is respectively positive or negative, the $n$-twist of a marked loop is obtained by 
applying respectively $n$ left twists or $-n$ right-twists. We denote then by $F_{tw}$ the  $|n|$ faces of the new map bounded solely by newly added edges. 

\subsection{Vertex desingularisation and complexity}

\label{----sec:Vertex desingularisation}
Consider a regular map $\Gbb.$ Assume that  $\ell$ is a regular loop and $v\in V_{\ell}$ is an intersection point. We denote  by $\ell_1$ and $\ell_2$ the two sub-loops of $\ell$ based at $v$ such that  $\ell\sim_c\ell_1\ell_2.$ We then set
$$\delta_v \ell=\ell_1\otimes \ell_2 \in \C[\Ld_c(\Gbb)]^{\otimes 2},$$
with the convention that $\ell_1$ is left of $\ell_2$ at $v$ as displayed on Figure \ref{Fig----MMDeSing}.  By definition of Makeenko-Migdal vectors given in section \ref{----sec: Makeenko--Migdal Vectors}, there are\footnote{We fix them arbitrarily, for instance using the pseudo-inverse of the Gram matrix of the spanning family $(\a_v)_{v\in V_\ell}  $ and $(\b_e)_{e\in E\setminus E_\ell}$.} linear forms $(\a_v)_{v\in V_\ell}  $ and $(\b_e)_{e\in E\setminus E_\ell}$ on $\mfm_\ell$ such that
$$X= \sum_{v\in V_\ell} \a_v(X)\mu_v+\sum_{e\in E\setminus E_\ell} \b_e(X)d\omega_e,\ \forall X\in \mfm_\ell.$$
We then set 
$$\delta_X\ell= \sum_{v\in V_\ell} \alpha_{v}(X)\delta_v \ell.$$

Let us define a complexity on loops that strictly decreases after such operations.  Let us set
\begin{equation}
\mathcal{C}(\ell)=|\ell|_D+\# V_{c,\ell}
\end{equation}
when $\ell$ is a regular loop and
\begin{equation}
\mathcal{C}^{\mathfrak{m}}(x)=|\ell|_D+\# V_{c,\ell^\wedge}
\end{equation}
when $x=(\ell,\g_{nest})$ is a marked loop.  

\begin{lem} \label{----lem: desingularisation}\begin{enumerate}
\item For any regular loop $\ell$, $v\in V_\ell,$ if $\delta_v\ell=\ell_1\otimes \ell_2,$ then
\[ \mathcal{C}(\ell_1), \mathcal{C}(\ell_2)<  \mathcal{C}(\ell). \]
Moreover if $[\ell]\not=0$,  then $[\ell_1]$ or $[\ell_2]\not=0.$
\item For any marked loop $x,$ $\mathcal{C}^{\mfk m}(x)$ only depends on $x^{\wedge_*}.$ Moreover, when $y=(\ell,\g_{nest})$ is a marked loop with $y^{\wedge_*}=x^{\wedge_*},$  for  $v\in V_{\ell^\wedge},$ if $\delta_v \ell=\ell_1\otimes \ell_2,$ then there are  $\ell'_1,\ell'_2,$ with $\ell'_i\sim_c\ell_i$ and subpaths $\g_1,\g_2$ of  
$\ell'_1,\ell'_2,$ such that $x_1=(\ell'_1,\g_1),x_2=(\ell'_2,\g_2)$ are marked loops with 
\[ \mathcal{C}^{\mfk m}(x_1), \mathcal{C}^{\mfk m}(x_2)<  \mathcal{C}^{\mfk m}(x). \]
\end{enumerate}
\end{lem}

\begin{proof} Consider a regular loop $\ell.$ When $v\in V_{c,\ell}$,  one  can assume that $\ell_2$ is an inner loop that is $|\ell_2|_D=0$, and 
$$\#V_{c,\ell_1}+\#V_{c,\ell_2}+1= \#V_{c,\ell}. $$
Otherwise, both $\ell_1$ and $\ell_2$ are regular loops both crossing $\pl E$ at least twice so that 
$|\ell_1|_{D},|\ell_2|_{D}>0.$ Moreover  
$$|\ell_1|_{D}+|\ell_2|_{D}=|\ell|_D$$
since both count the number of edges of $\pl E$ crossed by $\ell.$ Therefore, 
\begin{equation}
|\ell_1|_{D},|\ell_2|_{D}<|\ell|_D.\label{eq:TilingL Decre}
\end{equation}
Moreover   $\omega_{\ell}=\omega_{\ell_1}+\omega_{\ell_2},$ $[\ell]=[\ell_1]+[\ell_2].$ In particular, if $[\ell]\not=0,$ $[\ell_1]\not=0$ or $[\ell_2]\not=0.$  This concludes the proof of the first point.
Consider now two marked loops  $x=(\ell',\g'_{nest}),y=(\ell,\g_{nest})$ with  $y^{\wedge_*}=x^{\wedge_*}.$ Then
\[  |\ell^\wedge|_D=|{\ell'}^\wedge|_D\text{ and }\#V_{c,\ell^\wedge}=\# V_{c,{\ell'}^\wedge},\]
so that $\mathcal{C}^{\mfk m}(x)=\mathcal{C}^{\mfk m}(y).$
Assume that $v\in V_{\ell^\wedge}$ and  $\delta_v\ell=\ell_1\otimes \ell_2$ such that $\g_{nest}$ is a subpath  of $\ell_2.$ Consider $e$ the first edge of $\ell_1$ and $\ell'_2\sim_c\ell_2$ with $\underline{\ell_2'}=\underline{\g_{nest}}.$ Then  $(\ell_1,e),(\ell'_2,\g_{nest})$ are marked loops.  If $v\in V_{c,\ell^\wedge},$ 
$$\#V_{c,\ell_1^\wedge}+\#V_{c,{\ell_2'}^\wedge}+1= \#V_{c,\ell}. $$
Otherwise, $|\ell_1^\wedge|_D,|\ell_2^\wedge|_D>0$ and the proof of 2.  follows as for the first point.  \end{proof}
Let us fix a choice for $x_1,x_2$ used in the above lemma. Consider a marked loop $x=(\ell,\g_{nest}),$ $v\in V_\ell$ and  assume $\delta_v \ell=\ell_1\otimes \ell_2$. When $v\in V_{\ell^\wedge},$ exactly one loop say $\ell_1$ has $\g_{nest}$ as sub-path and we set  $x_1=(\ell'_1,\g_{nest})$   and $x_2=(\ell_2,e)$ where $e$ is the first edge of $\ell_2,$ and $\ell'_1\sim_c\ell_1$ with  $\underline{\ell'_1}=\underline{\g_{nest}}.$  When $v\in V_{\ell_{nest}},$   exactly one loop say $\ell_1$ is a sub-loop of $\ell_{nest},$  set  $x_1=(\ell_1,\ell_1)$   and $x_2=(\ell',\g_{nest})$ where $(\ell',\g'_{nest})$ is obtained from  $(\ell,\g_{nest})$  by erasing the edges of $\ell_1$ (then   $\ell'\sim_c\ell_2).$ Otherwise, we set $x_1=(\ell_1,e_1),x_2=(\ell_2,e_2),$ where $e_i$ is the first edge of $\ell_i.$  We then write 

\begin{equation}
\delta_v x=x_1\otimes x_2.\label{eq-------Conv Desing MarkedL}
\end{equation}

\section{Yang--Mills measure and Makeenko--Migdal equations} 

\subsection{Metric and heat kernel on classical groups}

\label{----sec:HK}

We recall here briefly the definition and main properties of the heat kernel on classical groups that will be needed to define the discrete Yang--Mills measure. These results are quite standard, and can also be found for instance in 
\cite[Section 1]{Lev}. In this text, for any $N\ge 1,$ we denote by $G_N$ a compact classical group of rank $N$, that is $\U(N),\SU(N),\SO(N)$ or $\Sp(N),$ following the same conventions as in section 2.1.2 of \cite{DL}.

\vspace{0.5em}

For any compact Lie group $G$, its Lie algebra $\mathfrak{g}$ is endowed with an invariant inner product $\<\cdot,\cdot\>.$  Setting $$\mathcal{L}_X f(g)=\left.{\frac d{dt}}\right|_{t=0} f(ge^{tX}), \ \forall f\in C^{\infty}(G) \text{ and }g\in G, $$
the Laplacian associated to $\<\cdot,\cdot\>$ is the operator defined by
$$\Delta_G f= \sum_{1\le i\le d} \mathcal{L}_{X_i}\circ\mathcal{L}_{X_i}(f),\ \forall f\in C^{\infty}(G),  $$
where $(X_i)_{1\le i \le d}$ is an arbitrary orthonormal basis.    
\begin{dfn} 
The \emph{heat kernel} on $G$ is the solution $p:(0,\infty)\times G\to\R_+,(t,g)\mapsto p_t(g)$ of the heat equation, with $p_t\in C^{\infty}(G)$ for all $t>0$ and 
\begin{equation}\label{eq:heat_kernel}
\left\lbrace\begin{array}{ll}
\partial_t p_t(g)  &= \Delta_G p_t(g),\ \forall g\in G,\ \forall t>0,\\
\lim_{t\downarrow 0}p_t(g)dg  &=  \delta_{I_N},
\end{array}\right.
\end{equation}
where the convergence in the second line holds weakly.
\end{dfn}
It defines a semigroup for the convolution product, that is 
\begin{equation}
p_{t}*p_s=p_{t+s},\ \forall t,s>0.\label{eq: HK convolution}
\end{equation}
It inherits the following properties from the conjugation invariance of the scalar product: for all $g,h\in G$ and $t>0,$
\begin{equation}
p_t(hgh^{-1})=p_t(g)\label{eq:Inv HK Conj}
\end{equation}
and
\begin{equation}
p_t(g^{-1})=p_t(g).\label{eq:Inv HK Inverse}
\end{equation}

When $G_N$ is a compact classical group of rank $N$, we choose the invariant inner product $\<\cdot,\cdot\>$ as  in (1) of \cite[Section 2.1.2]{DL}.

\subsection{Area weighted maps, Yang--Mills measure and area continuity}

\label{------sec:YM area continuity}

We recall here a definition of the discrete and continuous Yang--Mills measure in two dimensions on arbitrary surfaces, with a focus on the former.

\vspace{0.5em}

\emph{Area vectors and area-weighted maps:} 
When $\Gbb=(V,E,F)$ is a topological map, an area vector is a function $a:F\to \R_+.$ We say that $(\Gbb,a)$ is an \emph{area-weighted map} with volume $\sum_{f\in F} a_f. $ When $K$ is a subset of faces of $\Gbb$ we then write $a(K)=\sum_{f\in K}a(f)$ its volume.   When  $m=(\Gbb,a)$ and $m'=(\Gbb',a')$ are area weighted maps with faces set $F$ and $F'$,  $m'$ is finer than $m$ if $\Gbb'$ is finer than $\Gbb$ and $a_f=\sum_{f'\in F': f'\subset F}a'_{f'}.$ When $T>0,$ we denote by 
$$\Delta_{\Gbb}(T)=\{a:F\to\R_+: \sum_{f\in F}a_f=T  \}$$
the closed simplex of area vectors of fixed volume $T$ and its interior by
$$\Delta^{o}_\Gbb(T)=\{a\in \Delta_\Gbb(T): a(f)>0,\ \forall f\in F\}.$$
Its faces are given as follows. For any subset $K\varsubsetneqq F$, we set
$$\Delta_{K,\Gbb}(T)=\{a\in \Delta_\Gbb(T): a(f)=0,\ \forall f\in K\}$$
and 
$$\Delta_{K,\Gbb}^o(T)=\{a\in \Delta_{K,\Gbb}(T): a(f)>0,\ \forall f\in F\setminus K\} .$$
When $(\Gbb,B)$ is a map with boundary faces $B,$ we set
$$\Delta_{\Gbb,B}(T)=\{a:F\setminus B\to\R_+: \sum_{f\in F\setminus B}a_f=T  \}$$
and
$$\Delta^{o}_{\Gbb,B}(T)=\{a\in \Delta_{\Gbb,B}(T): a(f)>0,\ \forall f\in F\}.$$  

When $\Gbb'=(V',E')$ is finer than $\Gbb,$  any face $F$ of $\Gbb$ can be identified with a subset of faces of $\Gbb',$ and for any $a\in \Delta_{\Gbb'}(T),$ we denote $r_\Gbb(a)\in \Delta_{\Gbb}(T)$ the associated area vector of  $\Gbb.$ We then say that the area weighted map $(\Gbb',a)$ is finer than $(\Gbb,r_\Gbb(a)).$

\vspace{0.5em}

\emph{Multiplicative functions and Wilson loops:}
Given a map $\Gbb=(V,E,F)$  and a compact group $G$, we say that a function $h:\mathrm{P}(\Gbb)\to G$ is \emph{multiplicative} if for any pair of paths $\g_1,\g_2$ with $\ov\g_1=\und{\g}_2,$
\begin{equation}
h_{\g_1\g_2}=h_{\g_2}h_{\g_1}.\label{eq:Multip}
\end{equation}
We denote their set by $\mathcal{M}(\mathrm{P}(\Gbb),G).$ Endowing it with pointwise multiplication, it is a compact group and fixing an orientation of the edges, the evaluation on these edges defines an isomorphism
$$\mathcal{M}(\mathrm{P}(\Gbb),G)\simeq G^E.$$
The Haar measure on $\mathcal{M}(\mathrm{P}(\Gbb),G)$ can be identified via this isomorphism to the tensor product of  the Haar measure on $G$, we denote 
it simply by $dh.$

\vspace{0.5em}

When $\Gbb'$ is a map finer than $\Gbb$, the restriction from $\mathrm{P}(\Gbb')$ to $\mathrm{P}(\Gbb)$ defines a map 
$$\mathcal{R}_{\Gbb}^{\Gbb'}:\mathcal{M}(\mathrm{P}(\Gbb'),G)\to \mathcal{M}(\mathrm{P}(\Gbb),G).$$ 

A \emph{Wilson loop} is a function of the form 
$$\begin{array}{cl} \mathcal{M}(\mathrm{P}(\Gbb),G)&\lto \C \\ h&\longmapsto \chi(h_{\ell})\end{array}$$
where $\chi:G\to \C$  is a function invariant by conjugation and $\ell\in \mathrm{L}(\Gbb).$  By centrality, the value $\chi(h_{\ell})$  depends on $\ell$ only through its $\sim_c$-equivalence class $l$ and we denote it by $\chi(h_l)$. When $G_N$ is a compact classical group, for any loop $\ell\in\mathrm{L}(\Gbb),$ we shall focus on the Wilson loop $W_\ell$ obtained considering as central function 
$$\chi=\tr_N,$$
where $\tr_N=d_N^{-1}\Tr$ is the standard trace $\Tr$ in the natural matrix representation normalised by the size $d_N$ of the matrix, that is $2N$ in the 
symplectic case and $N$ otherwise.

\vspace{0.5em}

\emph{Discrete Yang--Mills measure, non-singular case on closed surfaces:} When $T>0,$ $\Gbb$ is a map with boundary faces $B$ and 
$a\in \Delta^{o}_{\Gbb,B}(T), $ the Yang--Mills measure   is the probability measure $\YM_{\Gbb,B,a}$ on  the compact group 
$\mathcal{M}(\mathrm{P}(\Gbb),G)$ with density 
$$Z_{\Gbb,B,a}^{-1}\prod_{f\in F\setminus B} p_{a_f}(h_{\pl f})$$
with respect to the Haar measure on $\mathcal{M}(\mathrm{P}(\Gbb),G),$ where $Z_{\Gbb,B,a}=1$ if $B\not=\emptyset$ and
$$Z_{\Gbb,a}=\int_{\mathcal{M}(\mathrm{P}(\Gbb),G)} \prod_{f\in F} p_{a_f}(h_{\pl f})dh$$
otherwise. In the above formula, $\pl f$ is the boundary  of the face for some arbitrary choice of root and orientation. This does not change the value of
$p_{a_f}(h_{\pl f})$ thanks to \eqref{eq:Inv HK Conj} and \eqref{eq:Inv HK Inverse}. The fact that this density defines a probability measure when 
$B\not=\emptyset$ 
follows for instance from Lemma \ref{lem:integral_reducedloops} below. We denote  $\YM_{\Gbb,\emptyset,a}$  simply by  $\YM_{\Gbb,a}.$

\begin{lem} \begin{enumerate}
\item For any $a\in \Delta_{\Gbb}^o(T),$ the constant $Z_{\Gbb,a}$ depends only on $T$ and the genus $g$ of $\Gbb,$ we denote it by $Z_{g,T}.$
\item When $m'=(\Gbb',a'),m=(\Gbb,a)$ are two area weighted maps with $m'$ finer than $m$ and  $a'\in \Delta^{o}_{\Gbb'}(T),$ then 
$${\mathcal{R}^{\Gbb'}_{\Gbb}}_*(\YM_{\Gbb',a'})=\YM_{\Gbb,a}.$$
\end{enumerate}
\end{lem}

\vspace{0.5em}

\emph{Uniform continuity and compatibility:} The Yang--Mills measure is also well defined on the faces on the simplex of area vectors. For any $r,g\ge 1$ let us consider the set $ \mathrm{Hom}(\G_{g,r},G)$ of group morphisms. When endowed with point-wise multiplication it is a compact group and thanks to the  presentation of Lemma \ref{__Lem:Basis Reduced Loops}, 
$$\mathrm{Hom}(\G_{g,r},G)\simeq G^{r+2g-1}.$$
Moreover, this presentation allows to write the following integration formula.

\begin{lem}[\cite{Lev2}]\label{lem:integral_reducedloops}
Assume  $(\Gbb,a)$ is an area weighted map with $r$ faces, and  $(\ell_i,1\le i\le r)$ and $a_1,b_1,\ldots,a_g,b_g$ are as in Lemma \ref{__Lem:Basis Reduced Loops}. For any $1\le i\le r$, denote by $a_i$ the area of the face of $\ell_i.$ Then for any continuous function 
$\chi: G^{2g+r}\to \C$ and any $a\in \Delta^{o}_\Gbb(T)$ and $1\le k\le r,$
\begin{align*}
\E_{\YM_{\Gbb,a}} &(\chi(h_{\ell_1},\ldots,h_{\ell_r},h_{a_1},\ldots,h_{b_g}))\\ &= Z_{g,T}^{-1}\int_{G^{2g+r-1}} \chi(z_1,\ldots,z_r,x_1,\ldots ,y_g )p_{a_{k}}(z_k) 
\prod_{i=1,i\not=k}^{r} p_{a_i}(z_i)dz_i \prod_{l=1}^g dx_ldy_l,
\end{align*}
where we set $z_{k}=(z_1\ldots z_{k-1})^{-1}[a_1,b_1]\ldots [a_g,b_g](z_{k+1}\ldots z_{r})^{-1}.$ When $B$ is a non-empty subset of faces of $\Gbb$ and  lassos with faces in its complement  have labels $i_1,\ldots,i_{p},$  
\begin{align*}
\E_{\YM_{\Gbb,B,a}} &(\chi(h_{\ell_{i_1}},\ldots, h_{\ell_{i_p}},h_{a_1},\ldots,h_{b_g}))\\ &=\int_{G^{2g+p}} \chi(z_1,\ldots,z_p,x_1,\ldots ,y_g )
\prod_{i=1}^{p} p_{a_i}(z_i)dz_i \prod_{l=1}^g dx_ldy_l.
\end{align*}
\end{lem}
The above expression  yields the following continuity in the area parameter. For any vertex $v$ of a map $\Gbb$, the restriction of a multiplicative function to loops based at $v$ depends only on the $\sim_r$-class of a loop and  the restriction operation defines a map $\mathcal{R}_v:\mathcal{M}(\mathrm{P}(\Gbb),G)\to \Hom(\mathrm{RL}_v(\Gbb,G))$. For any $a\in \Delta_{\Gbb}^o(T),$ we set  $\YM_{a,\Gbb,v}={\mathcal{R}_v}_*(\YM_{\Gbb,a}).$ Using the weak convergence of the heat kernel  \eqref{eq:heat_kernel}, we directly deduce the following result.

\begin{lem}   \label{__Lem:YM Faces simplex} The family of measures  $(\YM_{a,\Gbb,v},a\in \Delta_{\Gbb}^o(T))$ on $ \Hom(\mathrm{RL}_v(\Gbb),G)$ has a weakly continuous extension to $\Delta_{\Gbb}(T).$ It has the following properties.
\begin{enumerate}
\item Consider  $K\subset F$ with $K\not= F$, let $S\subset\{1,\ldots ,r\}$ be the labels of the lassos with meander in $F\setminus K$ and set $s=\#S$. 
Then for any $a\in \Delta_K^o(T)$, any continuous function $\chi:G^{2g+r}\to\C$ and $k\in S,$
\begin{align*}
&\E_{\YM_{\Gbb,a,v}} (\chi(h_{\ell_1},\ldots,h_{\ell_r},h_{a_1},\ldots,h_{b_g}))\\ 
&= \frac{1}{Z_{g,T}}\int_{G^{2g+r-1}} \chi(z_1,\ldots,z_r,x_1,\ldots ,y_g )p_{a_{k}}(z_k) \prod_{i\in S,i\not=k}p_{a_i}(z_i)dz_i \prod_{l=1}^g dx_ldy_l,
\end{align*}
where we set $z_{k}=(z_1\ldots z_{k-1})^{-1}[a_1,b_1]\ldots [a_g,b_g](z_{k+1}\ldots z_{r})^{-1}$ and  $z_i=1$ for all $i\not\in S.$
\item Consider a weighted map $(\Gbb',a')$ finer than  $(\Gbb,a)$  and denote the restriction map $\mathcal{R}^{\Gbb'}_{\Gbb}:\Hom(\mathrm{RL}_v(\Gbb',G))\to \Hom(\mathrm{RL}_v(\Gbb,G))$. Then, 
$${\mathcal{R}^{\Gbb'}_{\Gbb}}_*(\YM_{\Gbb',a',v})=\YM_{\Gbb,a,v}.$$
\item Consider $K\subset F$ with $K\not= F$ and  $a\in \Delta_K(T).$ Then for any loops $\ell,\ell'\in \mathrm{RL}_v(\Gbb)$ with $\ell\sim_K\ell',$ 
$h_\ell$ and $h_{\ell'}$ have same law under $\YM_{\Gbb,a,v}.$ 
\end{enumerate}

\end{lem}

\emph{Continuous Yang--Mills measure:} \label{section------ continuous YM statements}Thanks to the invariance by subdivision of the discrete Yang--Mills measure, given a Riemannian metric it is possible to take the projective limit of measures defined on graphs embedded in $\Sigma$ whose edges are piecewise geodesic. It allows to define a multiplicative random process $(H_{\g})_{\g}$ indexed by all piecewise geodesic paths, whose marginals are given by the discrete Yang-Mills measure.

\vspace{0.5em}

This was done in  \cite{Lev2}, where the author is furthermore able to show a weak convergence result allowing to define uniquely the distribution of a multiplicative function $(H_{\g})_{\mathrm{P}(\Sigma)}$ indexed by all path of finite length.  Let us recall this result.

\vspace{0.5em}

Denote by $\mathrm{P}(\Sigma)$ the set of Lipschitz functions $\g:[0,1]\to \Sigma$ with speed bounded from above and from below, considered up to bi-Lipshitz re-parametrisations of $[0,1]$.  The set $\mathrm{P}(\Sigma)$ is endowed with the starting and endpoint maps, $\g \mapsto  \und{\g},\ov{\g}$ and of the operations of concatenation and reversion as above. A path of $\Sigma$ is an element of $\g\in\mathrm{P}(\Sigma).$ It is \emph{simple} if for any parametrisation $p:[0,1]\to\Sigma$, $p:[0,1)\to\Sigma$ is injective. We consider then the set 
$$\mathcal{M}(\mathrm{P}(\Sigma),G)$$
of multiplicative functions as in \eqref{eq:Multip}.  It is a compact subset of  $G^{\mathrm{P}(\Sigma)}$ when the latter is endowed with the product topology. A loop is a path $\ell\in \mathrm{P}(\Sigma)$ such that $\und{\ell}=\ov \ell$. We denote their set by $\Ld(\Sigma).$ For any $x,y\in\Sigma $, we endow  
$\mathrm{P}_{x,y}(\Sigma)=\{\g\in \mathrm{P}(\Sigma):\und \g=x,\ov\g=y\}$ with a metric setting for any $\g_1,\g_2\in\mathrm{P}_{x,y}(\Sigma),$
$$d(\g_1,\g_2)= \inf_{ p_1,p_2} \|p_1-p_2\|_\infty+ |\mathscr{L}(\g_1)-\mathscr{L}(\g_2)|$$
where the infimum is taken over all parametrisations $p_1,p_2$ of $\g_1,\g_2$ and for any $\g\in\mathrm{P}(\Sigma)$, $\mathscr{L}(\g)$ denotes the Riemannian length of $\g.$ Endowing $\mathcal{M}(\mathrm{P}(\Sigma),G)$ with the cylindrical sigma field $\mathcal{B}_{\Sigma,G}$, we denote by $(H_{\g})_{\g\in \mathrm{P}(\Sigma)}$ the canonical process. 
When $G=G_N$ is a classical compact matrix  Lie group of size $N,$ we write for any path $\g\in \mathrm{P}(\Sigma)$, 
$$W_\g=\tr_N(H_\g).$$ 
When $(\Gbb,a)$ is an area weighted map of genus $g\ge 0,$ an \emph{embedding} of $(\Gbb,a)$ in a Riemann surface with volume $\vol$, is a collection of simple paths $(\g_e)_{e\in E}$ of $\Sigma$ indexed by edges of $\Gbb$, that do not cross but at their endpoints with the following properties:
\begin{enumerate}
\item The ranges of all paths $(\g_e)_{e\in E}$ form the $1$-cells of a  CW complex isomorphic to the CW complex of $\Gbb.$  
\item Fixing such an isomorphism, each $2$-cell of the complex associated to $(\g_e)_{e\in E}$ is a subset of $\Sigma$ of Riemannian volume $a(f)$, whenever it is identified with a face $f$ of $\Gbb.$ 
\end{enumerate}
When $\Sigma$ is the Euclidean plane or the hyperbolic disc, while $\Gbb$ is a map of genus $0,$ $f_\infty$ is a face of $\Gbb$ and $a\in\Delta_{\Gbb,\{f_\infty\}}(T),$  an embedding in $\Sigma$ of the 
area weighted map $(\Gbb,\{f_\infty\},a)$ with one boundary component  is a collection of simple paths 
$(\g_e)_{e\in E}$ of $\R^2$ indexed by edges of $\Gbb$, that do not cross but at their endpoints with the  following properties: 
\begin{enumerate}
\item The ranges of all paths $(\g_e)_{e\in E}$ form the $1$-cells of a  CW complex isomorphic to the CW complex of $\Gbb,$ such that the unique unbounded $2$-cell is mapped to $f_\infty.$  
\item Fixing such an isomorphism, each bounded $2$-cell of the complex associated to $(\g_e)_{e\in E}$ is a subset of $\Sigma$ of Riemannian volume $a(f)$, whenever it is identified with a face $f$ of $\Gbb.$ 
\end{enumerate}
In each case, we say that $\Gbb$ is embedded in $\Sigma$ if there is an area vector $a$ satisfying the  property 2.

When $\Gbb=(V,E,F)$ is a map, $\ell\in \Ld(\Gbb)$, $\Sigma$ is a two-dimensional Riemannian manifold and $\ell\in \Ld(\Sigma)$, we say that $\ell$ is a \emph{drawing} of $\ell=e_1\ldots e_n$ if there is an embedding  $(\g_e)_{E\in E} $ of $\Gbb$ into $\Sigma,$ such that $\ell$ is the concatenation $\g_{e_1}\ldots \g_{e_n}.$  The next two theorems are due to L\'evy \cite{Lev2}.  

\begin{thm}
Let $\Sigma$ be a compact Riemannian surface with area measure $\vol$, $G$ a  fixed compact Lie group such that $\mathfrak{g}$ is endowed with a $G$-invariant inner product. There  exists a unique measure $\YM_{\Sigma}$ on $(\mathcal{M}(\mathrm{P}(\Sigma),G), \mathcal{B}_{\Sigma},G)$, with  following properties.
\begin{enumerate}
\item  If  $(\g_e)_{e\in E}$ is an embedding in $\Sigma$ of an area-weighted map  $(\Gbb,a)$ with edges  $E$, the distribution of $(H_{\g_e})_{e\in E}$ is the discrete Yang--Mills measure $\YM_{\Gbb,a}$.
\item For any $x,y\in\Sigma,$ if $(\gamma_n)_{n\ge 1}$ is a sequence of paths of $\mathrm{P}_{x,y}(\Sigma) $ with $\lim_{n\to \infty}d(\g_n,\g)=0$ for some $\gamma\in\mathrm{P}(\Sigma),$ then under $\YM_{\Sigma},$ the sequence of random variables $(H_{\gamma_n})_{n\ge 1}$ converges in probability to $H_\gamma$.
\end{enumerate}
The process $(H_\gamma)_{\gamma\in\mathrm{P}(\Sigma)}$ is called the \emph{Yang--Mills holonomy process}. 
\end{thm}

\begin{thm}
Let $\Sigma$ be a Euclidean plane $\R^2$ or the hyperbolic disc $D_\mfh,$ endowed with their area measure $\vol$, $G$ a  fixed compact Lie group such that $\mathfrak{g}$ is endowed with a $G$-invariant inner product. There a  exists a measure $\YM_{\Sigma}$ on $(\mathcal{M}(\mathrm{P}(\Sigma),G), \mathcal{B}_{\Sigma},G)$, with  following properties.
\begin{enumerate}
\item If  $(\g_e)_{e\in E}$ is an embedding in $\Sigma$ of an area-weighted map of genus $0$ with one boundary $(\Gbb,\{f_\infty\},a)$ and edge set $E$, the distribution of $(H_{\g_e})_{e\in E}$ is the discrete Yang--Mills measure 
$\YM_{\Gbb,a}$.
\item For any $x,y\in\Sigma,$ if $(\gamma_n)$ is a sequence of paths of $\mathrm{P}_{x,y}(\Sigma) $ with $d(\g_n,\g)\underset{{n\to \infty}}\to0$ for some $\gamma\in\mathrm{P}(\Sigma),$ then under $\YM_{\Sigma},$ the sequence of random variables $(H_{\gamma_n})_{n\in\N}$ converges in probability to $H_\gamma$.
\end{enumerate}
The process $(H_\gamma)_{\gamma\in\mathrm{P}(\Sigma)}$ is called the \emph{Yang--Mills holonomy process}.
\end{thm}

The first author showed with C\'ebron, Gabriel and Norris showed in \cite{CDG,DN} that the proof of the above theorem can be adapted to yield the following extension result, when $G$ is allowed to vary. Let us denote by $A(\Sigma)$ the subset of paths of $\mathrm{P}(\Sigma)$ with a piecewise geodesic bi-Lipschitz parametrisation.

\begin{prop}\label{-->:Prop Extension Piecewise Geo} Let  $(G_N)_N$ be a sequence of compact classical groups. Assume the  following two properties.
\begin{enumerate}
\item For any $\g\in A(\Sigma),$  $\Phi(\g)=\lim_{N\to\infty} W_\g$ where the convergence holds in probability under $\YM_\Sigma$ and $\Phi(\g)$ is constant. 
\item There is a constant $K>0$ independent of $N,$ such that for any simple contractible loop $\ell\in \Ld(\Sigma)$ bounding an area $t>0,$ 
$$\E_{\YM_\Sigma}[1-\Re(W_\ell) ] \le K t. $$
\end{enumerate}
Then $\Phi: A(\Sigma)\to \C$ has a unique extension to $\mathrm{P}(\Sigma)$ such that for all $x,y\in\Sigma$,  $\Phi:\mathrm{P}_{x,y}(\Sigma)\to\C$ is continuous and for any $\g\in \Ld(\Sigma),$ $W_\g$ converges in probability towards $\Phi(\g)$ as $N\to\infty.$  
\end{prop}
The argument  given in section 5 of \cite{DN} for the sphere applies verbatim on any compact surface $\Sigma$ to yields the above statement, we will not repeat it in the current version. The same applies for the following lemma.

\begin{lem}  \label{__Lem:Back to regular path}
\begin{enumerate}
\item For any map $\Gbb$ there is a regular graph $\Gbb'$ finer than $\Gbb.$ 
\item For any $\g\in A(\Sigma)$ there is an embedded graph $\Gbb$ such that $\g$ is the drawing of a path of $\Gbb.$
\item For any area weighted map $(\Gbb,a)$ and $\g\in \mathrm{P}(\Gbb)$, there is a regular area weighted map $(\Gbb',a')$ finer than $(\Gbb,a)$, $\g'$ a regular path of $\Gbb'$ and $K$ a subset of faces of $\Gbb'$, such that 
$$ \g\sim_{K}\g'.$$
\item For any compact Lie group $G$ and  any $\g\in A(\Sigma),$ there is a regular path $\mathfrak p$ in a regular graph $\Gbb$ and $a\in \Delta_{\Gbb}(T)$ such that under $\YM_{\Sigma},$ $W_\g$ has same law as $W_{\mathfrak{p}}$ under $\YM_{\Gbb,a}.$
\end{enumerate}
\end{lem}
Together with the last proposition, this lemma reduces the study of Wilson loops for all loops of finite length to the case of regular loops.

\subsection{Planar master field, main results and conjecture}

\label{-----sec:WL MF}

\label{----sec: Statements}

In the above setting, the  following was proved in \cite{Lev} and\footnote{In \cite{Lev},  to get uniqueness (b) is replaced by an additional set of differential equations} \cite{Hal2}, see 
\cite{Xu,AS} for a weaker statement with a smaller class of loops and of groups $G_N$.  Recall the definition of the de-singularisation operation in section 
\ref{----sec:Vertex desingularisation}. 

\begin{thm} \label{-->THM Conv Planar Master Field}Assume that $G_N$ is a compact classical group of rank $N$. Assume that $(\Gbb,\{f_\infty\},a)$ is  any area weighted map of genus 
$0$, with one  boundary component and $\ell\in \Ld(\Gbb),$ or that $\ell\in \Ld(\R^2).$ Then the following convergences hold in probability\footnote{It is also shown in \cite{Lev} that the following
convergences are almost sure.}   and the limits are constant and independent 
of the type of series of $G_N$:
$$\Phi_{\ell}^{f_\infty}(a)=\lim_{N\to\infty}W_\ell \text{ under }\YM_{\Gbb,\{f_\infty\},a }$$
and 
$$\Phi_{\R^2}(\ell)=\lim_{N\to\infty}W_\ell \text{ under }\YM_{\R^2}.$$
The function $\Phi_{\R^2}$ is characterised by the following properties:
\begin{enumerate}
\item For any $x\in \R^2,$ $\Phi_{\R^2}: \mathrm{P}_{x,x}(\R^2) \to \C$ is continuous. 
\item Whenever $\ell\in \Ld(\R^2)$ is a drawing of a loop $\ell $ of an area weighted  map of genus $0$ with one boundary component $(\Gbb,\{f_\infty\},a),$
$$\Phi_{\R^2}(\ell)=\Phi_{\ell,f_\infty}(a).$$
\item For any map  of genus $0$  with one boundary component  $(\Gbb, \{f_\infty\})$, $T>0$, and any loop $\ell\in \Ld(\Gbb),$ $\Phi_\ell$  is uniformly continuous on $\Delta_{\Gbb,\{f_\infty\}}(T)$ and differentiable on $\Delta_{\Gbb,\{f_\infty\}}^o(T)$  such that 
\begin{enumerate}
\item if  $\Gbb$ is regular, $\ell$ is a tame loop and $v\in V_\ell$ is a transverse  intersection with $\delta_v \ell=\ell_1\otimes \ell_2,$
$$\mu_v.\Phi_{\ell,f_\infty}=\Phi_{\ell_1,f_\infty}\Phi_{\ell_2,f_\infty} \text{ in }\Delta_{\Gbb,\{f_\infty\}}^o(T).$$
\item Whenever  $\ell$ is the boundary of a topological disc of area $t$,
$$\Phi_{\R^2}(\ell)=e^{-\frac t 2}.$$ 
\end{enumerate}
\end{enumerate}
\end{thm}
See the appendix of \cite{Lev} for a table of values of $\Phi_{\R^2}.$ Alternatively, the master field can be characterised using free probability as follows. For any real $t\geq 0$ and any integer $n\geq 0$, set
\[
\nu_t(n) = e^{-\frac{nt}{2}}\sum_{k=0}^{n-1}\frac{(-t)^k}{k!}n^{k-1}{{n}\choose{k+1}}.
\]
It is known since the work of Biane \cite{Bia} that these quantities are related to the limits of the moments of Brownian motions on $\U(N)$, and L\'evy proved in \cite{Lev} that it is still the case for the other compact classical matrix Lie groups.

\begin{lem} \label{__Lem: Conv Planar MF Maps} Consider an area weighted map $(\Gbb,\{f_\infty\},a)$ of genus $0$ with one boundary component. Assume $\Gbb=(V,E,F)$, $\#F=r+1$, $F= \{f_1,\ldots,f_{r},f_\infty\}$    
and  $v\in V.$  
For any $\ell\in \Ld_v(\Gbb)$ depends on $\ell$ only through its $\sim_r$ class. Setting 
$$\tau_v(\ell)=\Phi_{\ell,f_\infty}(a) \text{ and }\ell^*=\ell^{-1},\ \forall  \ell\in \mathrm{RL}_v(\Gbb)$$
and extending these maps linearly and sesquilinearly, defines a non-commutative probability space $(\C[\mathrm{RL}_v(\Gbb)],\tau_v, *)$.  Assume that $\ell_1,\ldots,\ell_r,\ell_\infty$ is 
a family of lassos as in Lemma  \ref{__Lem:Basis Reduced Loops} with $\ell_i$ bounding $f_i$ for $1\le i\le r $ and $\ell_{\infty}$ for 
$f_\infty.$ Then $\tau_v$ is the unique state on 
$(\C[\mathrm{RL}_v(\Gbb)], *)$ such that 
\begin{enumerate}
\item for all $n\in \Z^*,$ $\tau_v(\ell_i^n)=\nu_{a(f_i)}(n),$
\item $\ell_1,\ldots,\ell_r$ are freely independent under $\tau_v.$ 
\end{enumerate}
\end{lem}

Similarly the following  lemma follows from the classical result of \cite{Bia} and Lemma \ref{lem:integral_reducedloops}. It shows that the conclusion of the former one is 
valid when the genus condition is dropped. 

\begin{lem} \label{__Lem: Convergence MF One BD}Consider an area weighted regular map with boundary $(\Gbb,\{f_\infty\},a)$ of genus $g\ge 1$.  Assume $\Gbb=(V,E,F)$, $\#F=r+1$ with  
$F=\{f_1,\ldots,f_r,f_\infty\}$ and  $v\in V.$    
Assume that $a_1,\ldots, b_g$  and $\ell_1,\ldots,\ell_{r+1}$ are $2g$ simple loops  and 
$r+1$ lassos  as in Lemma  \ref{__Lem:Basis Reduced Loops}, with $\ell_i$ bounding $f_i$ for $1\le i\le r$ and $f_\infty$ for $i=r+1.$  
Assume that $G_N$ is a sequence of compact classical matrix Lie groups of size $N.$ Then for any $T>0,$ $a\in\Delta_{\Gbb,\{f_\infty\}}(T)$ and   
$\ell\in \mathrm{RL}_v(\Gbb)$,  
$$W_\ell \to \Phi_\ell^{1,g}(a) \text{ under } \YM_{\Gbb,\{f_\infty\},a},$$
where $\Phi^{1,g}_\ell(a)$ is constant. Moreover there is a constant  $K>0$ independent of $\Gbb$ and $N\ge 1,$ such that for any face $f\in F\setminus \{f_\infty\},$ 
\begin{equation*}
\E[1-\Re (W_{\pl f})]\le K a(f).\tag{*}
\end{equation*}
The $*$-algebra $(\C[\mathrm{RL}_v(\Gbb)],*)$ is endowed with a unique state $\tau_{v}$ satisfying  
$$\tau_v(\ell)=\Phi^{1,g}_{\ell}(a),\ \forall \ell\in \mathrm{RL}_v(\Gbb).$$
Moreover, $\tau_v$ is characterised by the following three properties:
\begin{enumerate}
\item $\ell_1,\ldots,\ell_r,a_1,\ldots, b_g$  are freely independent under $\tau_v.$ 
\item under $\tau_v,$ $a_1,\ldots ,b_g$ are $2g$ Haar unitaries.
\item for any $1\le i\le r$ and $n\in \Z^*,$ 
$$\tau_v(\ell_i^n)= \nu_{a(f_i)}(n).$$ 
\end{enumerate}
\end{lem}
A sketch of the proof is given in section \ref{-------sec: Convergence after surgery}.

\vspace{0.5em}

From  Lemma \ref{__Lem: Convergence MF One BD} and the absolute continuity result of \cite{DL}  follows  the corollary \ref{coro---off handle}, for loops avoiding at least one handle. Let us give now a discrete reformulation of corollary \ref{coro---off handle}.  Its proof is  given below in section \ref{-------sec: Convergence after surgery}. Let us recall the definition of the universal cover $\tilde \Gbb=(\tilde V,\tilde E,\tilde F)$ of a regular graph $(\Gbb,\Gbb_b)$ given in section \ref{----sec:DiscreteFundamentalCover}, with 
a canonical covering map $p:\tilde F\to F.$  When $a\in \Delta_\Gbb(T),$ let us set  $\tilde a= a\circ p: \tilde F\to [0,T].$   

\begin{thm}\label{-->THM: CUT Conv}  \label{-->TH: CV MF Disc} Assume that $(\Gbb,a)$ is an  area weighted map  cut along a simple loop  $\ell\in\Ld(\Gbb)$ given by $(\Gbb_1,\{f_{1,\infty}\})$ and
$(\Gbb_2,\{f_{2,\infty}\})$, with 
the 
same 
convention as in section \ref{----sec:Maps}.  Assume that
$\Gbb_2$ 
has genus $g_2\ge 1.$ Then,  for any loop $\ell\in \Ld(\Gbb_1)$ and $a\in \Delta_\Gbb(T)$ with $0<\sum_{f\in F_2}a(f)<T$, 
\begin{equation}
W_\ell\underset{N\to\infty}\to \Phi_{\ell}(a)= \left\{\begin{array}{ll}\Phi_{\tilde\ell}(\tilde a) & \text{ if }\ell\sim_h c_{\und\ell}, \\&\\ 0& \text{ if }\ell\not\sim_h c_{\und\ell}, \end{array}\right.\label{eq:Def MF Funda Cover} \text{ in 
probability under }\YM_{\Gbb,a},
\end{equation}
where $\tilde \ell$ is a lift of $\ell$ in $\tilde \Gbb.$
Moreover, when $g_2\ge 2,$ the convergence holds true uniformly in $a\in \Delta_\Gbb(T).$
Besides, there is a constant  $K>0$  independent of $\Gbb$ and $N\ge 1,$ and depending only on $a(F_2)\in (0,T)$  such that for any face $f\in F_1,$ 
\begin{equation}
\E[1-\Re (W_{\pl f})]\le K a(f). \label{eq:Bound area YM Bd}
\end{equation}
\end{thm}  

When $\Gbb$ has genus $1$ the above result gives information about loops included in a topological disc but does not say anything about other loops, for instance  contractible 
loops obtained by concatenation of simple loops of non trivial homology. A more satisfying  answer is then given by the following theorem.

\begin{thm}\label{-->THM: Torus} Consider a compact classical group $G_N$ of rank $N$, a torus $\Tbb_T$ of volume $T>0$ obtained as a quotient of the Euclidean plane $\R^2$ by the lattice $\sqrt{T}\Z^2$. Then, the following convergence holds in probability under $\YM_{\Tbb_T},$
$$W_\ell\underset{N\to\infty}\to  \Phi_{\Tb_T}(\ell)=\left\{\begin{array}{ll} \Phi_{\R^2}(\tilde \ell) & \text{if }\ell\text{ is contractible,} \\ &\\  0 &\text{ otherwise,} \end{array}\right. $$
where for any   loop $\ell\in \Ld(\Tbb_T) ,$ $\tilde\ell\in \mathrm{P}(\R^2)$ is a finite length path with projection to $\Tbb_T$ given by  $\ell.$ Besides, $\Phi_{\Tbb_T}: \Ld(\Tbb_T)\to \C$  
is the unique function satisfying
\begin{enumerate}
\item For any $x\in \Tbb_T,$ $\Phi_{\Tbb_T}: \Ld_x(\Tbb_T)\to \C $ is continuous for the length metric $d.$ 
\item  For any regular loop $\ell$ in a regular map $\Gbb$ of genus $1,$ there is a differentiable function 
$$\Phi_{\ell}:\Delta_\Gbb(T)\to \C$$
such that for any transverse intersection $v\in V_\ell,$ with $\delta_v\ell=\ell_1\otimes \ell_2,$
\begin{equation}
\mu_v\Phi_\ell=\Phi_{\ell_1}\Phi_{\ell_2},\label{eq:MM Torus}
\end{equation}
and such that $\Phi_\ell(a)=\Phi_{\Tbb_T}(\ell)$ if the weight $a$ corresponds to the area measure on $\Tbb_T$.
\item For any  loop $\ell\in \Ld(\Tbb_T)$ obtained by projection of a loop $\tilde\ell\in \Ld(\R^2)$ included in a fundamental domain of $\Tbb_T$, 
$$\Phi_{\Tbb_T}(\ell)=\Phi_{\R^2}(\tilde \ell).$$
\item For any non-contractible simple loop $\ell \in \Ld(\Tbb_T^2)$ and $n\in \Z^*,$
$$\Phi_{\Tbb_T}(\ell^n)=0.$$ 
\end{enumerate}
\end{thm}

When $g\ge 2,$ we were unable to show a satisfying version of conjecture \ref{conj_Lift}, but are able to prove the following conditional results.

\begin{thm} \label{-->THM Wilson Loops Higher Genus} Consider a compact classical group $G_N$  of rank $N$,  $g\ge 2$ and $T>0$. Assume that for any regular area weighted map $(\Gbb,a)$ of genus $g,$
\begin{equation}
W_\ell\underset{N\to \infty}\to \Phi_{\tilde \ell}(\tilde a) \text{ in probability under }\YM_{\Gbb,a},\label{eq:Conv Discrete Statement}
\end{equation}
whenever $\ell\in\Ld( \Gbb)$ such that
\begin{enumerate}
\item any lift $\tilde\ell\in \Ld(\tilde\Gbb)$ of $\ell$ is included in a fundamental domain, or
\item $\ell= \g_{nest} \g,$ where $\g_{nest}$  is a nested loop and $\g$ is a geodesic path.\footnote{See Fig. \ref{Fig----Nested loops} and Section \ref{----sec:DiscreteFundamentalCover}.} 
\end{enumerate}
Then for any regular map $\Gbb$ of genus $g$,  \eqref{eq:Conv Discrete Statement} holds true for all $\ell\in \Ld(\Gbb).$ \end{thm}

Besides, the following weaker statement can be proved independently.

\begin{prop} \label{-->Prop Wilson Loops Non zero Homol}Consider a compact classical group $G_N$  of rank $N$ and $g\ge 2$. Assume that for any regular area weighted map $(\Gbb,a)$ of genus $g,$
\begin{equation}
W_\ell\underset{N\to \infty}\to 0 \text{ in probability under }\YM_{\Gbb,a},\label{eq:Vanishing Discrete Statement}
\end{equation}
whenever $\ell\in\Ld( \Gbb)$ is a geodesic  loop with non zero-homology. Then for any regular map $\Gbb$ of genus $g$, \eqref{eq:Vanishing Discrete Statement} holds true for all 
$\ell\in \Ld(\Gbb) $ with non-zero 
homology. \end{prop}

\begin{rmk} The above statements may give the impression that any possible master field is expressed in terms of the planar case. This is nonetheless not the case as the Wilson loops  on the sphere converge to different limits \cite{DN}.  See also the discussion in \cite[Section 2.5]{DL}. 
\end{rmk}

The proofs of Theorems \ref{-->THM: Torus} , \ref{-->THM Wilson Loops Higher Genus} and Proposition \ref{-->Prop Wilson Loops Non zero Homol} are provided in the end of Section \ref{--------sec: MM Existence Uniqueness PB}.

\subsection{Invariance in law and Wilson loop expectation}

Before proceeding to the main part of this paper, let us give  a partial result that only holds in expectation, but relies on a simpler argument: the invariance in law by an action of the center of the structure group $G_N$. Consider a regular map $\Gbb=(V,E,F)$ with $r$ faces, $v\in V$ and a basis $\ell_1,\ldots,\ell_r,a_1,\ldots,b_g$ of the free group $\mathrm{RL}_v(\Gbb)$ as in Lemma \ref{__Lem:Basis Reduced Loops}.  For any  $h\in G^{2g}$ 
and $\phi\in\Hom(\mathrm{RL}_v(\Gbb),G)$, let us denote by 
$h.\phi\in\Hom(\mathrm{RL}_v(\Gbb),G)$ the unique group morphism with 
\[h.\phi (\ell_i)\mapsto \phi(\ell_i),\text{ for }1\le i\le r\] 
and 
\[ h.\phi(a_i)=h_{2i-1}\phi(a_i)\text{ and }h.\phi(b_i)=h_{2i}\phi(b_i) \text{ for }1\le i\le g.\]

Let us denote by $Z$ the center of $G.$ When $h\in Z^{2g},$  it follows easily from  point 2. of Lemma \ref{__Lem: Homology basis choice} that 
\begin{equation}
h.\phi(\ell)=\phi_{h}([\ell]_\Z)\phi(\ell),\ \forall \ell\in \mathrm{RL}_v(\Gbb),  \label{eq:Action Cohomol}
\end{equation}
where $\phi_h\in \Hom(H_1(d^*,\Z),Z)$ is the unique group morphism such that 
\[ \phi_h([a_i]_\Z)=h_{2i-1}\text{ and }\phi_h([b_i]_\Z)=h_{2i} \text{ for }1\le i\le g.\]

\begin{lem} Let  $\Gbb$ be regular  map, $T>0,$ $a\in \Delta_{\Gbb}(T)$. Denoting by $(H_\ell)_{\ell\in\mathrm{RL}_v(\Gbb)}$ the canonical $G$-valued random variable  on $\Hom(\mathrm{RL}_v(\Gbb),G),$  the following assertions hold true.
\begin{enumerate}
\item The measure $\YM_{a,\Gbb,v}$ on $\mathrm{Hom}(\mathrm{RL}_v(\Gbb),G)$ is invariant under the action of $Z^{2g}.$ 
\item Assume that  $\chi: G\to \C$ is continuous and $\a: Z\to \C$  is such that $\chi(z.h)=\a_\chi(z)\chi(h),\ \forall (z,h)\in Z\times G.$ Then

\begin{enumerate}
\item for any $h\in Z^{2g}$ and $\ell\in \mathrm{RL}_v(\Gbb),$
$$\E_{\YM_{a,\Gbb,v}}[\chi(H_{\ell})]= \a_\chi\circ\phi_{h}([\ell]_{\Z}) \E_{\YM_{a,\Gbb,v}}[\chi(H_{\ell})].$$
\item If   there is  $\phi\in  \Hom(H_1(d^*,\Z),Z)$ with $\phi([\ell]_\Z)\not=0,$ then 
\[\E_{{\YM_{a,\Gbb,v}}}[\chi(H_\ell)]=0.\]
\end{enumerate}
\item When $G$ is a classical compact matrix Lie group, for any $\ell \in\mathrm{RL}_v(\Gbb),$   $\E[W_\ell]= 0 $ if one of the following conditions is satisfied:
\begin{enumerate}
\item $G=\U(N)$ and $[\ell]_\Z\not=0.$
\item $G=\SU(N)$ and $[\ell]_{\Z_n}\not=0$
\item $G=\SO(2N)$ and $[\ell]_{\Z_2}\not=0.$ 
\end{enumerate}
\end{enumerate}
\end{lem}
\begin{proof}
The implication  $2.a)\Rightarrow 2.b)\Rightarrow 3$ are elementary. Thanks to \eqref{eq:Action Cohomol}, $1\Rightarrow 2.a).$  Lastly, consider 1. Denote by $d\phi$ the Haar measure on $\Hom(\mathrm{RL}_v(\Gbb),G)$ endowed with pointwise multiplication. By Lemma \ref{__Lem:YM Faces simplex}, it is enough to consider  $a\in \Delta_{\Gbb}^o(T)$ and denote by $a_1,\ldots a_r$ the area enclosed by the meanders of $\ell_1,\ldots,\ell_r$ and set $a_{r+1}=T-\sum_{i=1}^ra_i.$ For any continuous function $\chi:\Hom(\mathrm{RL}_v(\Gbb),G)\to\C$ and $h\in Z^{2g}$, $d\phi$ is invariant by the action of $Z^{2g}$ and
\begin{align*}
\int&_{\Hom(\mathrm{RL}_v(\Gbb),G)}\chi(h^{-1}.\phi){d\YM_{a,\Gbb,v}}(\phi)\\
&=\int_{\Hom(\mathrm{RL}_v(\Gbb),G)}\chi(h^{-1}.\phi) p_{a_{r+1}}(\phi((\ell_1\ldots\ell_r)^{-1}[a_1,b_1]\ldots[a_g,b_g])) \prod_{i=1}^r p_{a_i}(\phi(\ell_i)) d\phi\\
&= \int_{\Hom(\mathrm{RL}_v(\Gbb),G)}\chi(\phi) p_{a_{r+1}}(h.\phi((\ell_1\ldots\ell_r)^{-1}[a_1,b_1]\ldots[a_g,b_g])) \prod_{i=1}^r p_{a_i}(h.\phi(\ell_i)) d\phi\\
&=\int_{\Hom(\mathrm{RL}_v(\Gbb),G)}\chi(\phi) p_{a_{r+1}}(h.\phi((\ell_1\ldots\ell_r)^{-1}[a_1,b_1]\ldots[a_g,b_g])) \prod_{i=1}^r p_{a_i}(h.\phi(\ell_i)) d\phi,
\end{align*}
where in the last line we used that $h.\phi([a_i,b_i])=[\phi(a_i)h_{2i-1},\phi(b_i)h_{2i-1}])=\phi([a_i,b_i])$ for  $1\le i\le g$ and $h.\phi(\ell_j)=\phi(\ell_j),$ for $1\le j\le r.$
\end{proof}

\subsection{Makeenko--Migdal equations, existence and uniqueness problem}
\label{--------sec: MM Existence Uniqueness PB}

The main tool of the current article  are approximate versions of equations \eqref{eq:MM Torus}, satisfied on any surface when $G_N$ is a compact classical group and $N\to\infty.$  Let us introduce a  setting to 
prove existence and uniqueness of these equations.

\vspace{0.5em}

For any regular map $\Gbb$ and any vertex $v$ of $\Gbb$, let $\mathcal{A}_{v}(\Gbb)$ be the algebra with elements in $\C[\Ld_v(\Gbb)]$ endowed with the multiplication given by concatenation and setting $\ell^{*}=\ell^{-1}$ for 
all $\ell\in \Ld_{v}(\Gbb)$ and 
extending it skew-linearly. When $w$ is another vertex, we consider the $*$-algebra $\mathcal{A}_{v,w}(\Gbb)$ with elements in $\C[\Ld_v(\Gbb)] \otimes \C[\Ld_w(\Gbb)]$ and  multiplication and $*$-operation defined for all $(x_i,y_i) \in \mathcal{A}_v(\Gbb)\times\mathcal{A}_w(\Gbb)$ by
$$(x_1\otimes y_1).(x_2\otimes y_2)= (x_1x_2)\otimes (y_1y_2) \text{ and } (x_1\otimes y_1)^*=x_1^*\otimes y_1^*.$$

Let us fix $g\ge 1  $ and $T>0.$  A \emph{Wilson loop system} is a family of continuous functions  $\phi_{\ell_1},\phi_{\ell_1\otimes \ell_2}:\Delta_\Gbb(T)\to \C$ given for each  regular graph $\Gbb$ of genus $g$ and each pair of  loops  
$\ell_1,\ell_2\in \mathrm{L}(\Gbb)$, with the following properties:
\begin{enumerate}
\item For any constant loop $c,$
$$\phi_{\ell_1\otimes c} = \phi_{\ell_1} \text{ and } \phi_{c}=1. $$
\item For any pair of loops $\ell_1,\ell_2$ within a same regular graph of genus $g$,
$$\phi_{\ell_1\otimes \ell_2}=\phi_{\ell_2\otimes \ell_1}$$
depend on $\ell_1,\ell_2$ only through their $\sim_{r,c}$ equivalence class.
\item If $\Gbb'$ is finer than $\Gbb$ of genus $g,$ then for all loops $\ell,\ell_1,\ell_2\in \Ld(\Gbb)$ 
$$\phi_{\ell}\circ r_\Gbb^{\Gbb'}=\phi_{\ell}\text{ and }\phi_{\ell_1\otimes \ell_2}\circ r_\Gbb^{\Gbb'}=\phi_{\ell_1\otimes \ell_2}$$ 
where loops are identified in the right-hand-sides with elements of $\Ld(\Gbb').$  
\item  If $\Gbb'$ is isomorphic to $\Gbb$ of genus $g$, $a\in \Delta_\Gbb(T)$ is mapped to $a'\in\Delta_{\Gbb'}(T),$  while $\ell_1',\ell_2'\in\Ld(\Gbb)$ with $\ell_1\sim_\Sigma \ell'_1,\ell_2\sim_\Sigma\ell_2',$ through the same  isomorphism map, then  
$$\phi_{\ell_1}(a)=\phi_{\ell_1'}(a')\text{ and } \phi_{\ell_1\otimes \ell_2}(a)=\phi_{\ell_1'\otimes \ell_2'}(a').$$
\item If $\Gbb=(V,E,F)$ is a regular graph of genus $g$, $\ell_1,\ell_1',\ell_2\in \mathrm{L}(\Gbb),$ $K\subset F$ with $\ell_1\sim_K \ell_1',$ then  
$$\phi_{\ell_1\otimes \ell_2}(a)=\phi_{\ell_1'\otimes \ell_2}(a),\ \forall a\in \Delta_{K,\Gbb}(T).$$
\item For any regular graph $\Gbb$ of genus $g$ with vertex $v$, for any $a\in \Delta_\Gbb(T),$ extending $\ell \in\Ld_v(\Gbb) \mapsto\phi_{\ell}(a)$ 
linearly defines a non-negative states $\phi_{a,v}$ on $(\mathcal{A}_{v}(\Gbb),*)$ while for any $x\in\mathcal{A}_v(\Gbb),$
$$\phi_{x\otimes x^*}\ge 0.$$
\end{enumerate}
Whenever $G_N$ is a compact classical group, from the above definition of the Yang-Mills measure, the collection
$$a\in\Delta_{\Gbb}(T)\mapsto (\E_{\YM_{\Gbb,a}}[W_\ell],  \E_{\YM_{\Gbb,a}}[W_{\ell_1}W_{\ell_2}])$$
for all  regular maps $\Gbb$ of genus $g$ and loops $\ell,\ell_1,\ell_2\in \Ld(\Gbb),$ is a Wilson loop system. 

Moreover, it then follows from 1.   that for any vertex $v$ and  $\ell \in \Ld_v(\Gbb) $ $\ell$ has a unitary distribution in $(\cA_v,\phi_{v,a})$.  When $\phi$ is a Wilson loop system, for any regular graph $\Gbb$ and any loop $\ell\in \Ld(\Gbb),$  the second part of point 6. and point 1. yield
$$\scV_{\phi,\ell}=\phi_{\ell\otimes\ell^{-1}}-|\phi_{\ell}|^2= \phi_{\ell\otimes\ell^{-1}}-\phi_{\ell}\phi_{\ell^{-1}}\ge 0.$$

We say that a Wilson loop system $\phi$ is an \emph{exact solution} of Makeenko--Migdal equations if
\begin{enumerate}
\item   For any loop $\ell$ within a regular graph $\Gbb$ of genus $g,$ $\phi\in C^1(\Delta_\Gbb^o(T))$ and for any $v\in V_\ell,$ 
$$\mu_v. \phi_\ell=\phi_{\ell_1}\phi_{\ell_2}.$$
\item  For any pair of regular loops within the same graph, $\phi_{\a\otimes \b}=\phi_{\a}\phi_\b.$
\item For any regular loop $\ell$ with $\ell\not\sim_h c_{\und \ell},$  $ \phi_\ell=0. $ 
\end{enumerate}

We say that a sequence $(\phi^N)_{N\ge 1}$ of Wilson loop systems is an \emph{approximate solution of Makeenko--Migdal equations} if for any regular graph of genus $g$, any  loop $\ell$  in $\Ld(\Gbb)$, 
$\phi^N_{\ell}$ and $\scV^N_{\phi,\ell}$ are in $C^1(\Delta_\Gbb^o(T))$,  there is a constant $C>0$ independent of $\ell$ and $N\ge 1,$ such that for any intersection point $v\in V_\ell,$ 
\begin{equation}
|\mu_v.\phi^N_\ell- \phi^N_{\delta_v( \ell)}| \le \frac{C}{N},\label{eq:MM Wilson S E}
\end{equation}

\begin{equation}
|\mu_v.\scV_{\phi^N,\ell} |\le \scV_{\phi^N,\ell}+\scV_{\phi^N,\ell_1}+\scV_{\phi^N,\ell_2}+\frac{C}{N}  \label{eq:MM Wilson S Var}
\end{equation}
and 
\begin{equation}
|\mu_v.\scV_{\phi^N,\ell} |\le  \sqrt{\scV_{\phi^N,\ell_1}\scV_{\phi^N,\ell_2}}+|\phi_{\ell_1}^N| \sqrt{\scV_{\phi^N,\ell_2}}+|\phi_{\ell_2}^N| \sqrt{\scV_{\phi^N,\ell_1}}+\frac{C}{N}  \label{eq:MM Wilson S Var}
\end{equation}
where    $\ell_1\otimes\ell_2=\delta_v\ell$.

\begin{rmk} Note that it follows from point 3. that if  $\phi$ is a Wilson loop system and $\ell,\ell_1,\ell_2$ are a regular loops of a regular graph $\Gbb=(V,E,F)$ with  $e\in E^o\setminus( E^o_\ell \cup E^o_{\ell_1}\cup 
E^o_{\ell_2}),$ then 
\begin{equation}
d\omega_e .\phi_{\ell}=d\omega_e. \phi_{\ell_1\otimes\ell_2}=0.\label{eq:Non visited Edges}
\end{equation}
Consequently, for any regular loop $\ell$, using the same linear forms as in  section \ref{----sec:Vertex desingularisation}, if $\phi^\infty$ and $(\phi^N)$ are respectively  exact and  approximate solutions of Makeenko--Migdal 
equations, for any regular loop $\ell$ and $X\in \mfm_\ell,$
\begin{equation}
X.\phi^\infty_\ell=\phi_{\delta_X \ell}^\infty \text{ and } |X.\phi^N_\ell-\phi^N_{\delta_X \ell}|\le \frac {\|X\| C} N \label{eq:MM MM Vector}
\end{equation}
while 
\begin{equation}
|X.\scV_{\phi^N,\ell}|\le C \|X\| \left( \sum_{v\in V_\ell} \left(\sqrt{\scV_{\phi^N,\ell_1}\scV_{\phi^N,\ell_2}}+|\phi_{\ell_1}^N| \sqrt{\scV_{\phi^N,\ell_2}}+|\phi_{\ell_2}^N| \sqrt{\scV_{\phi^N,\ell_1}}\right)+\frac{1}{N}\right)\label{eq:MM MM Vector Variance CS} 
\end{equation}
and 
\begin{equation}
|X.\scV_{\phi^N,\ell}|\le  \|X\|C \left( \scV_{\phi^N,\ell} +\sum_{v\in V_\ell} (\scV_{\phi^N,\ell_{v,1}}+\scV_{\phi^N,\ell_{1,2}})+\frac{1}{N}\right) \label{eq:MM MM Vector Variance Young} 
\end{equation}
where for any $v\in V_\ell$ we wrote $\pl_v\ell=\ell_{1,v}\otimes \ell_{2,v}.$  
\end{rmk}

The existence problem of these equations is a consequence of \cite{DGHK} and \cite{Lev}  for the approximate solutions, and given Theorem \ref{-->THM Conv Planar Master Field}, of a simple computation for the exact ones.

\begin{lem} \label{__Lem:Existence}Consider $g\ge 1,T>0$. \begin{enumerate}
\item Assume that $G_N$ is a compact classical group of rank $N$, then setting for all regular graph $\Gbb,$ $a\in \Delta_\Gbb(T)$ and all  loops $\ell,\ell_1,\ell_2\in \Ld(\Gbb)$
$$\phi^N_\ell(a)= \E_{\YM_{\Gbb,a}}[W_\ell],\ \phi^N_{\ell_1\otimes\ell_2}(a)=\E_{\YM_{\Gbb,a}}[W_{\ell_1}W_{\ell_2}]$$
defines an approximate solution of the Makeenko--Migdal equations.
\item Denoting by $c_v$ the constant loop at a vertex $v,$ setting for any regular graph $\Gbb,a\in \Delta_\Gbb(T)$ and $\ell\in \Ld(\Gbb),$  
$$\phi_\ell(a)=\left\{\begin{array}{ll}\Phi_{\tilde\ell}(\tilde a) & \text{ if }\ell\sim_h c_{\und\ell}, \\&\\ 0& \text{ if }\ell\not\sim_h c_{\und\ell}, \end{array}\right.$$
defines an exact solution of the Makeenko--Migdal equations. 
\end{enumerate}
\end{lem}

\begin{proof}  Point 1. is a direct consequence of Proposition \ref{prop:MM_phi} below, together with Cauchy--Schwarz or  arithmetic-geometric mean inequality to get \eqref{eq:MM MM Vector Variance CS} and \eqref{eq:MM MM Vector Variance Young}.  For point 2., we shall only check that the Makeenko--Migdal equations are satisfied and leave the other points to the reader. Consider  a regular graph $\Gbb$ of genus $g$ with $\ell\in \Ld(\Gbb)$ and $v\in V_\ell.$  Consider $\delta_v\ell=\ell_1\otimes\ell_2$ and let us show that $\mu_v\phi_\ell=\phi_{\ell_1}\phi_{\ell_2}.$  If $\ell\not\sim_h c_{\und\ell}$, then the rerooting $\ell'$ at $v$ of $\ell$ satisfies  $\ell'\not\sim_h c_{v}.$ Therefore $\ell_1\not\sim_{h}c_v$ or $\ell_2\not\sim_{h}c_v$ and we conclude that $\phi_\ell=\phi_{\ell'}=0=\phi_{\ell_1}\phi_{\ell_2}.$ Assume now $\ell\sim_h c_{\und{\ell}}$. Consider the universal cover $\tilde \Gbb=(\tilde V,\tilde E,\tilde F )$ of $\Gbb$ with projection map $p$.  For all $a\in\Delta_{\Gbb}^o(T),$   
$$\mu_v .\phi_\ell(a)=\mu_v.(\Phi_{\tilde \ell}(\tilde a ))= \sum_{\tilde v\in p^{-1}(v)\cap T_{\tilde\ell}} (\mu_{\tilde v}.\Phi_{\tilde \ell})(\tilde a),$$
where $T_{\tilde \ell}$ is the set of vertices of $\tilde \Gbb$ visited by $\tilde \ell.$ Since $\ell$ is regular, whether  $\# p^{-1}(v)\cap T_{\tilde \ell}=2$ and $\#(V_{\tilde \ell}\cap p^{-1}(v))=0,$ or  $\# (p^{-1}(v)\cap T_{\tilde \ell})=\#( p^{-1}(v)\cap V_{\tilde \ell})=1.$

In the first case,  $\ell_1\not\sim_h c_v$ and 
$\ell_2\not\sim_hc_v$, so that $\phi_{\ell_1}=\phi_{\ell_2}=0.$ Moreover for any  $\tilde v\in p^{-1}(v)\cap T_{\tilde \ell}$ and $e_1,\ldots,e_4\in \tilde E^o$   four  cyclically ordered, outgoing edges at $\tilde v,$  we may assume that $\tilde\ell$ uses $e_1^{-1}$ and $e_3$ while $e_2,e_4\not\in E_{\tilde\ell}.$   Therefore $d\omega_{e_2}.\Phi_{\tilde \ell}=d\omega_{e_4}.\Phi_{\tilde \ell}=0$ and as $\mu_{\tilde v}=\pm (d\omega_{e_2}+d\omega_{e_4}),$ $(\mu_{\tilde v}.\Phi_{\tilde \ell})(\tilde a) =0=\phi_{\ell_1}(a)\phi_{\ell_2}(a).$ 

In the second case, for  $\tilde v\in V_{\tilde \ell}\cap p^{-1}(v)=T_{\tilde \ell}\cap p^{-1}(v)$, by definition of the universal cover, $\ell_1\sim_h c_{v} \sim_h \ell_2$. Then $\delta_{\tilde v} \tilde\ell'= \tilde \ell_1\otimes\tilde \ell_2,$ where  $\tilde \ell_1,\tilde \ell_2$ are lift with initial condition $\tilde v,$  so that using 3. a) of Theorem \ref{-->THM Conv Planar Master Field}, we get
\[ (\mu_{\tilde v}.\Phi_{\tilde \ell})(\tilde a)= \Phi_{\tilde \ell_1}(\tilde a)\Phi_{\tilde \ell_2}(\tilde a)=\phi_{\ell_1}(a)\phi_{\ell_2}(a). \]
\end{proof}

The main technical result of this article is the proof of the following uniqueness statements. Denote by $\mathfrak{L}_{g}$ the space of regular loops of regular maps of genus $g\ge 1.$ Let us say that a subset $\mathcal{F}$ of $\mathfrak{L}_g$ is a \emph{good boundary condition of the Makeenko--Migdal equations} if  for any pair $\phi^\infty$ and $(\phi^N)_{N\ge 1}$ made of an exact and an approximate solutions of Makeenko--Migdal equations,
\begin{equation}
\lim_{N\to\infty} \|\phi^N_\ell-\phi^\infty_\ell\|_\infty+ \|\scV_{\phi^N,\ell}\|_\infty=0,\ \forall \ell\in \mathcal{F} 
\end{equation}
implies 
\begin{equation}
\lim_{N\to\infty} \|\phi^N_\ell-\phi^\infty_\ell\|_\infty+\|\scV_{\phi^N,\ell}\|_\infty=0,\ \forall \ell\in \mathfrak{L}_g.
\end{equation}
Setting
\begin{equation}
\Psi_\ell^N= \phi^N_{(\ell-\phi^\infty_\ell c)\otimes (\ell-\phi^\infty_{\ell}c)^*}=\scV_{\phi^N,\ell}+|\phi^N_\ell-\phi^\infty_\ell|^2,\label{eq:Psi}
\end{equation}
where $c$ is the constant loop at $\und \ell,$ this is equivalent to 
$$\lim_{N\to\infty}\|\Psi^N_\ell\|_\infty=0,\ \forall \ell\in\mathcal{F} \Rightarrow \lim_{N\to\infty}\|\Psi^N_\ell\|_\infty=0,\ \forall \ell\in\mathfrak{L}_g. $$

\begin{prop} \label{-->Prop: MM  Gen} For any genus $g\ge 1$ and total volume  $T>0$, the family of loops $\ell\in \mathfrak{L}_g$ with a sub-path $\g$ such that $(\ell,\g)$ is a marked loop  and $(\ell,\g)^\wedge$ is geodesic, is a good boundary condition. 
\end{prop}
Denote by $\mathfrak{L}^*_{g}$ the subset of $\mathfrak{L}_g$ of loops $\ell$ with $[\ell]_\Z\not=0.$ Let us say that a subset $\mathcal{F}^*$ of $\mathfrak{L}_g^*$ is a \emph{good boundary condition in homology} if for any pair $\phi^\infty$ and $(\phi^N)_{N\ge 1}$ made of an exact and an approximate solution of Makeenko--Migdal equations, using the same notation as in \eqref{eq:Psi},
$$\lim_{N\to\infty}\|\Psi^N_\ell\|_\infty=0,\ \forall \ell\in\mathcal{F}^* \Rightarrow \lim_{N\to\infty}\|\Psi^N_\ell\|_\infty=0,\ \forall \ell\in\mathfrak{L}_g^*. $$
The following can be proven independently from Proposition \ref{-->Prop: MM  Gen}. 

\begin{prop} \label{-->Prop: MM  Gen Non Zero Homol} For any genus $g\ge 1$ and  total volume $T>0$, the family of geodesic loops in $\mathfrak{L}^*_g$ is a good boundary condition in homology.
\end{prop}

When $g=1,$ for any loop $\ell\in\ell_g$,  $\ell\sim_hc_{\und{\ell}}$ if and only if $[\ell]_\Z=0$ and any geodesic loop  is of the form $s^d$ where $s$ 
is a simple loop and $d\ge 1$. Therefore the  Proposition \ref{-->Prop: MM  Gen Non Zero Homol} and \ref{-->Prop: MM  Gen} have the following consequence. 

\begin{coro} \label{-->Prop: MM  Torus} Consider $g=1 ,$  $T>0,$ the set of regular loops $\ell\in \mathfrak{L}_g$ such that  $|\ell|_D=0$ or $\ell=s^d$ for some simple loop $s$ and some  integer $d\ge 1$ is a good boundary 
condition. \end{coro}

\begin{proof}[Proof of Theorem \ref{-->THM Wilson Loops Higher Genus} and Proposition \ref{-->Prop Wilson Loops Non zero Homol}] Since $L^2$ convergence implies convergence in probability, both statements follow from Lemma \ref{__Lem:Existence} and of respectively proposition \ref{-->Prop: MM  Gen} and \ref{-->Prop: MM  Gen Non Zero Homol}. 
\end{proof}

\begin{proof}[Proof of Theorem \ref{-->THM: Torus}]

Using the  solutions given by 1. and 2. of Lemma \ref{__Lem:Existence}, Theorem \ref{-->THM: CUT Conv} implies that the boundary condition of corollary \ref{-->Prop: MM  Torus} are satisfied. Therefore the convergence in 
probability holds true for any regular loops.  Using Lemma \ref{__Lem:Back to regular path},  it follows that the convergence holds for all $\g\in A(\Sigma)\cap \Ld(\Sigma).$ When $\g\in A(\Sigma)\setminus \Ld(\Sigma),$ under  $\YM_{\Sigma},$
$W_\g$ is  Haar distributed, so that $\E_{\YM_{\Sigma}}[|W_\g|^2]\to0 $ as $N\to\infty$ by \cite{DiaconisEvans}. To prove the convergence in probability for any path of finite length,   it is now enough to combine the area bound  \eqref{eq:Bound area YM Bd}  with Proposition \ref{-->:Prop Extension Piecewise Geo}. 
The uniqueness claim is proved identically considering in place of the above approximate solution, a constant sequence given by an exact solution.
\end{proof}

\section{Proof of the main result, stability of convergence under deformation}

In this section we give a proof  first of Proposition \ref{-->Prop: MM  Gen Non Zero Homol}, then of Proposition \ref{-->Prop: MM  Gen}. We consider  exact and  approximate solutions $\phi^\infty$ and $(\phi^N)_{N\ge 1}$ of Makeenko--Migdal equations in genus $g\ge 1$ and volume $T>0$, define $\Psi^N$ as in \eqref{eq:Psi} and consider the subset $\mathfrak{B}_{g}\subset\ell_g$ of loops  $\ell$ with map $\Gbb,$ satisfying
\begin{equation}
\Psi^N_\ell\underset{N\to\infty}{\to}0 \text{ uniformly on } \Delta_\Gbb(T).\label{eq:MC}
\end{equation}
Our aim  is to find a small subset $\mathfrak{C}_g$ of loops in $\ell_g,$ such that $\mathfrak{C}_g\subset \mfB_g$ implies $\mathfrak{B}_g=\ell_g.$  In the first and second second sections, we shall use respectively the following bounds. Thanks to \eqref{eq:MM MM Vector}, \eqref{eq:MM MM Vector Variance CS} and \eqref{eq:MM MM Vector Variance Young}, using the same notation, for any $\ell\in\mathfrak{L}_g$ and $X\in \mfm_\ell,$ 
\begin{equation}
|X.\Psi^N_\ell|\le\|X\|C'_\ell \left( \sum_{v\in V_\ell} \left(\sqrt{\Psi^N_{\ell_{v,1}}}+ |\phi^\infty_{\ell_{v,1}}| \right)\left(\sqrt{\Psi^N_{\ell_{v,2}}}+|\phi^\infty_{\ell_{v,2}}|\right)+ \frac{1}{N}\right) \label{eq:MM Psi CS}
\end{equation}
and
\begin{equation}
|X.\Psi^N_\ell|\le\|X\|C'_\ell \left( \Psi^N_{\ell}+\sum_{v\in V_\ell} (\Psi^N_{\ell_{v,1}}+ \Psi^N_{\ell_{v,2}})+ \frac{1}{N}\right) \label{eq:MM Psi Young}
\end{equation}
where $C'_\ell>0$ is a constant independent of $N\ge 1.$

\subsection{Non-null homology loops}

\label{-----sec:Section Non-null Homology Proofs}
Let us denote by $\mfB_g^*$ the subset $\mfB_g\cap\mathfrak{L}^*_g$. The purpose of this section is to prove proposition \ref{-->Prop: MM  Gen Non Zero Homol}. It is equivalent to the following statement.

\bth \label{-->TH:  Homotopy for non null homology}Denote by $\mfC^*_g$ the subset of $\ell_g^*$ of regular loops with non-zero homology which are  geodesic. If $\mfC^*_g\subset\mfB_g^*$, then $\mfB_g^*=\ell_g^*.$
\eth

The proof of this Theorem  hinges on the following application of Makeenko--Migdal equations, similarly to the argument of \cite{DN,Hal2}. 

\begin{lem}\label{__Lem:MMStrict} Let $\ell,\ell'\in \ell_g^*$ be two loops of a regular map $\Gbb$ with faces set $F$, such that there is $K\subset F$ with $K\not=F$ and $\ell\sim_K\eta\ell'\eta^{-1}$ where $\eta$ is a path with 
$\und{\eta}=\und{\ell}'$ and $\ov{\eta}=\und{\ell}.$  Assume that $\ell'\in\mfB_g^*$ and that for any $v\in V_\ell,$
\begin{equation}
(\delta_v(\ell)=\ell_1\otimes\ell_2 )\Rightarrow (\ell_1 \text{ or } \ell_2 \text{ belongs to } \mfB_g^*).\label{---eq: Stab Conv deSing}
\end{equation}
Then $\ell\in\mfB_g^*.$
\end{lem}

\begin{proof}[Proof of  Lemma \ref{__Lem:MMStrict}] Setting
\begin{equation}
a'(f)= \left\{\begin{array}{ll} \frac{T}{\#F-\#K}  & \text{ if }f\not\in K,\\ &\\ 0& \text{ if }f\in K.\end{array}\right.\label{eq: one face zero area projection}
\end{equation}
defines an element of $\Delta_K(T).$ According to the compatibility condition 3. of Lemma \ref{__Lem:YM Faces simplex}

and using that $\ell'\in \mfB^*_g$, 
\begin{equation*}
\Psi_{\ell}^N(a')=\Psi_{\eta\ell'\eta^{-1}}^N(a')=\Psi_{\ell'}^N(a')\lto 0 \text{ as }N\to+\infty.\tag{*}
\end{equation*}
Now since $[\ell]\not=0$ and $a,a'\in \Delta_\Gbb(T),$ according to Lemma  \ref{__Lem: Charac MM}, $X=a- a'\in \mfm_{\ell}.$ 

Using the assumption \eqref{---eq: Stab Conv deSing} and the inequality \eqref{eq:MM Psi CS},  each term of the summand vanishes uniformly on $\Delta_\Gbb(T)$ as $N\to\infty$ and  for any $t\in (0,1),$
\begin{equation*}
|\partial_t \Psi_{\ell}^N(a+t X)|=  |X. \Psi_{\ell}^N(a+t X)|\le C_\ell \|X\| \varepsilon_N\le C_\ell(\|a\|+\|a'\|)\varepsilon_N \tag{**}
\end{equation*}
where $\varepsilon_N\to 0.$  
Thanks to the boundary condition (*), we conclude that   
$$\Psi_{\ell}^N(a)= \Psi_{\ell'}^N(a')+\int_0^1 \partial_t \Psi_{\ell}(a'+t X)dt$$
converges  to $0$ uniformly in $a\in \Delta_\Gbb(T),$ as $N\to \infty,$  that is $\ell\in \mfB_g^*.$
\end{proof}

We split the proof Theorem \ref{-->TH:  Homotopy for non null homology} into two steps.  The first one allows to contract inner loops, the second allows to follow a shortening 
sequence from proper loops to loops conjugated to a geodesic. Denote by $\mfP_g^*$ the subset of  $\ell_g^*$ of loops which are  proper or included in a fundamental domain.  Theorem \ref{-->TH:  Homotopy for non null homology} is a direct consequence of the following. 

\begin{prop} \label{Prop------------Boundary Cond non zero homol}
\begin{enumerate}[a)]
\item If $\mfP^*_g\subset\mfB^*_g$, then $\mfB^*_g= \ell_g^*.$ 
\item If $\mfC^*_g\subset\mfB_g^*$, then $\mfP^*_g\subset\mfB^*_g.$
\end{enumerate}
\end{prop}
\begin{proof} Let us recall the definition of $\mathcal{C}$ above Lemma \ref{----lem: desingularisation}.  
Let us prove first point a). Assume $\mfP^*_g\subset\mfB^*_g$ and introduce for any $n\ge0$ the subset  $\ell_{n,g}^*$ of loops  $\ell\in\ell_g^*$ with $\mathcal{C}(\ell)\le n.$ By assumption  $\ell^*_{0,g}\subset\mfP^*_g\subset \mfB^*_g.$

\vspace{0.5em}

Consider  $n>0$ and assume $\ell^*_{n-1,g}\subset \mfB^*_g.$ Consider $\ell\in \ell_{n,g}$  with  $\#V_{c,\ell}>0$.  According to   Lemma \ref{----lem: desingularisation}, for all $v\in V_\ell$ with $\delta_v \ell=\ell_1\otimes\ell_2,$ $\mathcal{C}(\ell_1),\mathcal{C}(\ell_2)<n$ and $[\ell_1]$ or $[\ell_2]\not=0.$ Hence $\ell_1$ or $\ell_2$ belongs to $\ell^*_{n-1,g}.$ Choosing $K$ as the bulk of an inner loop 
$\a$ of $\ell,$ and $ \ell'$ the  loop obtained from $\ell$ by erasing the edges of $\a,$   $\ell'\sim_{K} \ell,$ $\ell'\in \ell^*_{n-1,g}$  and   Lemma \ref{__Lem:MMStrict} implies $\ell\in\mfB_g^*.$ Point a) follows by induction.

\vspace{0.5em}

Let us now prove b), assume  that $\mfC^*_g\subset\mfB_g^*$ and introduce for any $n\ge0$ the subset  $\mfP_{n,g}^*$ of proper loops  $\ell\in\mfP_g^*$ with $|\ell|_D\le n$. 

By assumption $ \mfP_{0,g}^*\subset \mfC^*_g\subset \mfB_{g}^*.$  Assume that $n>0$ and $\mfP_{n-1,g}^*\subset \mfB_g^*,$ and consider $\ell\in\mfP_{n,g}^*.$ According to Proposition \ref{__Prop: Shortening Homotopies}, there is a geodesic loop $\ell'\in\mfC_g^*$ and shortening homotopy 
sequence $\ell_1,\ldots,\ell_m$ of 
proper loops with $\ell_1=\ell$ and $\ell_m\sim_K\eta \ell'\eta^{-1}$ for some path $\eta$ and proper subset of faces $K$.  By assumption $\ell'\in \mfB_g^*$.   Using Lemma  \ref{----lem: desingularisation} and Lemma  \ref{__Lem:MMStrict}, by induction on $m,$ $\ell\in \mfB_g^*.$

This concludes the proof of b) by induction on $n$.
\end{proof}

\begin{rmk} In the above proof,  if we furthermore assume simple loops with non vanishing homology to be included in $\mfB_g^*,$ it is also possible to argue by induction on the number of vertices. 
\end{rmk}

\subsection{Null homology loops}

\label{----sec: Proof Vanishing Homology}

The purpose of this sub-section is to prove Proposition \ref{-->Prop: MM  Gen}.  It is equivalent to the following statement.

\bth \label{-->TH:  Homotopy for null homology}Denote by $\mfC^\vee_g$ the subset of $\ell_g$ of regular loops $\ell$, such that there is a nested sub-path $\g_{nest}$ of $\ell$ making $(\ell,\g_{nest})$ a marked loop on a map of genus $g$ and with  $\ell^{\wedge}$ geodesic. If  $\mfC_g^\vee\subset\mfB_g$, then $\mfB_g=\ell_g.$
\eth

To prove this theorem, we shall use the following  lemma, formally analog to Lemma \ref{__Lem:MMStrict}. Though, unlike Lemma \ref{__Lem:MMStrict}, due to the new constraint on the Makeenko--Migdal vectors, we 
work here with marked loops  and change the nested part in order to keep the contraint satisfied while performing the required homotopy.  This will break the induction on the number of intersection points or the complexity $\mathcal{C}$ on regular 
loops.

\vspace{0.5em}

The following Lemma hinges on the observation, appearing in step 4 of the proof below,  that  loops obtained by de-singularisation at  the intersection points of the nested part   of a marked loops yields whether inner loops or a contraction of faces bounded by inner loops of the nested part. The Makeenko-Migdal equation  leads then to  a Gr\"onwall inequality that allows to use an induction on the complexity $\mathcal{C}^{\mathfrak{m}}$ on marked loops. 

\vspace{0.5em}

\emph{Uniqueness of Makeenko--Migdal equations, example of Figure \ref{Fig----Homotopie Raquette 2}:}  \label{example------Raq} Let us illustrate the main idea used in the lemma by a simple example related to the deformation considered in Figure \ref{Fig----Homotopie Raquette 2}.  Consider  $\Delta= \{(a,b)\in \R^2_+: a+b\le T\}$  and a function  $F\in C^1(\Delta)$  
associated to a solution $\psi$ of the Makeenko--Migdal equations for the loop illustrated on the left of Figure  \ref{Fig----Homotopie Raquette 3}.   
\begin{figure}[!h]
\centering
\includegraphics[scale=0.6]{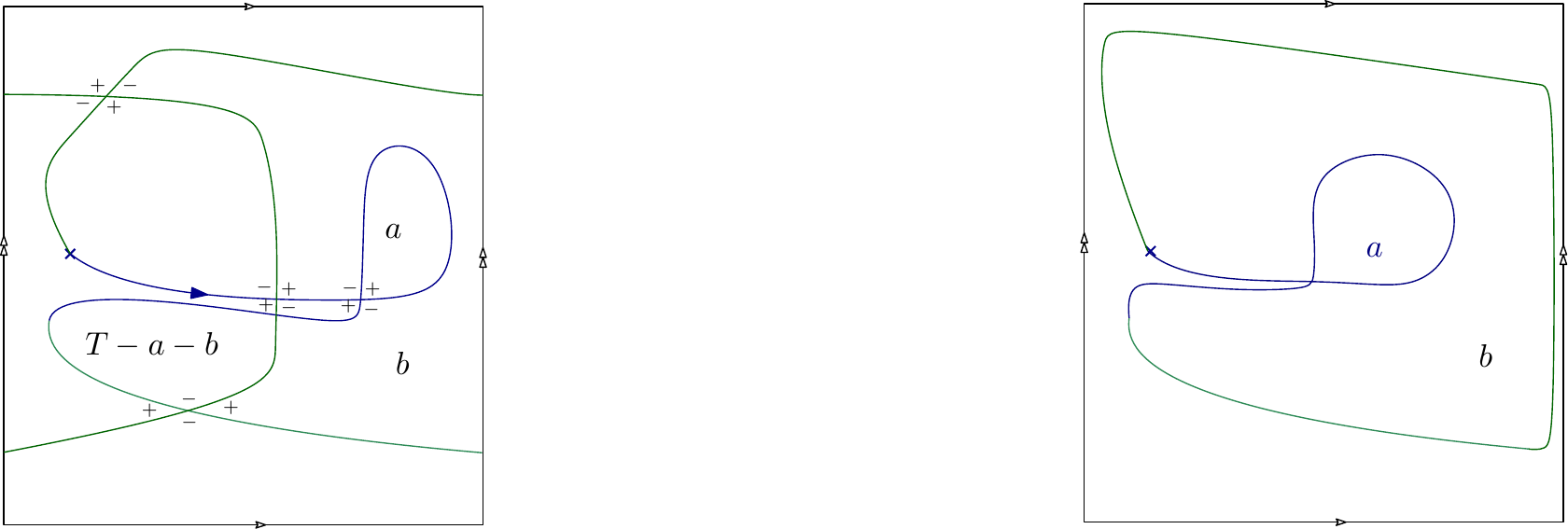}
\caption{\small Faces are labeled by their area. Faces without label have area $0$. In the left figure, $\pm$ symbols stand for the area change involved in the decomposition of $\delta_aF$ as  a sum of Makeenko--Migdal vectors acting on $\psi.$ Here only one vertex yields a de-singularisation with only null-homology loops.  }\label{Fig----Homotopie Raquette 3}
\end{figure}
Assume that  $\psi$ vanishes on loops of non-null homology and matches with the planar master field for loops included in a fundamental domain. Then the  restriction  $F_{|a+b=T}$  is  associated to the  loop on the right of Figure  
\ref{Fig----Homotopie Raquette 3} and  
\begin{equation*}
F(a,T-a)=(1-a)e^{-\frac{a+T}{2}},\ \forall a\in[0,T].\tag{$\clubsuit$}
\end{equation*}
Moreover, the Makeenko--Migdal equations imply 
\begin{equation*}
\partial_a F(a,b)=  -e^{-a}F(0,a+b),\ \forall (a,b)\in\Delta.\tag{$\spadesuit$}
\end{equation*}
The equation $(\spadesuit)$ with boundary condition $(\clubsuit)$ has a unique solution. Indeed denoting by $G$ the difference of two solutions, and  setting $H(t)=\sup_{a\in [0,t]}|G(a,t-a)|,$
\[H(t)\le \int_t^T H(s)ds,\ \forall t\in [0,T].\]
It follows easily that  $H(t)=0$ for all $t\in [0,T].$ We conclude that  
\[F(a,b)= (1-a)e^{-\frac{2T-b}{2}},\ \forall (a,b)\in\Delta.\]

Let us return to the proof of Theorem \ref{-->TH:  Homotopy for null homology}. Denote by $\ell_g^\mfm$ the set of marked loops on a regular map of genus $g$ and by $\mfB_g^\mfm$ the set of $( \ell,\g_{nest})\in \ell_g^\mfm $ such that  $\ell'\in \mfB_g,$ whenever $(\ell',\g'_{nest})\in\ell_g^\mfm$ with ${\ell'}^{\wedge_*}=\ell^{\wedge_*}.$ Recall the notation \eqref{eq-------Conv Desing MarkedL} for the de-singularisation of a marked loop.

\begin{lem}\label{__Lem:MM Gronwall}  Assume that for any regular loop $\ell$ with $|\ell|_D=0,$ $\ell\in\mfB_g.$ Let $x=(\a,\a_{nest}),y=(\b,\b_{nest})\in \ell_g^m$ be two marked loops on a same regular map $\mathbb{G}$ and $K$  a proper subset of faces of 
$\Gbb$, such that     $\a_{nest}=\b_{nest}$ with a moving edge that is not adjacent to any face of $K$,  $\a^{\wedge_*}\sim_K\b^{\wedge_*}$ and $y\in\mfB_g^\mfm, $ while
\begin{equation}
\forall v\in V_{\a^{\wedge}},   \delta_v(x)=x_1\otimes x_2 \text{ with  } x_1,x_2\in\mfB_g^{\mfm}. \label{eq--------Stab Desing Marked L}
\end{equation}
Then $\a\in\mfB_g.$
\end{lem}

\begin{proof}[Proof of  Lemma \ref{__Lem:MM Gronwall}] 

Since  $\a\sim_K\b$ and $\b\in\mfB_g,$ 
\begin{equation}
\Psi_\a^N=\Psi_\beta^N\to 0  \text{ uniformly on }\Delta_{\Gbb,K}(T).\label{eq:Unif CV Contraction Set}
\end{equation}

\und{\textbf{Step 1:}}   Let us first show that it is equivalent to show the convergence on another simplex.  Thanks to Theorem \ref{-->TH:  Homotopy for non null homology}, we can assume that $[\a]=0.$  Recall from Lemma \ref{__Lem: Homology basis choice}, that since $[\a]=0,$ $\a$ has a  
winding number function $n_{\a}\in \Omega^2(\Gbb)$ unique up to the choice of an additive constant.  Let $F_{nest}$ be the bulk of the nested part of  $(\a,\a_{nest})$ and let $f_o$ be its outer face. Let $n_{\a^\wedge}\in\Omega^2(\Gbb)$ be the winding number function of $\a^\wedge $ with $n_{\a^\wedge}(f_o)=n_{\a}(f_o).$  
Let us fix $n_\a,n_{\a^\wedge}$ setting $n_{\a}(f_o)=0,$  set 

\[\Delta_\pm(T)=\{a\in\Delta_\Gbb(T): \pm \<n_{\a^\wedge},a\>\ge 0\}\]
and consider  faces $f_-,f_+\in F_{nest}\cup\{f_o\}$ such that 
\[n_\a(f_-)=\min_{f\in F_{nest}\cup \{f_o\}} n_\a(f_-) =n_-\text{ and }n_\a(f_+)=\max_{f\in F_{nest}\cup \{f_o\}} n_\a(f_+)=n_+>\]

Since  $\Delta_+(T)\cup\Delta_-(T)=\Delta_\Gbb(T)$ and $\Psi_{{\a}^{-1}}^N=\Psi^N_{\a}$,  it is enough to show that  as $N\to \infty,$  $\Psi^N_{\a}\to 0$ uniformly on $\Delta_+(T)$.

\vspace{0.5em}

Let us modify $\a$ as follows. Consider $\l=2\max_{f\in F} |n_\a(f)|$ and define  $(\ell,\g_{nest})$ as the $\l$-twist of  $(\a,\a_{nest})$. Denote by $\Gbb^{'}=(V',E',F')$ the associated map finer than $\Gbb$ and by $F_{tw}$  the 
subset of $\l$ faces of $F'$ associated to the twist move such that $\ell\sim_{F_{tw}}\a$.  Denote by $f_l$ the face of $\Gbb$ left of the moving edge and respectively by $f'_l$ and $f'_c$ the unique face of $\Gbb'$ adjacent to $F_{tw}$  and the central face of $(\ell,\g_{nest}).$  Faces of $F\setminus \{f_l\}$ are not changed by the twist and can be identified with 
$F'\setminus \left(F_{tw} \cup \{f'_l\}\right). $ In particular, faces of $K$ can and will be identified with faces of $\Gbb'.$  We shall write $f'_-=f'_l$  when $f_-=f_l$,  and $f'_-=f_-$  otherwise.

Recall that  $[\ell]=[ \a]=0$ and denote by $n_\ell$ the winding number function of $\ell$ with $n_{\ell}(f'_o)=0.$ It satisfies
$$n_{\ell}(f_{c}')= \l+ n_\a(f_l), 1\le n_{\ell}(f)-n_{\a}(f_l)\ \le \l-1,\ \forall f\in F_{tw}\setminus \{f_{c}\} $$
while 
$$n_{\ell}(f)=n_\a(f),\ \forall f\in F'\setminus \left(F_{tw}\cup\{f'_o\}\right) \text{ and }n_{\ell}(f'_l)=n_\a(f_l).$$ 
It follows that
\begin{equation}
n_\ell(f_c')=\max_{f\in F'} n_\ell(f).\label{eq-------max Winding}
\end{equation}
Recall that   $\a^\wedge=\ell^\wedge$ viewed as loops in $\Gbb'$ and denote 
\[\Delta_+'(T)=\{a\in\Delta_{\Gbb'}(T): \<n_{\ell^\wedge},a\>\ge 0\}.\]
Since the restriction map from  $\Delta_+'(T)$ to $\Delta_{+}(T)$ is surjective,  it is enough to show that $\Psi_{\ell}^N\to0$ uniformly on $\Delta_+'(T).$  

\vspace{0.5em}

For any $a\in \Delta_+'(T),$ thanks to \eqref{eq-------max Winding} and since $n_{\ell}(f)\ge n_-$ for all 
$f\in F_{nest},$
\[ n_\ell(f_-') T\le \<n_{\ell^\wedge},a\>+a(F_{nest})n_-  \le \< n_\ell,a\>\le n_\ell(f_c')T .\]
Hence setting $K_*=F'\setminus \{f'_{c},f'_-\},$ there is a vector $a'\in\Delta_{K^*}(T)$ with 
\[\<n_{\ell},a'\>=\<n_{\ell}, a\>\]
and hence $X=a'-a\in \mfm_{\ell}.$ Moreover, since $n_{\ell^\wedge}$ vanishes on $\{f'_l,f'_-\}$, $\<n_{\ell^\wedge},a'\>=0$ and $a'\in\Delta'_+(T).$

\vspace{0.5em}

\und{\textbf{Step 2:}}  Let us now  use Makeenko--Migdal equations to  show the latter uniform convergence. Let us  bound $\delta_v \Psi_\ell  $ for all $v\in V_\ell.$ Denote $v_1,\ldots,v_n$  the  intersection points of the  nested part of $\ell,$ ordering them so that  $\ell_{nest}=(v_1\ldots v_nv_n\ldots v_1).$ Denote 
by $F'_{nest}$ the bulk  of $\ell_{nest}$.

\vspace{0.5em}

Note first that $V_{\ell^\wedge}=V_{\a^\wedge}$.   Writing $z=(\ell,\g_{nest})$, for all $v\in V_{\ell^\wedge}=V_{\a^\wedge}$, if   $\delta_v(x)=x_1\otimes x_2,$ then $\delta_v(z)=z_1\otimes z_2 ,$ where  one marked loop say $z_1$ is identical to or obtained 
from $x_i$ by $\l$-
twist at the moving edge $e,$ whereas  the other satisfies $z_2=x_2.$ In particular $z_i^\wedge=x_i^\wedge$ and using  \eqref{eq--------Stab Desing Marked L}, $z_i\in \mfB^\mfm_g$ and  $\ell_{v,i}\in \mfB_g$ for $i\in\{1,2\}.$ 
Consider next $V_{\ell_{nest}}.$ For all $1\le k\le n,$ w.l.o.g., $\delta_{v_k}(\ell)=\a_k\otimes\ell_k, $ where $\a_k$ is a nested loop with $|\a_k|_D=0$, hence  $\a_k\in\mathcal{B}_g,$ and $\ell_k$ is a sub-loop of $\ell$, with $\ell_1=\a$ and $\ell_k\sim_{F_{nest}} \ell$ for all  $1\le k\le n.$ Denote by $F_k$ the minimal subset of $F_{nest}$ with $\ell_k\sim_{F_{k}} \ell.$ Since  $X\in \mfm_{\ell},$  using the inequality \eqref{eq:MM Psi Young}, we find
\begin{equation}
|X. \Psi_{\ell}^N|\le C \left( \varepsilon_N+  \Psi^N_{\ell} + \sum_{k=1}^{n} \Psi^N_{\ell_k}\right), \label{eq:Bound Derivative GM} 
\end{equation}
where $C>0$ is a constant independent of $N$ and  $\varepsilon_N=\frac 1 N+ \sup_{1\le k\le n_+}\|\Psi^N_{\a_k}\|_\infty +\sup_{v\in V_{\ell^\wedge}}\left( \|\Psi^N_{\ell_{v,1}}\|_{\infty}+\|\Psi_{\ell_{v,2}}\|_\infty\right),$ and we have just shown that 
$\lim_{N\to\infty}\varepsilon_N=0.$ Consider now for all $t\in [0,1],$
$$\Delta_{in}(t)= \{a\in \Delta_{\Gbb'}(tT): a(f)=0,  \forall f\not\in  F_{nest}\cup\{f'_o\} \}$$ 
and for all $a\in\Delta_{F_{tw},+}(T)$ fixed, set
$$H_a^N(t)=\sup_{b\in \Delta_{in}(1-t)}\Psi_{\ell}( t a + b),\ \forall 0\le t\le 1.$$
On the one hand, for any $t\in (0,1)$ and  $b\in \Delta_{in}(1-t),$ 
$$\pl_s\Psi^N_{\ell}( s a +(t-s)a'+b)=X. \Psi^N_{\ell}(s a +(t-s)a'+b),\ \forall s\in (0,t).$$
On the other hand,
for all $s\in (0,t),$ since $a(F_k)=0$  and $\ell_k\sim_{F_k}\ell$ for all $k$, there are $b_1,\ldots, b_{n}\in \Delta_{in}(1-s)\cap \Delta_{F_k}((1-s)T)$ (see Figure \ref{Fig---: Area swap Gronwall}) such that 

\begin{figure}[!h] 
\centering
\includegraphics[scale=0.4]{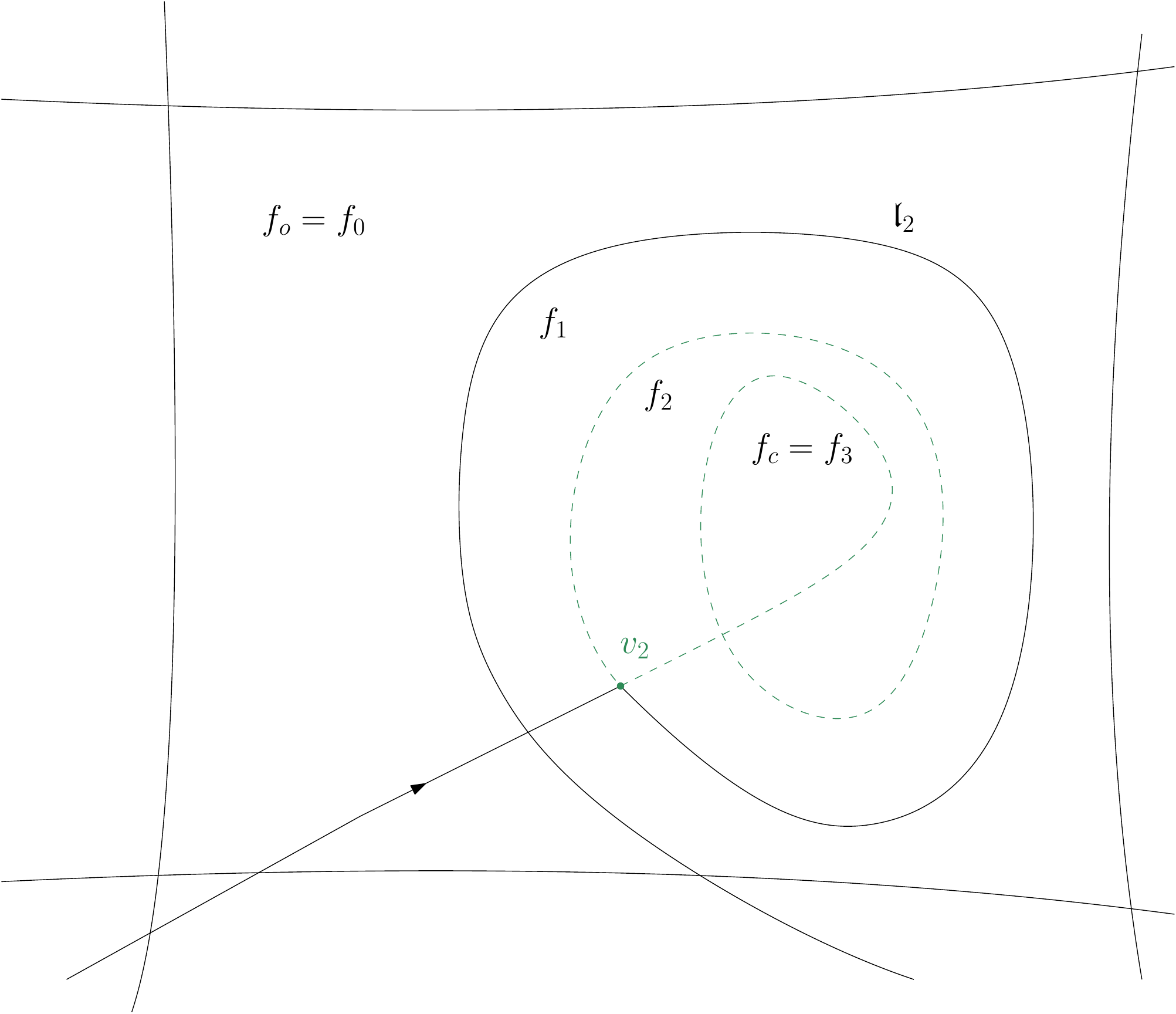} 
\caption{Example of a $n$-left twist with $n=3.$ We consider here $k=2$, the area of $F_2$ needs to be ``moved''  into $f_1$. We have $a(f_1)=a(f_2)=a(f_3)=0=a'(f_1)=a'(f_2).$ For all $0<s<t<1,$ define  $b_2$ setting $b_2(f_1)=b(F_1)+(t-s)a'(F_1)$ and $0$ for other faces. Denote  $a_{s,t}=sa+(t-s)a'+b$ and $\tilde a_{s,t}=sa+b_2$. On the one hand, for any face $f\not\in F_1,$  $a_{s,t}(f)=a'_{s,t}(f)$ while $a_{s,t}(F_1)=a'_{s,t}(F_1)$, therefore $\Psi_{\ell_2}^N(a_{s,t})=\Psi_{\ell_2}^N(\tilde a_{s,t}).$    On the other hand,   $\tilde a_{s,t}(F_2)=0$ so that $\Psi_{\ell_2}^N(\tilde a_{s,t})=\Psi_\ell^N(\tilde a_{s,t}).$
}\label{Fig---: Area swap Gronwall}
\end{figure}
\begin{equation}
\Psi_{\ell_k}(s a +(t-s)a'+b)=\Psi_{\ell}(s a +b_k),\ \forall 1\le k\le n.\label{eq: LA restriction pour Gronwall I}
\end{equation}
Combining the last two equalities with the bound 
\eqref{eq:Bound Derivative GM}, we find 
$$H^N_a(t)\le H_a^N(0)+\varepsilon_N C+(n+1) C\int_0^t H_a^N(s)ds,\ \forall t\in [0,1],a\in\Delta_{+}'(T).$$
By Gr\"onwall's inequality, 
\begin{equation}
H_a^N(t)\le (H_a^N(0)+\varepsilon_N C)\exp((n+1)Ct),\ \forall t\in [0,1].\label{eq:Gronwall I}
\end{equation}
Since $\Delta_{in}(1)\subset \Delta_K(T),$ by \eqref{eq:Unif CV Contraction Set}, 
$$\sup_{a\in \Delta'_+(T)} H_a^N(0)\le \sup_{x\in \Delta_{\Gbb,K}(T)} \Psi_{\ell}(x)$$
vanishes as $N\to \infty.$  
Since $\varepsilon_N\to0$ as $N\to\infty,$    from \eqref{eq:Gronwall I},   
$$ \Psi_{\ell}^N(a)=H^N_a(1)\to 0 $$
uniformly in $a\in\Delta_+'(T).$  
\end{proof}

Using this lemma, the rest of the proof is a refinement of the null-homology case. Denote by $\mfC_g^\mfm,\mfP_g^\mfm $ the set of  marked loops  $(\ell,\ell_{nest})\in\ell_g^\mfm$  with $|\ell|_D=0$, or respectively $\ell^\wedge\in\mfC_g$ and $\ell^\wedge$ proper. Theorem \ref{-->TH:  Homotopy for null homology} is then a direct consequence of the following Proposition. 

\begin{prop} \label{--->Prop: Marked Loops Induction}
\begin{enumerate}[a)]
\item If $\mfP^\mfm_g\subset\mfB^\mfm_g$, then $\mfB^\mfm_g= \ell_g^\mfm.$ 
\item If $\mfC^\mfm_g\subset\mfB_g^\mfm$, then $\mfP^\mfm_g\subset\mfB^\mfm_g.$
\end{enumerate}
\end{prop}

\begin{proof} Let us recall the definition of $\mathcal{C}^\mfm$ above Lemma \ref{----lem: desingularisation}.  
Let us prove first point a). Assume $\mfP^\mfm_g\subset\mfB^\mfm_g$ and introduce for any $n\ge0$ the subset  $\ell_{n,g}^\mfm$ of marked loops  $x\in\ell_g^\mfm$ with $\mathcal{C}^\mfm(x)\le n.$ By assumption  $\ell^\mfm_{0,g}\subset\mfP^\mfm_g\subset \mfB^\mfm_g.$

\vspace{0.5em}

Consider  $n>0$ and assume $\ell^\mfm_{n-1,g}\subset \mfB^\mfm_g.$ Consider $x=(\a,\a_{nest})\in \ell_{n,g}$  with  $\#V_{c,x^\wedge}>0$.  According to   Lemma \ref{----lem: desingularisation}, for all $v\in V_\ell,$  $\delta_v x=x_1\otimes x_2,$
with  $\mathcal{C}^\mfm(\ell_1),\mathcal{C}^\mfm(\ell_2)<n.$ Hence $x_1,x_2 \in\ell^\mfm_{n-1,g}.$   Thanks to Proposition \ref{Prop------------Boundary Cond non zero homol},  we can assume $[\a]=0.$ Choosing $K$ as the bulk of an inner loop 
$\ell$ of $x^\wedge,$ and $y=(\b,\b_{nest})$ the marked loop obtained from $x$ by erasing the edges of $\ell,$   $\a\sim_{K} \b,$ $y\in \ell^\mfm_{n-1,g}.$ Since $\a_{nest}$ do not intersect inner loops of $\a^\wedge,$ the moving edge of $x$ is not 
adjacent to any face of  $K.$   Lemma \ref{__Lem:MM Gronwall} applies and yields  $\ell\in\mfB_g^\mfm.$ Point a) follows by induction.  

\vspace{0.5em}

Consider now b), assume  that $\mfC^\mfm_g\subset\mfB_g^\mfm$ and introduce for any $n\ge0$ the subset  $\mfP_{n,g}^\mfm$ of marked loops  $x\in\mfP_g^\mfm$ with $|x^\wedge|_D\le n$. By assumption $ \mfP_{0,g}^\mfm\subset \mfC^\mfm_g\subset \mfB_{g}^\mfm.$  Assume that $n>0$ and $\mfP_{n-1,g}^\mfm\subset \mfB_g^\mfm,$ and consider $x=(\a,\a_{nest})\in\mfP_{n,g}^\mfm.$ According to 
Proposition \ref{__Prop: Shortening Homotopies}, there is a geodesic loop $\ell'\in\mfC_g^\mfm$ and shortening homotopy 
sequence $x_1,\ldots,x_m$ of 
marked loops with $x_i^\wedge $ proper,  $x_1=x$ and $x_m= (\ell_m,\g_m)$ such that $\ell_m\sim_K\eta \ell'\eta^{-1}$ for some path $\eta$ and proper subset of faces $K$. By assumption $\ell'\in \mfB_g^\mfm$. Consider the first proper set of faces 
$K_1$ with $x^{\wedge_*}_1\sim_{K_1} x^{\wedge_*}_2. $ Denote by $x',y'$ the pull of $x_1$ and $x_2$ to a face  that does not belong to $K_1.$    Lemma \ref{__Lem:MM Gronwall} applies to $x',y'$. Since $\mathcal{C}^\mfm(x_i)$ is non-increasing,  
we conclude by induction on $m$  that  $x\in\mfB_g^\mfm.$ This concludes the proof of b) by induction on $n$.
\end{proof}

\section{Proof of convergence after surgery}

\label{-------sec: Convergence after surgery}

We give here  the main arguments to prove Theorem \ref{-->THM: CUT Conv}. 

\begin{proof}[Proof of Lemma \ref{__Lem: Convergence MF One BD}] Thanks to the second part of Lemma \ref{lem:integral_reducedloops},  under $\YM_{\Gbb,\{f_\infty\},a},$ $(h_{\ell_1},\ldots, h_{\ell_r}, h_{a_1},\ldots ,h_{b_g})$ 
are independent random 
variables on $G_N,$ such that for all $1\le i\le g,$ $h_{a_i},h_{b_i}$ are  Haar distributed, while for any $1\le k\le r,$ $h_{\ell_k}$ has same law as a Brownian motion at time $a(f_k)$. It is now standard, see \cite[Section 3]{Lev}, that as $N\to \infty,$ this tuple of matrices is asymptotically freely independent and its joint non-commutative distribution  converges towards $\tau_v$ satisfying the properties (*), 1,2 and 3. 
\end{proof}
Let us use the same notation as in Theorem \ref{-->THM: CUT Conv}. In what follows, we will denote by $\E$ (resp. $\E_i$, $\E'_i$) the expectation with respect to $\YM_{\Gbb,a}$ (resp. $\YM_{\Gbb_i,a_i}$, $\YM_{\Gbb'_i,a}$). In a previous paper, we proved that the restriction to $\Gbb'_1$ of $\YM_{\Gbb,a}$ is absolutely continuous with respect to $\YM_{\Gbb'_1,a}$.

\begin{prop}[\cite{DL}, Corollary 4.3]\label{prop:abs_cont_gluing}
Let $\ell\in\mathrm{RL}_v(\Gbb'_1)$. For any $f:G_N\to\C$ bounded, measurable and central,
\begin{equation}
\E[f(H_{\ell})]=\E'_1[f(H_{\ell})I(H_{\ell_0}^{-1})],
\end{equation}
where $I:G_N\to\C$ is a bounded measurable function such that
\[
\|I\|_\infty\leq \frac{Z_{g_2,a(F_2)}}{Z_{g,T}}.
\]
\end{prop}

Note that the bound in the previous proposition ensures that $I$ is uniformly bounded, because for any considered sequence $(G_N)_N$, the corresponding sequences of partition functions converge towards a non-zero limit.\footnote{Besides, when $g_2\ge 2$ and $G_N\neq\U(N),$  it remains bounded uniformly in $a\in \Delta_\Gbb(T)$, which  allows then to drop the condition $a(F_2)>0$ in  Theorem \ref{-->THM: CUT Conv}.  }

\begin{proof}[Proof of Theorem \ref{-->THM: CUT Conv}] Without loss of generality, we can assume that $\Gbb$ is a regular map, with  $v=p(\tilde v)$ where $\tilde v\in \tilde V_g.$  Let $\ell$ be a loop in $\mathrm{L}_v(\Gbb_1)$. According to Proposition \ref{prop:abs_cont_gluing},

\[
\E[W_\ell]=\E'_1[W_\ell I(H_{\ell_0^{-1}})]
\]
where $I$ is uniformly bounded in $N.$ From Lemma \ref{__Lem: Convergence MF One BD}, $W_\ell$ converges  in probability towards $\Phi_\ell^{1,g_1}(a_1)$ under $\YM_{\Gbb_1,a}$. Because $I$ is uniformly bounded in $N$, this convergence holds true as well under $\YM_{\Gbb,a}$. It remains to identify $\Phi_\ell^{1,g_1}(a_1)$ with $\Phi_\ell(a).$

\vspace{0.5em}

Consider a free basis $\ell_1,\ldots, \ell_{r},a_1,b_1,\ldots,a_{g}, b_{g}$ of $\mathrm{RL}_v(\Gbb)$ as in 
Lemma \ref{__Lem:Basis Reduced Loops Cut} and let us identify $\mathrm{RL}_v(\Gbb_1)$ as a subgroup of $\mathrm{RL}_v(\Gbb).$ Denote  by $\tilde\tau$ the linear functional on $(\C[\mathrm{RL}_v(\Gbb)],*)$ that satisfies for 
all $\ell\in \mathrm{RL}_v(\Gbb),$ 
\[\tilde\tau(\ell)=\Phi_{\ell}( a).\]
It is enough to show that the restriction of $\tilde\tau$ to $\C[\mathrm{RL}_v(\Gbb_1)]$ satisfies  1,2 and 3 of Lemma \ref{__Lem: Convergence MF One BD}. 

\vspace{0.5em}

Point 3 follows  from point 1 of Lemma \ref{__Lem: Conv Planar MF Maps}.  Consider point 2. For any $\ell\in S_{top}=\{a_1,b_1,\ldots,a_{g_1},b_{g_1}\}$ and $k\in \Z^*,$  $\ell^k$ is not contractible and therefore
$\tilde\tau(\ell^k)=0.$ Let us now prove point 1. Note that $\ell_1,\ldots,\ell_{r_1}$ have same joint distribution under $\tau_v$ and $\tilde\tau.$  Hence, thanks to point 2 of Lemma \ref{__Lem: Conv Planar MF Maps}, 
$\ell_1,\ldots, \ell_{r_1}$ are freely independent under $\tilde\tau.$ 

\vspace{0.5em}

Since $g_2\ge 1,$ according to Lemma \ref{__Lem:Basis Reduced Loops}, identifying $\pi_{1,v}(\Gbb)$ with $\pi_{1,v}(\Gbb),$ the images of $a_1,b_1,\ldots, a_{g_1},b_{g_1}$ in $\pi_{1,v}(\Gbb)\simeq\Gamma_g$   span a free sub-group $\Gamma^{\#}$ of  $\G_g$ of rank $2g_1,$ isomorphic to the group $\mathrm{RL}_{top,1}$ generated by $a_1,\ldots,b_{g_1}$ in $\mathrm{RL}_v(\Gbb).$  Therefore $\tilde V_{L}=\Gamma^\#.\tilde V_g$ is included in a spanning tree  $\mathcal{T}$ of $\tilde V_g.$ Choosing $(\gamma_x)_{x\in \tilde V_g}$ as in Lemma \ref{__Lem:Basis Reduced Loops lift}, $(\widetilde{\g_x\ell_i\g_x^{-1}})_{x\in \tilde V_g,1\le i\le r}$ is a free basis of lassos of $\mathrm{RL}_{\tilde v}(\tilde\Gbb).$ Therefore, thanks to Lemma \ref{__Lem: Conv Planar MF Maps}, $(\g_x\ell_i\g_x^{-1})_{x\in \tilde V_L,1\le i\le r_1}$ are freely independent under $\tilde\tau$.  For any $\g\in\mathrm{RL}_{top,1},$ denote by $\cA_{\g}$ the subalgebra  generated by $\g\ell_1\g^{-1},\ldots,\g\ell_{r_1}\g^{-1}.$ We infer in particular that the sub-algebras $(\cA_\g)_{\g\in \mathrm{RL}_{top,1}}$ are freely independent under $\tilde \tau.$ 

\vspace{0.5em}

Now since $\Gamma^\#$ is free over the image of $S_{top}$,  for any alternated word $w$ in $a_1,\ldots,b_{g_1},$ the image in $\Gamma^\#$ is not trivial and the associated loop  $\ell_w\in\mathrm{RL}_{top,1}$ is not contractible, 
hence $\tilde\tau(w)=0.$   To conclude, it remains to show that the  sub-algebra  $\cA_{\mfk f}$ and $\cA_{top}$ of $\C[\mathrm{RL}_v(\Gbb_1)]$ spanned respectively   by $S_\mathfrak{f}=\{\ell_1,\ldots\ell_{r_1}\}$ and $S_{top}$ are freely 
independent under $\tilde\tau.$ Since $\tilde\tau$ is tracial and unital, it is enough to show \[\tilde \tau(w_1\a_1 w_2\ldots w_n\a_n w_{n+1} )=0\]  whenever $w_1,\ldots,w_{n}\in\mathrm{RL}_{top,1}\setminus\{c_v\},w_{n+1}\in \mathrm{RL}_{top,1}$ and 
$\a_1,\ldots,\a_n\in \cA_{\mfk f}$ with $\tilde\tau (\a_1)=\ldots=\tilde\tau(\a_n)=0.$  Denote by $G_{\mfk f}$ the sub-group of $\mathrm{RL}_v(\Gbb_1)$ generated by $S_{\mfk f}.$  Since $\mathrm{RL}_{top,v}$ is isomorphic to $\Gamma^\#$, if $w_1\ldots w_{n+1}$ does not reduce to the constant loop,  then for any $x_1,\dots,x_n\in G_{\mfk f},$ $w_1x_1w_2\ldots w_n x_nw_{n+1}\sim_h w_1\ldots w_{n+1}$ is not contractible, and the claim follows. Otherwise, $w_1\ldots w_{n+1}=1\in\mathrm{RL}_{top,1}$ and
\begin{equation}
w_1\a_1 w_2\ldots w_n\a_n w_{n+1} =\g_1\a_1\g_1^{-1}\g_2\a_1\g_2^{-1}\ldots \g_n\a_1\g_n^{-1},\label{eq:Word decomposition lassos cover}
\end{equation}
where $\g_i=w_1\ldots w_i$ for all $1\le i\le n.$ Now for all $1\le i<n,$ since $w_{i+1}\not= 1\in\mathrm{RL}_{top,1},$ $\g_i\not=\g_{i+1}$ and it follows that \eqref{eq:Word decomposition lassos cover} is an alternated word in centered elements of $(\cA_g)_{g\in \mathrm{RL}_{top,1}}.$ Since these sub-algebras are free under $\tilde \tau,$ the claim follows.
\end{proof}

\section{Interpolation between regular representations}

\subsection{State extension and interpolation}
In this section, we remark that  the maps considered in conjecture \ref{conj_Lift} have a positivity property and can be seen as states of a non-commutative probability space.

\begin{lem}\label{lem---positivity extension} Consider two groups $G,\Gamma,$ a surjective morphism $\pi:G\to \Gamma$, and $\tau$ a unital state on $(\C[K],1_G,*)$, where $K=\mathrm{ker}(\pi)$ and  $1_G$ denote the neutral element of $G.$  For any $g\in G, $ set
\begin{equation}
\tilde\tau(g)=\left\{\begin{array}{ll}\tau(g)& \text{ if }\pi(g)=1_\Gamma,\\&\\0& \text{ otherwise.}\end{array}\right.
\end{equation}
Assume that for any $ (g,k)\in G\times K, $ 
\begin{equation}
\tau(gng^{-1})=\tau(g).\label{eq: invariance inner Aut}
\end{equation}
Then $\tilde\tau$ extends linearly to a unital state on $(\C[G],1,*).$
\end{lem}

\begin{proof} Let us check that $\tilde \tau$ is tracial. For any $a,b\in G,$ if $\pi(a)\not=\pi(b)^{-1},$ then $\pi(ab),\pi(ba)\not=1_\Gamma$ and $\tilde \tau(ab)=\tilde\tau(ba)=0.$ Otherwise, thanks to \eqref{eq: invariance inner Aut}, $\tilde \tau(ab)=\tau(ab)=\tau(babb^{-1})=\tau(ba)=\tilde \tau(ba).$ Let us check now the positivity condition.  Since $\pi$ is surjective,  there is a right-inverse map $s: \Gamma\to G$   satisfying $\pi\circ s(\gamma)=\gamma$ for all $\g\in\Gamma.$ Consider $x=\sum_{g\in G}\a_g g$ for some finitely supported sequence $(\a_g)_{g\in G}.$ Then
\begin{align*}
\tilde\tau(xx^*)&=\sum_{a,b\in G} \a_a\overline{\a}_b\tilde\tau(ab^{-1})=\sum_{a,b\in G:\pi(a)=\pi(b)} \a_a\overline{\a}_b \tau(ab^{-1})\\
&=\sum_{\g\in \Gamma} \sum_{a,b\in K:} \a_{a s(\g)}\overline{\a}_{bs(\g)} \tau\left(a b^{-1}\right)\\
&=\sum_{\g\in \Gamma} \tau\left( y_\g y_\g^*\right)\ge 0,
\end{align*}
where we set for any $\g\in \G,$  $y_\g=   \sum_{a\in K} \a_{a s(\g)}a.$
\end{proof}

When $G$ is a group, let us denote by $\tau_{reg_G}$ and $\tau_{triv_G}$ the regular and the trivial states on $(\C[G],1_G,*)$ defined by
\[\tau_{reg_G}(g)= \left\{\begin{array}{ll}1& \text{ if } g =1_G,\\&\\0& \text{ otherwise,}\end{array}\right. \text{and } \tau_{triv_G}(g)=1,\ \forall g\in G. \]
The following lemma is straightforward and gives states  interpolating between regular representations of $G$ and $K.$ 
\begin{lem}\label{lem---interpolation}  Consider   $G,\Gamma,$    $\pi$ and $K$ as in Lemma \ref{lem---positivity extension} and $(\tau_T)_{T>0}$ a family of states on $(\C[K],1,*)$ satisfying \eqref{eq: invariance inner Aut}, such that for any $k\in K,$
\begin{equation}
\lim_{T\to 0}\tau_T(k)=\tau_{triv_K}(k) \text{ and } \lim_{T\to\infty}\tau_T(k)=\tau_{reg_K}(k). \label{eq---interpol trivial reg}
\end{equation}
Then for any $g\in G,$
\[\lim_{T\to 0}\tilde \tau_T(g)=\tau_{reg_\Gamma}\circ\pi(g) \text{ and } \lim_{T\to\infty}\tilde \tau_T(g)=\tau_{reg_G}(g).  \]
\end{lem}

Let us consider two examples of extensions of the surface group $\G_g$.

\noindent   \textbf{\emph{Extensions to the free group of even rank:}} Consider the free group $\Fbb_{2g}$ in $2g$ generators $a_1,b_1,\ldots,a_g,b_g$ and  the morphism 
\[\pi: \Fbb_{2g}\to \Gamma_g=\langle x_1,y_1\ldots,x_g,y_g\ |\ [x_1,y_1]\ldots[x_g,y_g]\rangle\]
with $\pi(a_i)=x_i,\pi(b_i)=y_i,\ \forall i$ and $K=\mathrm{ker}(\pi).$ Identifying $\Fbb_{2g}$ with $\Gamma_{1,g},$ this morphism coincides with $\Gamma_{1,g}\to\Gamma_g$ considered in 3. of Lemma \ref{__Lem:Basis Reduced Loops lift}, and accordingly,  there is a right-inverse $s:\Gamma_g\to \Fbb_{2g}$ such that $K$ is free over 
\[ (w_\g)_{\g\in \G_g}=(s(\g) [a_1,b_1]\ldots [a_g,b_g]s(\g)^{-1})_{\g\in \Gamma_g}.\]
Assume that $(\mu_T)_{T>0}$ is a family of measures on the unit circle such that for any integer $n\not=1,$
\[\lim_{T\to 0}\int_{\mathbb U}\omega^n\mu_T(d\omega) =1\text{ and }\lim_{T\to \infty}\int_{\mathbb U}\omega^n\mu_T(d\omega) =0. \]
Denote by $1\in \Fbb_{2g}$  the empty word and consider the unique state  $\tau_T$ on $(\C[K],1,*)$ such that  under $\tau_T,$ $(w_{\g})_{\g\in \G_g}$ are freely independent and identically distributed with distribution $\mu_T.$  

\begin{prop} For any $g\ge 1,$for all $T>0,$ $\tilde\tau_T$ is a state on $(\C[\Fbb_{2g}],1,*)$ with 
\[\lim_{T\to 0}\tilde\tau_T(w)=\tau_{reg_{\G_g}}\circ\pi (w)\text{ and }\lim_{T\to \infty}\tilde\tau_T(w)=\tau_{reg_{\Fbb_{2g}}}\circ\pi (w),\ \forall w\in \Fbb_{2g}.\]
\end{prop}
\begin{rmk}
\begin{enumerate}
\item Mind that under $\tau_{reg_{\Fbb_{2g}}},$ $a_1,b_1,\ldots,a_g,b_g$ are freely independent, whereas when $g=1,$ under $\tau_{reg_{\Z^2}},$ $a_1,b_1$ are classically independent. Hence when $g=1,$  $(\tau_T)_{T>0}$ gives an interpolation between freely and classically independent Haar unitaries.

\item Recall that when $(\mu_T)_{T>0}$ is given by a free unitary Brownian motion,  according to Lemma \eqref{__Lem: Conv Planar MF Maps}, $\tau_T$  can be identified with the restriction of the master field on $\tilde \Gbb_g$ where all polygon faces  have area $T.$
\end{enumerate}
\end{rmk}
\begin{proof} It is enough to prove that  $(\tau_T)_{T>0}$ satisfies the assumptions of Lemma \ref{lem---interpolation}.  We shall only prove \eqref{eq: invariance inner Aut} and leave the proof of the other conditions to the reader.  According to Lemma \ref{lem---- groupoid} there is a surjective group morphism $p:  P\to \Fbb_{2g},$   a state $\eta_T$ on  $(\C[ P],1,*)$ and a sub-group $L$ such that $ p:L\to K$ is surjective with
\[ \eta_T(\ell)=\tau_T\circ \pi(\ell),\ \forall \ell\in L. \]
Hence for any $w\in \Fbb_{2g}$ and $k\in K,$ there are  $\g\in P,\ell\in L$ with $p(\g)=w,$  $p (\ell)=k$,  and since $\eta_T$ is a trace
\[ \tau_T(wkw^{-1})=\tau_T(p(\g \ell\g^{-1}))= \eta_T(\g \ell\g^{-1})=\eta_T(\ell)=\tau_T(k).\]
\end{proof}

Following the same convention as in section \ref{----sec:DiscreteFundamentalCover}, consider the covering map $\tilde \Gbb_g=(\tilde V,\tilde E,\tilde F)$ of the  2g-bouquet map, $\Gbb_g=(V,E,F),$ its $2g$ distinct edges $a_1,b_1,\ldots,a_g,b_g,$ a  vertex $r\in \tilde V,$  an orientation $\tilde E_+$ of the edges of $\tilde \Gbb_g,$ and  the free group $P=\Fbb(\tilde E_+)$ over the $\tilde E_+.$ When $e\in \tilde E_+,$ let us identify  the inverse of $e$ in $H$ with the edge $e^{-1}$ of $\tilde\Gbb$ with reverse orientation.  Denote by $p:P\to\Fbb_{2g}$  the group morphism mapping any edge $\tilde e\in\tilde E$ to its projection $p(e)\in E.$ Note that we can identify any non-trivial reduced path of $(\tilde V,\tilde E)$ with a (strict) subset of $P$ and through this identification, the group of reduced loops of $(\tilde V,\tilde E)$ based at $r$  is identified with a subgroup $L_r$ of $P$ such that 
\[p:L_r\to\ker(\pi)=K\]  
is an isomorphism. Let us fix a spanning tree $\mathcal{T}$ of $(\tilde V,\tilde E).$ As in Lemma \ref{__Lem:Basis Reduced Loops lift},  consider the associated basis $(\omega^r_\g)_{\g\in\G_g} $ of $N$ and denote by 
$(\ell^r_{\g.r})_{\g\in\G_g} $ its pre-image  in $L_r.$   Let us recall another basis of $L_r.$ Denote by $E_+(\mathcal{T})$ the subset of edges of $\mathcal{T}$ in $\tilde E_+.$  For any vertex $v\in\tilde V,$ there is a unique reduced path in $\mathcal{T}$ from $r$ to $v$ and we identify it with an element  $[r,v]_{\mathcal{T}}\in P. $   Then, setting for any $e\in \tilde E_+\setminus E_+(\mathcal{T}), $ 
\[\ell_{r,e}=[r,v]_{\mathcal{T}}e [r,v]_{\mathcal{T}}^{-1} \] 
defines a free basis of $L_r$ indexed by $\tilde E_+\setminus \mathcal{T}.$ It is easy to check that the family $(\ell_{r,e})_{e\in E_+\setminus\mathcal{T}}, (e)_{e\in E_+(\mathcal{T})}$ form a free basis of $P.$ In particular, $P$ is isomorphic to the free product
\[\Fbb(\tilde E_+)=  \Fbb(E_+(T))*L_r. \]
Consider now freely independent unitary non-commutative random variables indexed by $\tilde E_+,$ such that a random variable of this family is Haar unitary if it is indexed by $E_+(\mathcal{T})$, and is distributed according to $\mu_T$ otherwise. Denote by $\eta_T:\C[\Fbb(\tilde E_+)]\to \C$ its non-commutative distribution.  Since the distribution of $(w_\g)_{\g\in\Gamma_g}$ under $\tau_T$ is identical to the one of $(\ell_{\g.r})_{\g\in\G_g}$ the next lemma follows.

\begin{lem}\label{lem---- groupoid}  For any $T>0,$ $\eta_T:\C[\Fbb(\tilde E)]\to \C$ is a state, the morphism  $p:\Fbb(\tilde E_+)\to \Fbb_{2g}$ is surjective, such that $p:L_r\to K$ is an isomorphism with 
\[\eta_T(\ell)= \tau_T(p(\ell)),\ \forall \ell\in L_r.\] \end{lem}

\noindent   \textbf{\emph{Extension to the group of reduced loops:}} Consider a compact surface $\Sigma$ and $r$  a point of $\Sigma.$  The set  $\mathrm{L}_r(\Sigma)$ of 
Lipschitz\footnote{recall the notation introduced in page \pageref{section------ continuous YM statements}}  loop  of $\Sigma$ based at $r$ is  a monoid with multiplication given by concatenation, whose unit element is 
the constant loop at $r.$  It can be turned into a group through the following quotient \cite{HamblyLyons,Lev}.  Following \cite[Sect. 6.7]{Lev}, let us say that a loop $\ell\in\mathrm{L}_r(\Sigma)$ is a \emph{thin loop} if it is homotopic to the constant loop 
at $r$ within its own range.  For any pair $\ell,\ell'\in\mathrm{L}_r(\Sigma)$,  let us define a binary relation setting   $\ell\sim\ell'$ whenever $\ell'\ell^{-1}$ is a thin loop. Let us recall the following. 
\begin{thm}[\cite{Lev}] \begin{enumerate}
\item The relation $\sim$ is an equivalence relation and $\mathrm{RL}_r(\Sigma)=\mathrm{L}_r(\Sigma)/\sim$ is a group. 
\item When $\Sigma=\mathbb{R}^2$ or $\Dbb_\mfk{h},$ the master field $\Phi_\Sigma$ on $\Sigma$ satisfies
\begin{enumerate}
\item    for any pair $\ell,\ell'\in\mathrm{L}_r(\Sigma),$ 
\begin{equation}
\ell\sim\ell'\Rightarrow\Phi_\Sigma(\ell)=\Phi_{\Sigma}(\ell')
\end{equation}
and 
\begin{equation}
\Phi_{\Sigma}(\ell)=1\Rightarrow \ell\sim 1.
\end{equation}
Setting $\Phi_\Sigma(l)=\Phi_\Sigma(\ell)$ for any   $\ell\in \mathrm{L}_r$ with quotient image $l\in\mathrm{RL}_r(\Sigma)$ defines by linear extension a state $\Phi_\Sigma$ on the group algebra $(\C[\mathrm{RL}_r(\Sigma)],1,*).$
\item For any path $a,b\in \mathrm{P}(\Sigma)$ with $\overline a=\underline b$ and $\overline b=\underline a,$
\begin{equation}
\Phi_{\Sigma}(ab)=\Phi_\Sigma(ba).\label{eq----inner aut master field}
\end{equation}
\end{enumerate}

\end{enumerate}
\end{thm}
Consider now a compact orientable Riemannian manifold $\Sigma$, its foundamental cover $p:\tilde \Sigma\to \Sigma,$ a point $\tilde r$ of $\tilde\Sigma$ and $r=p(\tilde r).$ It is elementary to check that the map 
\[\pi:\mathrm{RL}_r(\Sigma)\to\pi_1(\Sigma)\]
sending a based loop  to its based-homotopy class, is a group morphism and that its kernel is given by
\[K =p(\mathrm{RL}_{\tilde r}(\tilde\Sigma)).\] 
For any $l\in K$, let $\tilde l\in \mathrm{RL}_{\tilde r}(\tilde \Sigma)$ be its unique lift starting at $\tilde r.$
\begin{lem} \label{lem-----positivity Projected MF}Setting  
\begin{equation}
\Phi_\Sigma(l)=\left\{\begin{array}{ll}\Phi_{\tilde \Sigma}(\tilde l)& \text{ if }\pi(l)=1,\\&\\0& \text{ otherwise,}\end{array}\right.
\end{equation}
and extending $\Phi_{\Sigma}$ linearly defines a unital state on $(\C[\mathrm{RL}_{r}(\Sigma)],1,*).$
\end{lem}
\begin{proof} Since $\Phi_{\tilde \Sigma}\circ p^{-1}$ defines a state on $(\C[K],1,*),$ thanks to Lemma \ref{lem---positivity extension} it is enough to check \eqref{eq: invariance inner Aut}. The latter follows from \eqref{eq----inner aut master field} applied to $\Phi_{\tilde \Sigma}.$
\end{proof}
\subsection{Master field on the torus and $t$-freeness}

\label{-----sec:Liberation}

Let us give here a proof of corollary \ref{-->CORO: Interpol Class <--> Free}. For $T>0,$ let us consider the two dimensional torus $\Tbb^2_{T}$ obtained as the quotient $\R^2/\sqrt{T}\Z^2$ endowed with the push-forward of the Euclidean metric, so that it has total volume $T$. Denote by  $\a$ and $\b$  the  loop of $\Tbb^2_T$ obtained by projecting the segments from $(0,0)$ to respectively $(\sqrt{T},0)$ and $(0,\sqrt{T}).$ Then, under $\YM_\Sigma,$ the law $(a,b)$  on $G^2$ is given by \eqref{eq:Law GenLoops}. Therefore, for any word $w$ in $\a,\b,\a^{-1},\b^{-1}$ denoting by $[w]\in\Z^2$ the signed number of  occurences of $\a$ and $\b$ and by  $\tilde \g_w$ the path of $\R^2$ starting from $(0,0)$ obtained by lifting the loop  $\Sigma$ formed by $w,$ under $\YM_{\Sigma}$, the following converge holds in probability as $N\to\infty,$
$$\tau_{\rho_N}(w)\to  \left\{\begin{array}{ll}\Phi_{\R^2}(\tilde\g_w) & \text{ if }[\g_w]=0 \\ &\\ 0& \text{ if }[\g_w]\not=0.  \end{array}\right. $$

The first statement of Corollary \ref{-->CORO: Interpol Class <--> Free} follows considering the non-commutative distribution $\Phi_T$ of $\a$t and $\b$ under the limit of $\tau_{\rho_N}$ as $N\to \infty.$ 

\vspace{0.5em}

On the one hand,  for any word $w$ with $[w]=0,$ $\g_w$ is a loop and by continuity of the master field (Point 1 of Theorem \ref{-->THM Conv Planar Master Field}), $\Phi_T(w)=\Phi_{\R^2}(\g_w)\to 1 \text{ as }T\to0.$ On the other hand, for any word $w$ with $[w]\not=0,$ $\tilde\g_w$ is not a loop, $[\g_w]\not=0,$  and  for all $T>0,$ $\Phi_T(w)=0.$ Therefore, for any word in $\a,\b,\a^{-1},\b^{-1},$ $\lim_{T\to0}\Phi_T(w)= \tau_u\star_c\tau_u(w),$ 
since 
\[\tau_u\star_c\tau_u(w)=\left\{\begin{array}{ll}1& \text{ if }[w]=0,\\&\\0& \text{ otherwise.}\end{array}\right.\] 
Consider now the second limit of  corollary \ref{-->CORO: Interpol Class <--> Free}. When  $(\Gbb,a)$  an area weighted map embedded in $\R^2$ with $v$ a vertex of $\Gbb$ sent to $0$ by the embedding, consider the state $\hat{\tau}_T$ on $(\mathrm{RL}_v(\Gbb),*)$ such that $\hat\tau_T(\ell)=\Phi_{\R^2}(\ell_T),$ where  $\ell$ is the drawing of $\ell$ while $\ell_T=\sqrt{T}\ell.$    Consider a free basis of lassos $\ell_1,\ldots \ell_r$ of $\mathrm{RL}_v(\Gbb),$ with meanders given by distinct faces of area $a_1,\ldots, a_r.$  Under $\hat \tau _T,$ $\ell_1,\ldots, \ell_r$ are $r$ independent free unitary Brownian motion marginals at time $\sqrt{T} a_1,\ldots, \sqrt{T}a_r.$ It follows easily from its definition in moments, that the free unitary Brownian motion at time $s$ converges weakly towards a Haar unitary as $s\to\infty$. Since $a\in \Delta^o(T),$  $(\ell_1,\ldots, \ell_r)$ converges weakly toward $r$ freely independent Haar unitary variables as $T\to \infty.$ Therefore,  for any reduced loop $\ell,$  $ \lim_{T\to\infty}\hat{\tau}_T(\ell)=1$ if $\ell$ is the constant loop and $0$ otherwise. Now for any word $w$ in $\a,\b,\a^{-1},\b^{-1},$ with $[w]=0,$  it follows that \[\lim_{T\to\infty}\Phi_T(w)= \left\{\begin{array}{ll}1& \text{ if }\g_w\sim_r c \text{ with }c\text{ constant,}\\&\\0& \text{ otherwise.}\end{array}\right.\]
Since $\g_w\sim_r c$ where $c$ is a constant loop if and only if $w$ can be reduced to the empty word, it follows that  $\lim_{T\to\infty}\Phi_T(w)=\tau_u\star\tau_u(w).$

\vspace{0.5em}

Let us now recall a way introduced in  \cite{BenaychLev} to compute the evaluation of  $\tau_\cA\star_t\tau_\cB$ given $\tau_\cA$ and $\tau_\cB$, solving systems of ODEs in the parameter $t$ and present an argument for \eqref{eq:Free}.  Let us say that a  non-commutative monomial $P$ in $(X_{1,i})_{i\in I},(X_{2,j})_{j\in I}$ is alternated if it is of the form $X_{\varepsilon_1,i_1}X_{\varepsilon_2,i_2}\ldots X_{\varepsilon_n,i_n}$ with $\varepsilon_k\not=\varepsilon_{k+1}$ for all $1\le k<n.$ Denote by 
$d_{X_2}$ is degree in the variables $(X_{2,j})_{i\in I}.$ For such a monomial, let us set
\begin{align*}
\Delta_{ad}.P= &- \frac{d_{X_2}(P)}{2}(P\otimes 1 +1\otimes P)+\sum_{Q_1,Q_2,i}  X_{2,i}\otimes Q_1Q_2,   \\
&-\sum_{P_{1,1},P_{1,2},P_2,i,j} \big[   X_{2,i} P_2\otimes (P_{1,1}X_{2,j}P_{1,2})+ (P_{1,1}X_{2,i}P_{1,2}) \otimes P_2X_{2,j} \\
&\hspace{3 cm} -(P_{1,1}P_{1,2}) \otimes (X_{2,i} P_2X_{2,j}) -(P_{1,1}X_{2,i}X_{2,j}P_{1,2}) \otimes P_2 \big]
\end{align*}
where the first sum is over all monomials   $Q_1,Q_2$ and $i\in I$ such that $P=Q_1X_{2,i}Q_2, $ while the second is over all monomials  $P_{1,1},P_{1,2},P_2$ and $i,j\in I$ such that $P=P_{1,1}X_{2,i} P_2 X_{2,j}P_{1,2}.$ With these notations, Theorem 3.4 of \cite{BenaychLev} states that for all alternated non-commutative monomial  $P$ in $(X_{1,i})_{i\in I},(X_{2,j})_{j\in I}$, $\tau_\cA\star_t\tau_\cB(P)$ is differentiable with
$$\pl_t \tau_\cA\star_t\tau_\cB(P)=(\tau_\cA\star_t\tau_\cB)^{\otimes 2} (\Delta_{ad}.P),\ \forall t\ge0.$$
For instance  assume that for all $t\ge 0,$  $(a,b)$  is a $t$-free couple within a non-commutative probability space $(\cC, \tau_t),$ such that $a$ and 
$b$ are Haar unitaries for all $t>0. $   Then for any $n\ge1,$
$$\pl_t\tau_t(ab^n)=-\tau_t(ab^n)+\tau(a)\tau_t(b^n)=-\tau_t(ab^n),\ \forall t\ge 0$$
and since $\tau_0(ab^n)=\tau_0(a)\tau_0(b^n)=0,$ 
$$\tau_t(ab^n)=0.$$
Likewise
\begin{align*}
\pl_t\tau_t(ab^na^*(b^*)^n)&=-2\tau_t(ab^na^*(b^*)^n)+\tau_t(b^n)\tau_t(aa^*(b^*)^n)+\tau_t(ab^na^*)\tau_t((b^*)^n)\\
&\hspace{- 2 cm} -\tau_t(a(b^*)^n)\tau_t(b^na^*)-\tau_t(a(b^*)^n)\tau_t(b^na^*) +\tau_t(a)\tau_t(b^na^*(b^*)^n)+\tau_t(ab^n(b^*)^n)\tau_t(a^*)\\
&=-2\tau_t(ab^na^*(b^*)^n).
\end{align*}
Since $\tau_0(ab^na^*(b^*)^n)=\tau_0(aa^*)\tau_0(b^n(b^*)^n)=1,$ this implies
\begin{equation}
\tau_t(ab^na^*(b^*)^n)=e^{-2t}.\label{eq: Example tFree}
\end{equation}
A similar argument  together with \eqref{eq:ODE UBM} implies the following lemma. 

\begin{lem} \begin{enumerate}
\item For any word $w$ in $a,b,a^{-1},b^{-1},$ if $[w]\not=0,$
\[\tau_{t}(w)=0.\]
\item For any $n\ge 1$,
\[\pl_t\tau_t([a,b]^n)=-2n \tau_t([a,b]^n)- 2n \sum_{k=1}^{n-1}\tau_t([a,b]^k)\tau([a,b]^{n-k}).\]
\item For any  $n\in\Z$ and $t\ge 0,$ 
\[\tau_t([a,b]^n)=\nu_{4t}(|n|).\]
\end{enumerate}
\end{lem}
The last equality of corollary \ref{-->CORO: Interpol Class <--> Free} follows from the last point of the above Lemma. Besides for any $t>0,T>0$
\[\tau_u\star_t\tau_u(XYX^*Y^*)=e^{-2t}\text{ and }\Phi_T(XYX^*Y^*)=e^{-\frac T 2},\]
so that if $\tau_u\star_t\tau_u=\Phi_T$ then $T=4t.$ But \eqref{eq: Example tFree}  implies $\tau_u\star_t\tau_u(XY^2X^*Y^{-2})=e^{-2t}>e^{ -4t}=\Phi_{4t}(XY^2X^*Y^{-2}) $. 
Therefore for all $t,T>0,$ $\Phi_T\not= \tau_u\star_t\tau_u.$

\section{Appendix}
\subsection{Casimir element and trace formulas }

Let us recall some  tensor identities,  instrumental to prove Makeenko--Migdal relations. 

\begin{dfn}
Consider a Lie algebra $\mathfrak{g} $   endowed with an inner product $\<\cdot,\cdot\>.$ The \emph{Casimir element} of $(\mathfrak{g},\<\cdot,\cdot\>)$ is the tensor $C_\mathfrak{g}\in\mathcal{M}_d(\C)\otimes_\R\mathcal{M}_d(\C)$ defined by
\begin{equation}
C_\mathfrak{g}=\sum_{X\in\mathcal{B}} X\otimes X,
\end{equation}
where $\mathcal{B}$ is an orthonormal basis of $\mathfrak{g}$ for the inner product $\<\cdot,\cdot\>$.
\end{dfn}
It is simple to check that the definition of the  Casimir element does not depend on the choice of the basis but only on the inner product $\<\cdot,\cdot\>$.   We   focus on the setting recalled in section \ref{----sec:HK}; we consider the Lie algebra $\mathfrak{g}_N$ of a compact classical group $G_N$ with the inner product (1) considered in \cite[Section 2.1.]{DL}.   We set  the value $\b$ to be respectively $1$ and $4$ when $G_N$ is  $O(N)$ and $Sp(N)$ and $2$ otherwise, that is when $G_N$ is $SU(N)$ or $U(N).$ We set $\g=1$ when $G_N=SU(N)$ and $0$ otherwise.

\vspace{0.5em}

Most of the following results can be proved by a direct computation using an arbitrary chosen basis. For any $(a,b)\in\{1,\ldots,N\}^2$, the elementary matrix $E_{ab}\in\mathcal{M}_N(\R)$ is defined by $(E_{ab})_{ij}=\delta_{ai}\delta_{bj}$.

\vspace{0.5em}

We shall need the following standard result on the Casimir element in this setting, which gives computation rules for traces of products and product of traces involving elements of $\mathcal{B}$.

\begin{lem}\label{lem:magic_formula}
For any $A,B\in G_N$ we have :
\begin{equation}\label{eq:tr_1}
\sum_{X\in\mathcal{B}} \tr(AXBX) =-\tr(A)\tr(B)-\frac{\beta-2}{\beta N}\tr(AB^{-1})+\frac{\g}{N^2}\tr(AB) \end{equation}
and
\begin{equation}\label{eq:tr_2}
\sum_{X\in\mathcal{B}} \tr(AX)\tr(BX) = -\tr(AB)-\frac{\b-2}{\b N} \tr(AB^{-1}) +\g \tr(A)\tr(B). 
\end{equation}
\end{lem}

\begin{proof}
We only sketch the proof in order to show where the expressions come from. First of all, remark that by linearity they only need to be proved for $A=E_{ij}$ and $B=E_{k\ell}$. We have for instance
\[
\sum_{X\in\mathcal{B}} \tr(AXBX) = \frac{1}{N}\sum_{X\in\mathcal{B}} \sum_{a,b,c,d} A_{ab} X_{bc} B_{cd} X_{da} = \frac{1}{N}(C_{\mathfrak{g}})_{jk\ell i},
\]
where we have set
\[
\big(\sum_i X^i\otimes Y^i\big)_{abcd}=\sum_iX_{ab}^iY_{cd}^i.
\]
Using the expression of $C_{\mathfrak{g}}$ for each value of $\mathfrak{g}$ leads to Eq. \eqref{eq:tr_1}. By similar computations we also obtain Eq. \eqref{eq:tr_2}.
\end{proof}

In the unitary case, the formulas in Lemma \ref{lem:magic_formula} are known as the ``magic formulas'', as stated in \cite{DGHK} for instance, and appeared already in \cite{Sen2}; they are crucial to the derivation of Makeenko--Migdal equations for Wilson loops, that we briefly recall in the next section. Although we do not detail it, there exist a beautiful interpretation Lemma \ref{lem:magic_formula} in terms of Schur--Weyl duality; the interested reader can refer to \cite{Lev5} or \cite{Dah} for an explanation and discussion of this fact and  to \cite[Chap. I, Section 1.2 ]{Lev} about the above Lemma.

\subsection{Makeenko--Migdal equations}\label{sec:MM Reminder}

Given a topological map $\Gbb$ of genus $g$ with $m$ edges, a vertex of $\Gbb$ will be said to be an \emph{admissible crossing} if it 
possesses four outgoing edges labelled $e_1,e_2,e_3,e_4$ counterclockwise.

\begin{dfn}
Let $\Gbb$ be map of genus $g$ with $m$ edges, and $v$ be an admissible crossing. A function $f:G^m\to\C$ has an \emph{extended gauge invariance at} $v$ if for any $x\in G$,
\begin{equation}
f(a_1,a_2,a_3,a_4,\mathbf{b})=f(a_1x,a_2,a_3x,a_4,\mathbf{b})=f(a_1,a_2x,a_3,a_4x,\mathbf{b}),
\end{equation}
where $a_i$ denotes the variable associated to the edge $e_i$ and $\mathbf{b}$ denotes the tuple of other edge variables than $e_1,e_2,e_3,e_4$.
\end{dfn}

The extended gauge-invariance was first introduced by L\'evy in \cite{Lev2} to prove Makeenko--Migdal equations in the plane, then used in  \cite{DHK} to 
give  alternative, local proofs of these equations, which allowed in \cite{DGHK} to prove their validity on any surface; these last equations  were then applied in \cite{DN,Hal2}. 

\begin{thm}[Abstract Makeenko--Migdal equations]
Let $(\Gbb,a)$ be an area weighted map of area $T$ and genus $g$ with $m$ edges, and $f:G^m\to\C$ be a function with extended gauge invariance at an admissible crossing $v$. 
Denote by $f_1$ (resp. $f_2,f_3,f_4$) the face of $\Gbb$ whose boundary contains $(e_1,e_2)$ (resp. $(e_2,e_3)$, $(e_3,e_4)$, $(e_4,e_1)$). Denote by $t_i$ the area of the face 
$f_i$, choose an orthonormal basis $\mathcal{B}$ of $\mathfrak{g}$ with respect to the chosen inner product, and set
\[
(\nabla^{a_1}\cdot\nabla^{a_2} f)(a_1,a_2,a_3,a_4,\mathbf{b})=\sum_{X\in\mathcal{B}}\frac{\partial^2}{\partial s\partial t}f(a_1e^{sX},a_2e^{tX},a_3,a_4,\mathbf{b})\big\vert_{s=t=0}.
\]
We have
\begin{equation}\label{eq:abstract_MM}
\left(\frac{\partial}{\partial t_1}-\frac{\partial}{\partial t_2}+\frac{\partial}{\partial t_3}-\frac{\partial}{\partial t_4}\right)\int_{G^m} f d\mu=-\int_{G^m} \nabla^{a_1}\cdot\nabla^{a_2} f d\mu.
\end{equation}
\end{thm}

Equation \eqref{eq:abstract_MM} might be confusing, as it involves partial derivatives with respect to variables that do not appear explicitly in the function $\int_{G^m} fd\mu$; it becomes in fact clearer after being translated in terms of the area simplex. We define the differential operator $\mu_v$ on functions $\Delta_{\Gbb}(T)\to\C$ by
\[
\mu_v=\frac{\partial}{\partial a_1}-\frac{\partial}{\partial a_2}+\frac{\partial}{\partial a_3}-\frac{\partial}{\partial a_4},
\]
using the labelling of $a=(a_1,\ldots,a_p)\in\Delta_{\Gbb}(T)$ such that $a_i$ corresponds to the face $f_i$.
Equation \eqref{eq:abstract_MM} becomes then
\[
\mu_v \E(f)=-\E(\nabla^{a_1}\cdot\nabla^{a_2}f),
\]
and now everything only depends on the areas of the faces. We want to apply these abstract Makeenko--Migdal equations to functionals of Wilson loops, in order to obtain the convergence to the master field. We define, for $k$  unrooted loops $\ell_1,\ldots,\ell_k\in \Ld_c(\Gbb)$, the $k$-point function $\phi_{\ell_1,\ldots,\ell_k}^G:\Delta_\Gbb(T)\to\C$ by
\[
\phi_{\ell_1\otimes \ldots\otimes \ell_k}^G = \E(W_{\ell_1}\cdots W_{\ell_k}).
\]
and extend it linearly to $\C[\Ld_c(\Gbb)]^{\otimes k}.$
The following proposition offers an estimate of the face-area variation of the functions $\phi_{\ell_1\otimes \ldots\otimes\ell_k}^G$. 

\begin{prop}[Makeenko--Migdal equations for Wilson loops]\label{prop:MM_phi}
Assume that $G_N$ is a compact classical group and $\<\cdot,\cdot\>$ is fixed as in section \ref{----sec:HK}. Let $(\Gbb,a)$ be a weighted map of area $T$ and genus $g$ with $m$ edges, and $v$ be an admissible crossing in $\Gbb$.
\begin{enumerate}
\item If $v$ is a self-intersection of a single loop $\ell_1$ such that the edges $(e_j^{\pm 1}, 1\leq j\leq 4)$ are visited in the following order: $e_1,e_4^{-1},e_2,e_3^{-1}$, then define $\ell_{11}$ the subloop of $\ell_1$ starting at $e_1$ and finishing at $e_4^{-1}$, $\ell_{12}$ the subloop starting at $e_2$ and finishing at $e_3^{-1}$. We have, for any loops $\ell_2,\ldots,\ell_k$ that do not cross $v$,
\begin{equation}\label{eq:MM_Wilson1}
\mu_v\phi_{\ell_1\otimes\ldots\otimes\ell_k}^G = \phi_{\ell_{11}\otimes \ell_{12}\otimes \ell_2\otimes\ldots\otimes\ell_k}^G+\frac{2-\beta}{\beta N}\phi_{\ell_{11}\ell_{12}^{-1}\otimes\ell_2 \otimes\ldots\otimes \ell_k}^G+\frac{\g}{N^2}\phi_{\ell_1\otimes\ldots\otimes \ell_k },
\end{equation}
\begin{equation}\label{eq:MM_Wilson2}
\mu_v\phi_{\ell_1\otimes \ell_1^{-1}}^G = \phi_{\ell_{11}\otimes \ell_{12}\otimes\ell_1^{-1}}^G+\phi_{\ell_1\otimes \ell_{11}^{-1}\otimes \ell_{12}^{-1}}^G+\frac{R_{\ell_1}}{N},
\end{equation}
where the $|R_{\ell_1}|\le 10$ uniformly on $\Delta_\Gbb(T)$.
\item If $v$ is the intersection between two loops $\ell_1$ and $\ell_2$ such that $\ell_1$ starts at $e_1$ and finishes at $e_3^{-1}$, and $\ell_2$ starts at $e_2$ and finishes at $e_4^{-1}$, then define $\ell$ the loop obtained by concatenation of $\ell_1$ and $\ell_2$. We have, for any loops $\ell_3,\ldots,\ell_k$ that do not cross $v$,
\begin{equation}\label{eq:MM_Wilson3}
\mu_v\phi_{\ell_1\otimes \ell_2\otimes \ldots\otimes \ell_k}^G = \frac{R_{\ell_1\otimes \ell_2\otimes \ldots\otimes \ell_k}}{N^2}
\end{equation}
with $|R_{\ell_1\otimes \ell_2\otimes \ldots\otimes \ell_k}|\le 3$ uniformly on $\Delta_\Gbb(T)$.
\end{enumerate}

\end{prop}

It was  proved for all classical Lie algebras if $\Gbb$ is a planar combinatorial graph by L\'evy in \cite[Prop. 6.16]{Lev} when the loops form what he called a skein. If $\Gbb$ is a map of genus $0$ and $\mathfrak{g}$ is the Lie algebra of $\U(N)$, this result was proved by the first author with Norris in \cite[Prop. 4.3]{DN}.  See also \cite[Thm. 1.1]{DHK}

\begin{proof}[Proof of Prop. \ref{prop:MM_phi}]
Let us start with the first case, which is when $v$ is a self-intersection of a loop $\ell_1$. We take $E=\{e_1,e_2,e_3,e_4,e'_1,\ldots,e'_{m-4}\}$ as an orientation of $E$, with $e_1,e_2,e_3,e_4$ the four outgoing edges from $v$. We identify any multiplicative function $h\in\mathcal{M}(P(\mathbb{G}),G)$ to a tuple $(a_1,a_2,a_3,a_4,\mathbf{b})$ by setting $a_i=h_{e_i}$ and $\mathbf{b}=(h_{e'_i})_{1\leq i\leq m-4}$ the tuple of all other images of edges by $h$. There are words $\alpha,\beta,w_2,\ldots,w_k$ in the elements of $\mathbf{b}$ such that
\[
h_{\ell_1}=a_3^{-1}\alpha a_2a_4^{-1}\beta a_1, h_{\ell_i}=w_i\ \forall 2\leq i\leq k.
\]
It appears that $\phi_{\ell_1,\ldots,\ell_k}^G=E(f)$, where $f$ is the extended gauge-invariant function 
\[
f:\left\lbrace\begin{array}{ccc}
G^m & \to & \C\\
(a_1,a_2,a_3,a_4,\mathbf{b}) & \mapsto & \tr(a_3^{-1}\alpha a_2a_4^{-1}\beta a_1)\tr(w_2)\cdots\tr(w_k).
\end{array}\right.
\]
Then, by the abstract Makeenko--Migdal equation \eqref{eq:abstract_MM}, we get
\[
\mu_v \E(f)=-\E(\nabla^{a_1}\cdot\nabla^{a_2} f),
\]
and by definition
\[
\nabla^{a_1}\cdot\nabla^{a_2} f=\left(\sum_{X}\tr(a_3^{-1}\alpha a_2Xa_4^{-1}\beta a_1 X)\right)\tr(w_2)\cdots\tr(w_k)
\]
where $X$ runs through an orthonormal basis of $\mathfrak{g}$. A straightforward application of \eqref{eq:tr_1} from Lemma \ref{lem:magic_formula} yields \eqref{eq:MM_Wilson1}, by noticing that $h_{\ell_{11}}=a_4^{-1}\beta a_1$ and $h_{\ell_{12}}=a_3^{-1}\alpha a_2$.

Similarly, we have $\phi_{\ell,\ell^{-1}}^G=E(f')$, where
\[
f':\left\lbrace\begin{array}{ccc}
G^m & \to & \C\\
(a_1,a_2,a_3,a_4,\mathbf{b}) & \mapsto & \tr(a_3^{-1}\alpha a_2a_4^{-1}\beta a_1)\tr(a_1^{-1}\beta^{-1}a_4a_2^{-1}\alpha^{-1}a_3).
\end{array}\right.
\]
We have
\begin{align*}
\nabla^{a_1}\cdot\nabla^{a_2} f' =  \sum_{X} & \big\lbrace\tr(a_3^{-1}\alpha a_2Xa_4^{-1}\beta a_1 X)\tr(a_1^{-1}\beta^{-1}a_4a_2^{-1}\alpha^{-1}a_3)\\
& -\tr(a_3^{-1}\alpha a_2a_4^{-1}\beta a_1 X)\tr(a_1^{-1}\beta^{-1}a_4Xa_2^{-1}\alpha^{-1}a_3)\\
& - \tr(a_3^{-1}\alpha a_2Xa_4^{-1}\beta a_1)\tr(Xa_1^{-1}\beta^{-1}a_4a_2^{-1}\alpha^{-1}a_3)\\
& + \tr(a_3^{-1}\alpha a_2a_4^{-1}\beta a_1 )\tr(Xa_1^{-1}\beta^{-1}a_4Xa_2^{-1}\alpha^{-1}a_3)\big\rbrace,
\end{align*}
and a simultaneous application of \eqref{eq:tr_1} and \eqref{eq:tr_2} leads to the result. We detail the case of $\SU(N)$ and leave the others as an exercise: if we set $A=h_{\ell_{11}}$ and $B=h_{\ell_{12}}$, then
\begin{align*}
\sum_X \tr(AXBX)\tr(B^{-1}A^{-1}) = & -\tr(A)\tr(B)\tr((AB)^{-1})+\frac{1}{N}\tr(AB)\tr((AB)^{-1})\\
\sum_X\tr(ABX)\tr(B^{-1}XA^{-1}) = & -\frac{1}{N^2}\tr([A,B])+\frac{1}{N}\tr(AB)\tr(A^{-1}B^{-1})\\
\sum_X\tr(AXB)\tr(XB^{-1}A^{-1}) = & -\frac{1}{N^2}\tr([A,B]^{-1})+\frac{1}{N}\tr(BA)\tr(B^{-1}A^{-1})\\
\sum_X \tr(AB)\tr(XB^{-1}XA^{-1}) = & -\tr(A^{-1})\tr(B^{-1})\tr((AB))+\frac{1}{N}\tr(AB)\tr(A^{-1}B^{-1}).
\end{align*}
We can then take the expectation of the alternated sum of these expressions, and as all traces are bounded by 1 because they apply to special unitary matrices, we find that all terms with a coefficient $\frac{1}{N}$ or $\frac{1}{N^2}$ fall into $O\big(\frac{1}{N}\big)$ which does not depend on any loop\footnote{we add up a finite number of terms, 6 to be precise, which are bounded by $\frac{1}{N}$, so their sum is bounded by $\frac{6}{N}$ which is indeed independent from the loops or the face-area vector.}, so that
\[
\phi_{\ell_1\otimes\ell_1^{-1}}^{\SU(N)}= \phi_{\ell_{11}\otimes\ell_{12}\otimes(\ell_{11}\ell_{12})^{-1}}^{\SU(N)}+\phi_{\ell_{11}^{-1}\otimes\ell_{12}^{-1}\otimes(\ell_{11}\ell_{12})}^{\SU(N)}+O\big(\frac{1}{N}\big).
\] 

Let us now turn to the second case, when $v$ is the intersection of $\ell_1$ and $\ell_2$. We take $E=\{e_1,e_2,e_3,e_4,e'_1,\ldots,e'_{m-4}\}$ as an orientation of $E$, with $e_1,e_2,e_3,e_4$ the four outgoing edges from $v$. There are words $\alpha,\beta,w_2,\ldots,w_k$ in the elements of $\mathbf{b}$ such that
\[
h_{\ell_1}=a_3^{-1}\alpha a_1,h_{\ell_2}=a_4^{-1}\alpha a_2, h_{\ell_i}=w_i\ \forall 3\leq i\leq k.
\]
We have $\phi_{\ell_1,\ldots,\ell_k}^G=E(f)$, where $f$ is the extended gauge-invariant function 
\[
f:\left\lbrace\begin{array}{ccc}
G^m & \to & \C\\
(a_1,a_2,a_3,a_4,\mathbf{b}) & \mapsto & \tr(a_3^{-1}\alpha a_1)\tr(a_4^{-1}\beta a_2)\tr(w_2)\cdots\tr(w_k),
\end{array}\right.
\]
then
\[
\mu_v E(f)=-E(\nabla^{a_1}\cdot\nabla^{a_2} f),
\]
where
\[
\nabla^{a_1}\cdot\nabla^{a_2} f=\left(\sum_{X}\tr(a_3^{-1}\alpha a_1)\tr(Xa_4^{-1}\beta a_2 X)\right)\tr(w_2)\cdots\tr(w_k).
\]
The result follows then from \eqref{eq:tr_2}.
\end{proof}

By letting $N\to\infty$ in Prop. \ref{prop:MM_phi}, one immediately gets the following.

\begin{coro}[Makeenko--Migdal equations for a master field]\label{cor:MM_mf}
Assume for some some sequence $(G_N)_N$ of compact classical groups, we have for all maps $\Gbb$ of genus $g\ge 1$ and $\ell\in \Ld(\Gbb),$ 
$\lim_{N\to \infty}\Phi^{G_N}_\ell  $ and $\lim_{N\to \infty}\Phi^{G_N}_{\ell\otimes \ell^{-1}}= |\Phi_\ell|^2 $ uniformly on $\Delta_\Gbb(T),$ then $\Phi$ defines an exact solution of the Makeenko--Migdal solution as defined in section 
\ref{--------sec: MM Existence Uniqueness PB}.
\end{coro}

To address uniqueness  questions, it is convenient to work with \emph{centered Wilson loops}. Define, for any $\ell_1,\ldots,\ell_k$ in an area-weighted graph $(\Gbb,a)$,
\[
\psi_{\ell_1\otimes \ldots\otimes \ell_k}^G = \E\left[\prod_{i=1}^k(W_{\ell_i}-\Phi_{\ell_i})\right].
\]

\begin{prop}[Makeenko--Migdal equations for centered Wilson loops] Assume $g\ge 0,$ $T>0,$ $\ell\in \ell_g$, $v\in V_\ell$ with $\delta_v\ell=\ell_1\otimes \ell_2.$ Then for any compact classical group $G_N,$

\begin{align}
\begin{split}
\mu_v\psi^G_{\ell\otimes\ell^{-1}} = & \psi^G_{\ell_1\otimes\ell_2\otimes\ell^{-1}}+\psi^G_{\ell_1^{-1}\otimes\ell_2^{-1},\ell}+\psi^G_{\ell_1\otimes\ell^{-1}}\Phi_{\ell_2}+\psi^G_{\ell_1^{-1}\otimes\ell}\Phi_{\ell_2^{-1}}\\
&+\psi^G_{\ell_2\otimes\ell^{-1}}\Phi_{\ell_1}+\psi_{\ell_2^{-1}\otimes\ell}\Phi_{\ell_1^{-1}}+\frac{R_{\ell}}{N},
\end{split}
\end{align}
where the $|R_\ell|\le 10$  uniformly on $\Delta_\Gbb(T)$.  There is a constant $C_\ell$ independent of $G,$ such that for all $X\in \mfm_\ell,$
\begin{align*}
\mu_v\psi^G_{\ell\otimes\ell^{-1}} = & \psi_{\delta_X(\ell)\otimes\ell^{-1}}+\psi^G_{\ell,\delta_X(\ell^{-1})}+\psi_{\ell_1\otimes\ell^{-1}}\Phi_{\ell_2}+\psi_{\ell_1^{-1}\otimes\ell}\Phi_{\ell_2^{-1}}\\
&+\psi^G_{\ell_2\otimes\ell^{-1}}\Phi_{\ell_1}+\psi_{\ell_2^{-1}\otimes\ell}\Phi_{\ell_1^{-1}}+\frac{R_{\ell}}{N},
\end{align*}
with $|R_\ell|\le 10$ uniformly on $\Delta_{\Gbb}(T).$
\end{prop}

\section*{Acknowledgments} 

A.D. acknowledges partial support from ``The Dr Perry James (Jim) Browne Research Centre on Mathematics and its Applications.'' T.L acknowledges support from ANR AI chair BACCARAT (ANR-20-CHIA-0002).


\begin{thebibliography}{10}

\bibitem{AlbeverioI}
Sergio Albeverio, Raphael H{\diameter}egh-Krohn, and Helge Holden.
\newblock Stochastic {L}ie group-valued measures and their relations to
  stochastic curve integrals, gauge fields and {M}arkov cosurfaces.
\newblock In {\em Stochastic processes---mathematics and physics ({B}ielefeld,
  1984)}, volume 1158 of {\em Lecture Notes in Math.}, pages 1--24. Springer,
  Berlin, 1986.

\bibitem{AlbeverioII}
Sergio Albeverio, Raphael H{\diameter}egh-Krohn, and Helge Holden.
\newblock Stochastic multiplicative measures, generalized {M}arkov semigroups,
  and group-valued stochastic processes and fields.
\newblock {\em J. Funct. Anal.}, 78(1):154--184, 1988.

\bibitem{AS}
Michael Anshelevich and Ambar~N. Sengupta.
\newblock Quantum free {Y}ang--{M}ills on the plane.
\newblock {\em J. Geom. Phys.}, 62(2):330--343, 2012.

\bibitem{AB}
M.~F. Atiyah and R.~Bott.
\newblock The {Y}ang-{M}ills equations over {R}iemann surfaces.
\newblock {\em Philos. Trans. Roy. Soc. London Ser. A}, 308(1505):523--615,
  1983.

\bibitem{BenaychLev}
Florent Benaych-Georges and Thierry L\'{e}vy.
\newblock A continuous semigroup of notions of independence between the
  classical and the free one.
\newblock {\em Ann. Probab.}, 39(3):904--938, 2011.

\bibitem{Bia}
Philippe Biane.
\newblock Free {B}rownian motion, free stochastic calculus and random matrices.
\newblock In {\em Free probability theory ({W}aterloo, {ON}, 1995)}, volume~12
  of {\em Fields Inst. Commun.}, pages 1--19. Amer. Math. Soc., Providence, RI,
  1997.

\bibitem{BiS}
Joan~S. Birman and Caroline Series.
\newblock Dehn's algorithm revisited, with applications to simple curves on
  surfaces.
\newblock In {\em Combinatorial group theory and topology ({A}lta, {U}tah,
  1984)}, volume 111 of {\em Ann. of Math. Stud.}, pages 451--478. Princeton
  Univ. Press, Princeton, NJ, 1987.

\bibitem{BismutLab}
Jean-Michel Bismut and Fran\c{c}ois Labourie.
\newblock Symplectic geometry and the {V}erlinde formulas.
\newblock In {\em Surveys in differential geometry: differential geometry
  inspired by string theory}, volume~5 of {\em Surv. Differ. Geom.}, pages
  97--311. Int. Press, Boston, MA, 1999.

\bibitem{CDG}
Guillaume C\'{e}bron, Antoine Dahlqvist, and Franck Gabriel.
\newblock The generalized master fields.
\newblock {\em J. Geom. Phys.}, 119:34--53, 2017.

\bibitem{Chatt}
Sourav Chatterjee.
\newblock Rigorous solution of strongly coupled so(n) lattice gauge theory in
  the large n limit.
\newblock {\em Communications in Mathematical Physics}, 366(1):203--268, 2019.

\bibitem{Chev}
Ilya Chevyrev.
\newblock Yang-{M}ills measure on the two-dimensional torus as a random
  distribution.
\newblock {\em Comm. Math. Phys.}, 372(3):1027--1058, 2019.

\bibitem{ColGuionnetMS}
Beno\^{\i}t Collins, Alice Guionnet, and Edouard Maurel-Segala.
\newblock Asymptotics of unitary and orthogonal matrix integrals.
\newblock {\em Adv. Math.}, 222(1):172--215, 2009.

\bibitem{CollinsSnia}
Beno\^{\i}t Collins and Piotr \'{S}niady.
\newblock Integration with respect to the {H}aar measure on unitary, orthogonal
  and symplectic group.
\newblock {\em Comm. Math. Phys.}, 264(3):773--795, 2006.

\bibitem{Dah2}
Antoine Dahlqvist.
\newblock Free energies and fluctuations for the unitary {B}rownian motion.
\newblock {\em Comm. Math. Phys.}, 348(2):395--444, 2016.

\bibitem{Dah}
Antoine Dahlqvist.
\newblock Integration formulas for {B}rownian motion on classical compact {L}ie
  groups.
\newblock {\em Ann. Inst. Henri Poincar\'{e} Probab. Stat.}, 53(4):1971--1990,
  2017.

\bibitem{DL}
Antoine Dahlqvist and Thibaut Lemoine.
\newblock Large n limit of yang-mills partition function and wilson loops on
  compact surfaces, 2023.

\bibitem{DN}
Antoine Dahlqvist and James~R. Norris.
\newblock Yang--mills measure and the master field on the sphere.
\newblock {\em Communications in Mathematical Physics}, 377(2):1163--1226,
  2020.

\bibitem{DiaconisEvans}
Persi Diaconis and Steven~N. Evans.
\newblock Linear functionals of eigenvalues of random matrices.
\newblock {\em Trans. Amer. Math. Soc.}, 353(7):2615--2633, 2001.

\bibitem{Dri}
Bruce~K. Driver.
\newblock Y{M{${}_2$}}: continuum expectations, lattice convergence, and
  lassos.
\newblock {\em Comm. Math. Phys.}, 123(4):575--616, 1989.

\bibitem{DriverMM}
Bruce~K. Driver.
\newblock A functional integral approaches to the {M}akeenko-{M}igdal
  equations.
\newblock {\em Comm. Math. Phys.}, 370(1):49--116, 2019.

\bibitem{DGHK}
Bruce~K. Driver, Franck Gabriel, Brian~C. Hall, and Todd Kemp.
\newblock The {M}akeenko-{M}igdal equation for {Y}ang-{M}ills theory on compact
  surfaces.
\newblock {\em Comm. Math. Phys.}, 352(3):967--978, 2017.

\bibitem{DHK}
Bruce~K. Driver, Brian~C. Hall, and Todd Kemp.
\newblock Three proofs of the {M}akeenko-{M}igdal equation for {Y}ang-{M}ills
  theory on the plane.
\newblock {\em Comm. Math. Phys.}, 351(2):741--774, 2017.

\bibitem{Goldman}
William~M. Goldman.
\newblock The symplectic nature of fundamental groups of surfaces.
\newblock {\em Adv. in Math.}, 54(2):200--225, 1984.

\bibitem{GG}
Rajesh Gopakumar and David~J. Gross.
\newblock Mastering the master field.
\newblock {\em Nuclear Phys. B}, 451(1-2):379--415, 1995.

\bibitem{GrossMat}
David~J. Gross and Andrei Matytsin.
\newblock Some properties of large-n two-dimensional yang-mills theory.
\newblock {\em Nuclear Physics B}, 437(3):541--584, 1995.

\bibitem{GrossTaylor}
David~J. Gross and Washington Taylor, IV.
\newblock Two-dimensional {QCD} is a string theory.
\newblock {\em Nuclear Phys. B}, 400(1-3):181--208, 1993.

\bibitem{Gro}
Leonard Gross.
\newblock The {M}axwell equations for {Y}ang-{M}ills theory.
\newblock In {\em Mathematical quantum field theory and related topics
  ({M}ontreal, {PQ}, 1987)}, volume~9 of {\em CMS Conf. Proc.}, pages 193--203.
  Amer. Math. Soc., Providence, RI, 1988.

\bibitem{GKS}
Leonard Gross, Christopher King, and Ambar Sengupta.
\newblock Two-dimensional {Y}ang-{M}ills theory via stochastic differential
  equations.
\newblock {\em Ann. Physics}, 194(1):65--112, 1989.

\bibitem{Hal2}
Brian~C. Hall.
\newblock The large-{$N$} limit for two-dimensional {Y}ang--{M}ills theory.
\newblock {\em Comm. Math. Phys.}, 363(3):789--828, 2018.

\bibitem{HamblyLyons}
Ben Hambly and Terry Lyons.
\newblock Uniqueness for the signature of a path of bounded variation and the
  reduced path group.
\newblock {\em Ann. of Math. (2)}, 171(1):109--167, 2010.

\bibitem{KazPlan}
V.~A. Kazakov.
\newblock Wilson loop average for an arbitrary contour in two-dimensional
  {${\rm U}(N)$} gauge theory.
\newblock {\em Nuclear Phys. B}, 179(2):283--292, 1981.

\bibitem{KK}
V.~A. Kazakov and I.~K. Kostov.
\newblock Nonlinear strings in two-dimensional {${\rm U}(\infty )$} gauge
  theory.
\newblock {\em Nuclear Phys. B}, 176(1):199--215, 1980.

\bibitem{Ken}
Richard Kenyon.
\newblock Lectures on dimers.
\newblock In {\em Statistical mechanics}, volume~16 of {\em IAS/Park City Math.
  Ser.}, pages 191--230. Amer. Math. Soc., Providence, RI, 2009.

\bibitem{Lev3}
Thierry L{\'{e}}vy.
\newblock Yang-{M}ills measure on compact surfaces.
\newblock {\em Mem. Amer. Math. Soc.}, 166(790):xiv+122, 2003.

\bibitem{Lev8}
Thierry L\'{e}vy.
\newblock Discrete and continuous {Y}ang-{M}ills measure for non-trivial
  bundles over compact surfaces.
\newblock {\em Probab. Theory Related Fields}, 136(2):171--202, 2006.

\bibitem{Lev5}
Thierry L{\'{e}}vy.
\newblock Schur-{W}eyl duality and the heat kernel measure on the unitary
  group.
\newblock {\em Adv. Math.}, 218(2):537--575, 2008.

\bibitem{Lev2}
Thierry L{\'{e}}vy.
\newblock Two-dimensional {M}arkovian holonomy fields.
\newblock {\em Ast\'{e}risque}, (329):172, 2010.

\bibitem{Lev}
Thierry L{\'{e}}vy.
\newblock The master field on the plane.
\newblock {\em Ast\'{e}risque}, (388):ix+201, 2017.

\bibitem{Lev7}
Thierry L{\'e}vy.
\newblock {\em Two-Dimensional Quantum Yang--Mills Theory and the
  Makeenko--Migdal Equations}, pages 275--325.
\newblock Springer International Publishing, Cham, 2020.

\bibitem{LiechtyWang}
Karl Liechty and Dong Wang.
\newblock Nonintersecting brownian motions on the unit circle.
\newblock {\em The Annals of Probability}, 44(2):1134--1211, 2016.

\bibitem{Liu}
Kefeng Liu.
\newblock Heat kernels, symplectic geometry, moduli spaces and finite groups.
\newblock In {\em Surveys in differential geometry: differential geometry
  inspired by string theory}, volume~5 of {\em Surv. Differ. Geom.}, pages
  527--542. Int. Press, Boston, MA, 1999.

\bibitem{MageeII}
Michael {Magee}.
\newblock {Random Unitary Representations of Surface Groups II: The large $n$
  limit}.
\newblock {\em arXiv e-prints}, page arXiv:2101.03224, January 2021.

\bibitem{MageeI}
Michael Magee.
\newblock Random unitary representations of surface groups {I}: asymptotic
  expansions.
\newblock {\em Comm. Math. Phys.}, 391(1):119--171, 2022.

\bibitem{MM}
Yuri Makeenko and Alexander~A. Migdal.
\newblock Exact equation for the loop average in multicolor {QCD}.
\newblock {\em Physics Letters {B}}, 88:135--137, 1979.

\bibitem{MerPot}
Pierre Mergny and Marc Potters.
\newblock {Rank one HCIZ at high temperature: interpolating between classical
  and free convolutions}.
\newblock {\em SciPost Phys.}, 12:022, 2022.

\bibitem{Mig}
Alexander~A. Migdal.
\newblock Recursion equations in gauge field theories.
\newblock {\em Sov. Phys. JETP}, pages 413--418, 1975.

\bibitem{MingoSpeicher}
James~A. Mingo and Roland Speicher.
\newblock {\em Free probability and random matrices}, volume~35 of {\em Fields
  Institute Monographs}.
\newblock Springer, New York; Fields Institute for Research in Mathematical
  Sciences, Toronto, ON, 2017.

\bibitem{MR2044286}
Wojciech Mlotkowski.
\newblock {$\Lambda$}-free probability.
\newblock {\em Infin. Dimens. Anal. Quantum Probab. Relat. Top.}, 7(1):27--41,
  2004.

\bibitem{Moser}
J\"{u}rgen Moser.
\newblock On the volume elements on a manifold.
\newblock {\em Trans. Amer. Math. Soc.}, 120:286--294, 1965.

\bibitem{PPSY}
Minjae Park, Joshua Pfeffer, Scott Sheffield, and Pu~Yu.
\newblock Wilson loop expectations as sums over surfaces on the plane, 2023.

\bibitem{Sen0}
Ambar Sengupta.
\newblock Gauge theory on compact surfaces.
\newblock {\em Mem. Amer. Math. Soc.}, 126(600):viii+85, 1997.

\bibitem{Sen3}
Ambar Sengupta.
\newblock Yang-{M}ills on surfaces with boundary: quantum theory and symplectic
  limit.
\newblock {\em Comm. Math. Phys.}, 183(3):661--705, 1997.

\bibitem{Sen6}
Ambar~N. Sengupta.
\newblock The volume measure for flat connections as limit of the yang--mills
  measure.
\newblock {\em Journal of Geometry and Physics}, 47(4):398--426, 2003.

\bibitem{Sen2}
Ambar~N. Sengupta.
\newblock Traces in two-dimensional {QCD}: the large-{$N$} limit.
\newblock In {\em Traces in number theory, geometry and quantum fields},
  Aspects Math., E38, pages 193--212. Friedr. Vieweg, Wiesbaden, 2008.

\bibitem{Sin}
Isadore~M. Singer.
\newblock On the master field in two dimensions.
\newblock In {\em Functional analysis on the eve of the 21st century, {V}ol. 1
  ({N}ew {B}runswick, {NJ}, 1993)}, volume 131 of {\em Progr. Math.}, pages
  263--281. Birkh\"{a}user Boston, Boston, MA, 1995.

\bibitem{MR3552222}
Roland Speicher and Janusz Wysocza\'{n}ski.
\newblock Mixtures of classical and free independence.
\newblock {\em Arch. Math. (Basel)}, 107(4):445--453, 2016.

\bibitem{Sti}
John Stillwell.
\newblock {\em Classical topology and combinatorial group theory}, volume~72 of
  {\em Graduate Texts in Mathematics}.
\newblock Springer-Verlag, New York, second edition, 1993.

\bibitem{Hoo}
Gerard {'}{t}~Hooft.
\newblock A planar diagram theory for strong interactions.
\newblock {\em Nuclear Physics B}, 72(3):461 -- 473, 1974.

\bibitem{VoiFreeRM}
Dan Voiculescu.
\newblock Limit laws for random matrices and free products.
\newblock {\em Invent. Math.}, 104(1):201--220, 1991.

\bibitem{Voi}
Dan Voiculescu.
\newblock The analogues of entropy and of {F}isher's information measure in
  free probability theory. {VI}. {L}iberation and mutual free information.
\newblock {\em Adv. Math.}, 146(2):101--166, 1999.

\bibitem{VoiLect}
Dan Voiculescu.
\newblock Lectures on free probability theory.
\newblock In {\em Lectures on probability theory and statistics
  ({S}aint-{F}lour, 1998)}, volume 1738 of {\em Lecture Notes in Math.}, pages
  279--349. Springer, Berlin, 2000.

\bibitem{Wil}
Kenneth~G. Wilson.
\newblock Confinement of quarks.
\newblock {\em Phys. Rev.}, 10, 1974.

\bibitem{Wit2}
Edward Witten.
\newblock {The $1/N$ expansion in atomic and particle physics}.
\newblock {\em NATO Sci. Ser. B}, 59:403--419, 1980.

\bibitem{Wit}
Edward Witten.
\newblock On quantum gauge theories in two dimensions.
\newblock {\em Comm. Math. Phys.}, 141(1):153--209, 1991.

\bibitem{Xu}
Feng Xu.
\newblock A random matrix model from two-dimensional {Y}ang--{M}ills theory.
\newblock {\em Comm. Math. Phys.}, 190(2):287--307, 1997.

\end{thebibliography}
\end{document}